\def\UseTikzCache{}\pdfoutput=1
\documentclass[12pt,dvipsnames,oneside]{amsart}

\usepackage[english]{babel}
\usepackage[utf8]{inputenc}
\usepackage[normalem]{ulem}
\usepackage{amssymb}
\usepackage{a4wide}
\usepackage{microtype}
\usepackage{lmodern}
\usepackage{mathtools}
\mathtoolsset{mathic=true}
\usepackage{color, graphicx}
\usepackage[inline]{enumitem}
\usepackage{bbold}
\usepackage{quiver}
\usepackage{mathabx}
\usepackage{tikz-cd}
\tikzcdset{arrow style=math font}

\ifdefined\tikzexternalrealjob
  \def\UseTikzCache{}\def\MakeTikzCache{}
  \PassOptionsToPackage{bookmarks=false}{hyperref}
\fi

\ifdefined\UseTikzCache
  \usetikzlibrary{external}
  \tikzset{external/mode=graphics if exists}
  \ifdefined\MakeTikzCache
    \tikzset{external/mode=convert with system call}
  \fi
  \ExplSyntaxOn\makeatletter
  \let\orig@tikzcd\tikzcd
  \let\orig@endtikzcd\endtikzcd
  \RenewDocumentEnvironment{tikzcd}{O{}+b}{
    \begin{tikzpicture}[commutative~diagrams/.cd, every~diagram, #1]
      \tl_set:Nn \l_tmpa_tl { #2 }
      \tl_replace_all:Nnn \l_tmpa_tl { \& } { \pgfmatrixnextcell }
      \tl_replace_all:Nnn \l_tmpa_tl { & } { \pgfmatrixnextcell }
      \let\saved@tikzpicture\tikzpicture
      \let\saved@endtikzpicture\endtikzpicture
      \def\tikzpicture[##1]{}
      \let\endtikzpicture\relax
      \orig@tikzcd
        \l_tmpa_tl
      \orig@endtikzcd
      \let\tikzpicture\saved@tikzpicture
      \let\endtikzpicture\saved@endtikzpicture
    \end{tikzpicture}
  }{}
  \ExplSyntaxOff\makeatother
\fi

\usepackage{stmaryrd}
\usepackage{pdflscape}
\usepackage{aligned-overset} 

\usepackage[backend=bibtex, backref=true, maxnames=6, maxalphanames=6, style=alphabetic, citestyle=alphabetic]{biblatex}
\DeclareFieldFormat*{title}{#1}

\DeclareDelimFormat{finalnamedelim}{
  \ifnumgreater{\value{liststop}}{2}{,}{}%
  \addspace\bibstring{and}\space}
\bibliography{bib}

\usepackage{url}
\definecolor{red}{RGB}{175, 49, 39}
\usepackage[colorlinks=true,linkcolor=red,citecolor=red,urlcolor=red]{hyperref}
\usepackage{amsthm}
\usepackage[capitalise, noabbrev]{cleveref}

\makeatletter
\newtheorem*{rep@theorem}{\rep@title}
\newenvironment{reptheorem}[1]{\def\rep@title{\cref{#1}}\begin{rep@theorem}}{\end{rep@theorem}}
\makeatother

\DeclareFontFamily{U}{dmjhira}{}
\DeclareFontShape{U}{dmjhira}{m}{n}{ <-> dmjhira }{}

\usepackage{thmtools}
\theoremstyle{plain}

\newtheorem{proposition}{Proposition}[section]
\newtheorem{theorem}[proposition]{Theorem}
\newtheorem{lemma}[proposition]{Lemma}
\newtheorem{corollary}[proposition]{Corollary}
\declaretheorem[name=Theorem,qed={\tiny$\blacksquare$},numbered=no]{theorem*}

\theoremstyle{remark}
\declaretheorem[ name=Remark, style=definition, qed=$\diamondsuit$, sharenumber=proposition ]{remark}

\theoremstyle{definition}
\declaretheorem[ name=Definition, style=definition, qed=$\triangledown$, sharenumber=proposition ]{definition}
\declaretheorem[ name=Example, style=definition, qed=$\triangle$, sharenumber=proposition ]{example}

\usepackage{rotating}
\usepackage{floatpag}
\newenvironment{amssidewaysfigure}{%
  \begin{landscape}
    \thispagestyle{empty}
    \begin{figure}[htbp]
}{%
    \end{figure}
  \end{landscape}
}

\newcommand{\mathscale}[2]{\scalebox{#1}{\mbox{\ensuremath{\displaystyle #2}}}}

\DeclareMathOperator{\id}{\mathsf{id}}
\DeclareMathOperator{\Id}{\mathsf{Id}}

\newcommand*{\from}{\colon}
\newcommand*{\blank}{{-}}
\newcommand*{\bblank}{{=}}
\newcommand*{\eqdef}{\coloneqq}

\newcommand{\Tor}{\ensuremath{\mathrm{Tor}}}

\newcommand{\Set}{\mathsf{Set}}
\newcommand{\Vect}{\mathsf{Vect}}
\newcommand{\preord}{\mathsf{Preord}}
\newcommand{\plgraph}{\mathsf{Graph}}
\newcommand{\rgraph}{\mathsf{RGraph}}
\newcommand{\Cat}{\mathsf{Cat}}
\newcommand{\thinCat}{\mathsf{thinCat}}

\makeatletter
\newcommand{\cat}{%
  \@ifnextchar[
    {\@cat@square}%
    {\@cat@brace}%
}

\def\@cat@square[#1]{\ensuremath{\mathcal{#1}}}

\def\@cat@brace{%
  \@ifnextchar\bgroup
    {\@cat@with@brace}
    {\ensuremath{\mathcal{C}}}
}

\def\@cat@with@brace#1{\ensuremath{\mathcal{#1}}}
\makeatother

\newcommand{\mb}[1]{\boldsymbol{#1}}

\newcommand{\tu}{\mathbb{1}}
\newcommand{\act}{\mathfrak{a}}

\DeclareMathOperator{\End}{\mathsf{End}}
\DeclareMathOperator{\ev}{\mathsf{ev}}
\DeclareMathOperator{\coev}{\mathsf{coev}}
\newcommand{\op}{^{\textup{op}}}
\NewDocumentCommand{\hmodM}{O{H}O{}O{}O{}}{\prescript{#4}{#1}{\mathcal{M}}_{#2}^{#3}}

\newcommand*{\Hom}{\mathsf{Hom}}
\newcommand*{\term}{{\star}}
\newcommand{\dobj}{{d}}
\newcommand{\free}{F}
\newcommand{\forget}{U}
\newcommand{\G}{G}
\newcommand{\U}{U}

\makeatletter
\newcommand*{\coslice}{\mathpalette\@coslice{}}
\newcommand*{\@coslice}[2]{\raisebox{1pt}{$#1\downarrow$}}
\newcommand*{\points}[1][\cat{C}]{\mathpalette\@term{}\coslice#1}
\newcommand*{\@term}[2]{\raisebox{1pt}{$#1\term$}}
\newcommand*{\htens}{\mathpalette\@htens{}}
\newcommand*{\@htens}[2]{\mathbin{\raisebox{1pt}{\scalebox{0.7}{$#1\rightthreetimes$}}}}
\makeatother

\newcommand{\cC}{\mathcal{C}}
\newcommand{\cD}{\mathcal{D}}

\newcommand{\cM}{\mathcal{M}}

\newcommand{\cP}{\mathcal{P}}

\makeatletter
\patchcmd{\@setaddresses}{\indent}{\noindent}{}{}
\patchcmd{\@setaddresses}{\indent}{\noindent}{}{}
\patchcmd{\@setaddresses}{\indent}{\noindent}{}{}
\patchcmd{\@setaddresses}{\indent}{\noindent}{}{}
\makeatother

\usepackage{xspace}
\newcommand*{\adj}[4]{\ensuremath{#1 \colon #3 \rightleftarrows #4 \cocolon #2}\xspace}
\newcommand*{\cocolon}{%
  \nobreak%
  \mskip6mu plus1mu
  \mathpunct{}%
  \nonscript%
  \mkern-\thinmuskip%
  {:}%
  \mskip2mu
  \relax
}
\newcommand*{\defeq}{\coloneqq}

\usepackage{multicol} 

\title{Gabi-monads}

\author{Sebastian Halbig}
\address{S.H., Philipps-Universität Marburg, Arbeitsgruppe Algebraische Lie-Theorie, Hans-Meer\-wein-Straße 6, 35043 Marburg}
\email{sebastian.halbig@uni-marburg.de}

\author{Paolo Saracco}
\address{P.S., Dipartimento di Matematica `G.\ Peano', Universit\`a degli Studi di Torino, via Carlo Albero 10, 10123 Torino, Italy.}
\email{p.saracco@unito.it}
\urladdr{\url{https://sites.google.com/view/paolo-saracco}}

\author{Tony Zorman}
\address{T.Z., Universität Hamburg, Fachbereich Mathematik, Bereich Algebra und Zahlentheorie, Bundesstraße 55, D-20146 Hamburg, Germany}
\email{tony.zorman@uni-hamburg.de}
\urladdr{\url{https://tony-zorman.com/research}}

\date{\today}
\subjclass[2020]{Primary: 18C15, 18D15, 18C20; Secondary: 18M50, 16T05}
\thanks{T.Z.~was partially supported by the DFG grant KR 5036/2--1 and acknowledges partial support from the Deutsche Forschungsgemeinschaft -- SFB 1624 -- ``Higher structures, moduli spaces and integrability'' -- 506632645.
This paper was written while P.S.~was a member of the ``National Group for Algebraic and Geometric Structures and their Applications'' (GNSAGA-INdAM). P.S.~was supported by the project ``HADRA - Hopf Algebroids: Duality, Representations and Applications'' funded by the Italian Ministero dell'Universit\`a e della Ricerca, grant number PGR21S8QB2, and was partially supported by the project PID2024-157173NB-I00 funded by MCIN/AEI/10.13039/501100011033 and by FEDER, UE}

\begin{document}

\allowdisplaybreaks

\begin{abstract}
  We study gabi-monads on skew-closed categories, extending the gabi-algebras of Berger, the second author, and Vercruysse beyond the linear case.
  Our main reconstruction theorem identifies gabi-monad structures on a monad with skew-closed structures on its Eilenberg--Moore category for which the canonical forgetful functor is strict closed.
  We compare this notion with closed monads in the sense of Kock, showing that in representation-theoretic cases these notions are quite different.
  On closed monoidal categories, every left Hopf monad is a normal gabi-monad, but the converse fails in general.
  We characterise when a gabi-monad is Hopf by the invertibility of the corresponding parametric mates, which recovers the ring-theoretic result that normal gabi-algebras over a commutative base ring are Hopf algebras.
  The theory of gabi-monads admits several natural examples, such as torsion-free modules,
  reflexive digraphs, and simplicial complexes, that we will explore in detail; we also study pointed sets as a quasi-example.
\end{abstract}

\maketitle

\tableofcontents

\section{Introduction}

The relationship between monads and closed monoidal categories---categories admitting a tensor product and an internal hom---has been studied extensively through monoidal reconstruction, both in the language of bimonads developed by Moerdijk and McCrudden~\cite{Moerdijk2002,McCrudden2002} and via the theory of Hopf monads, introduced by Bruguières--Virelizier and developed further by Bruguières--Lack--Virelizier~\cite{Bruguieres2007,BruguieresLackVirelizier}.
Given a monad \(T\) on a closed monoidal category \(\cat{C}\), a bimonad structure on \(T\) ensures that the Eilenberg--Moore category \(\cat{C}^{T}\) inherits a monoidal structure such that the canonical forgetful functor \(\forget^{T}\from \cat{C}^{T} \to \cat{C}\) is strict monoidal.
Moreover, \(T\) is left Hopf monad if and only if \(\cat{C}^{T}\) is additionally closed and \(U^{T}\) is strict closed.
As such, the closed structure is reconstructed \emph{after} the tensor product, by lifting it as an adjoint.
Instead, the point of view of this paper is to start with the internal hom itself.

Thus, we work with the more general notion of \emph{\textup(skew\textup) closed category} studied extensively by for example Eilenberg--Kelly~\cite{Eilenberg-Kelly}, Street~\cite{Street-skew_closed}, and Uustalu--Veltri--Zeilberger~\cite{UVZ_Eilenberg-Kelly_reloaded}.
Briefly, this is a category \(\cat{C}\) equipped with an internal hom functor \([\blank,\bblank]\), a distinguished unit object \(\tu\), and natural transformations encoding identities and composition, subject to appropriate coherence conditions.
Every closed monoidal category, such as the category of vector spaces, is closed in this sense.

\medskip
In the monoidal case, a strong monoidal forgetful functor automatically yields a colax monoidal monad, which may be used for monoidal reconstruction immediately, see~\cite{Kelly-doctrines,Moerdijk2002,McCrudden2002}.
Things are less well-behaved in the closed case: a strong closed forgetful functor does not in general induce a closed monad in the sense of Kock~\cite{Kock-Monads,Kock-ClosedEM}.
Instead of a closed structure on a monad, what is needed to lift the internal hom to the Eilenberg--Moore category is that the monad admits the structure of a \emph{gabi-monad}, see~\cref{def:gabi_monad},
which generalises the gabi-algebras of~\cite{Berger-Vercruysse-Saracco}; see Theorem~3.4 of \emph{loc cit} and \cref{thm:gabi-bijection}.
Concretely, a gabi-monad is a monad \(T\) equipped with a \(T\)-algebra structure on the closed unit, as well as a natural transformation \(s\from T[T\blank,\bblank]\to[\blank,T\bblank]\),
subject to certain compatibility relations, which turn \([M,N]\) into a \(T\)-algebra whenever \(M\) and \(N\) are.
In the easiest case, on may choose \(T=A\otimes \blank \) for a \emph{gabi-algebra} \(A\):
an algebra, endowed with an augmentation and a map \(s \from A\to A \otimes A^{\op}\) that axiomatises the familiar plus-minus translation map \(a\mapsto a_+\otimes a_-\),
and whose associated morphism \(a\otimes b\mapsto a_+\otimes a_-b\) plays the role of the Galois map in Hopf theory.

Gabi-monads fit into an analogue of the doctrinal adjunction-type result used for monoidal reconstruction discussed above:

\begin{reptheorem}{prop:gabi-moand-VS-strong-closed}
  Consider an ordinary adjunction \(\adj{\free}{\forget}{\cat{C}}{\cat{D}}\) between two skew-closed categories \(\cat{C}\) and \(\cat{D}\).
  There is a bijective correspondence between strong closed structures on the right adjoint \(\forget\),
  and pairs consisting of a gabi-monad structure on \(\forget\free\) and a strong closed structure on the comparison functor.
\end{reptheorem}

A closed monad does not necessarily have a strong closed right forgetful functor, and will in general only carry a gabi-monad structure under additional hypotheses.

\begin{reptheorem}{prop:normal-closed-is-gabi,cor:nice-closed-gabi}
  Let \(\cat{C}\) be a skew-closed category, with \((T, \xi_0,\xi)\) a closed monad on \(\cat{C}\).
  If \(\xi_0\from \tu \to T\tu\) is an isomorphism, then \(T\) is a gabi-monad.

  If \(\cat{C}\) is an abelian closed monoidal category and \(T\) is right exact
  then \(T\) admits a gabi-monad structure such that \(\xi_0\) is a morphism of \(T\)-algebras
  if and only if
  the canonical unit comparison of \(\forget^T\) is invertible.
  In case \(\tu \in \cat{C}\) is a generator and \(\cat{C}\) is cocomplete, this implies \(T \cong \Id\).
\end{reptheorem}

On ordinary closed monoidal categories, the difference between gabi- and Hopf monads is measured by whether the lifted internal hom is still part of a lifted tensor--hom adjunction.
This is governed by the invertibility of certain ``parametric mates''.

\begin{reptheorem}{thm:lift-to-hopf-monads}
  Let \(\cat[C]\) be a closed monoidal category and \((T, s , \mathfrak{a}_{\tu})\) a gabi-monad on \(\cat[C]\).
  Then \(T\) is a left Hopf monad if and only if for every object \(\boldsymbol{M} \in \cat[C]^T\),
  the mate of
  \[
    {\left(\gamma^{\boldsymbol{M}}_X \eqdef s_{M,X}\circ T[\act_{M}, X] \from T[M,X] \to [M,TX]\right)}_{X \in \cat{C}}
  \]
  is invertible.
\end{reptheorem}

The ring-theoretic setting of~\cite{Berger-Vercruysse-Saracco} is recovered from the monadic theory:
for algebras over a commutative ring, Hopf algebra structures, normal gabi-algebra structures,\footnote{%
  \,That is, gabi-algebra structures where we don't just lift the weaker skew-closed, but instead the entire, closed structure.%
}
and gabi-algebra structures such that \(a\otimes b\mapsto a_+\otimes a_-b\) is invertible are all equivalent.
In particular, this more conceptual proof bypasses much of the element-wise verification involved in the original result, see \cref{cor:gabi-in-the-com-case}.

The monadic situation presents itself differently: normal gabi-monads---gabi-monads whose Eilenberg--Moore categories are normal closed as opposed to skew-closed---need not be Hopf.
A large supply of examples is discussed in \cref{sec:gabi-not-hopf} and comes from idempotent monads associated with reflective exponential ideals: such reflections lift internal homs automatically, and in fact one may even define a tensor product on the reflective exponential ideal.
However, this tensor product need not be compatible with the forgetful functor.
For example, for an integral domain \(R\), the reflection \(M \mapsto M/\Tor(M)\) from \(R\)-modules to torsion-free \(R\)-modules defines a normal idempotent gabi-monad. It is Hopf if and only if \(R\) is a Pr\"ufer domain (i.e., the tensor product of torsion-free modules is torsion-free~\cite[Theorem~4.2]{chase61:direc});
see \cref{thm:normal-gaby-monad-torsion}.
Examples of non-Pr\"ufer domains abound: the polynomial ring \(R[x]\) is Pr\"ufer if and only if \(R\) is a field.
Other examples of non-Hopf gabi-monads come from the reflection of reflexive digraphs onto preorders and from a very specific closed structure on finite simplicial complexes.

Algebras for the maybe monad, on the contrary, provide a natural example of a skew-closed category where the forgetful functor lifts the internal homs but not the unit object.
In particular, the maybe monad is not a gabi-monad.
The study of these weaker gabi-monads goes beyond the scope of this paper and will be the subject of a forthcoming work by the authors.

\medskip
The paper is organised as follows.
In \cref{sec:closec-cats}, we review basic results of (skew-)closed categories.
\Cref{sec:gabi} introduces gabi-monads, establishes the reconstruction theorem, and relates gabi- and closed monads.
\Cref{sec:normal-gabi} focuses on normal gabi-monads and characterises Hopf monads by the invertibility of certain parametric adjoints.
\Cref{sec:dyadic-gabi} studies gabi-monads in the special case of non-symmetric \(*\)-autonomous categories.
Finally, in \cref{sec:gabi-not-hopf}, we study a variety of examples and non-examples, including idempotent gabi-monads, torsion-free modules, preorders, simplicial complexes, and the maybe monad.

\section{Closed Categories}\label{sec:closec-cats}
Many categories, such as modules over a commutative ring \(R\), have a natural notion of internal morphism objects associated to every pair of objects \(M\) and \(N\). In our example, this is the \(R\)-module \(\Hom_{R}(M,N)\).
These internal morphisms come with additional structure such as an abstract notion of identities and composition.
The most prominent setting is that of a closed monoidal category, where the tensor-hom adjunction governs the behaviour of the internal morphism.
The category of \(R\)-modules is a classic example of this type.
In the absence of a tensor product, the idea of categories with coherent internal hom-objects is axiomatised by the notion of closed categories.
These form the basis for our investigation.

\subsection{Left skew-closed and left skew-monoidal categories}\label{ssec:skew-closed}
First we turn to a suitably lax version of closed categories,
following~\cite{Street-skew_closed, UVZ_Eilenberg-Kelly_reloaded},
which has the advantage of only involving conditions internal to \(\cat\).

\begin{definition}\label{def:skewclosed}
  A (\emph{left}) \emph{skew-closed category} is a tuple \((\cat, [\blank,\bblank], \tu, \Gamma, i, j)\), where
  \begin{enumerate}[label={(C\arabic*)},leftmargin=1.2cm]
    \item
    \(\tu \in \cat\) is an object;

    \item
    \([\blank,\bblank] \colon \cat\op \times \cat \to \cat\) is a functor;

    \item
    \(i \colon [\tu, \blank] \to \blank\) is a natural transformation;

    \item
    \(j \colon \tu \xrightarrow{\cdot\cdot} [\blank,\blank]\) is a dinatural transformation (see, e.g.,~\cite[\S IX.4]{Maclane}); and

    \item
    \(\Gamma\) is a family of morphisms \(\Gamma^X_{Y, Z} \colon [Y, Z] \to [[X, Y], [X, Z]]\) for  \(X,Y,Z\) objects in \(\cat\), which is natural in the lower indices and dinatural in the upper index.
  \end{enumerate}
  This data is subject to the following axioms: \par
  \begin{subequations}
    \noindent\begin{minipage}{.5\textwidth}
      \begin{equation}
        \begin{gathered}
          \begin{tikzcd}[ampersand replacement=\&,row sep=22.5pt,column sep=19pt]
            {\tu} \arrow[r, "{j_\tu}"]\arrow[rd, equal]\& {{[\tu, \tu]}} \arrow[d, "{i_\tu}"]\\ \&
            {\tu}
          \end{tikzcd}
        \end{gathered}
      \end{equation}
    \end{minipage}
    \begin{minipage}{.5\textwidth}
      \begin{equation}
        \begin{gathered}
          \begin{tikzcd}[ampersand replacement=\&,row sep=22.5pt,column sep=19pt]
            {{[X, Y]}} \arrow[r, "{\Gamma^X_{X, Y}}"]\arrow[d, equal]\& {{[[X, X], [X, Y]]}} \arrow[d, "{[j_X, [X, Y]]}"]\\
            {{[X, Y]}}
            \& {{[\tu, [X, Y]]}} \arrow[l, "{i_{[X, Y]}}"]
          \end{tikzcd}
        \end{gathered}
      \end{equation}
    \end{minipage}
    \par
    \noindent\begin{minipage}{.5\textwidth}
      \begin{equation}
        \begin{gathered}
          \begin{tikzcd}[ampersand replacement=\&,row sep=22.5pt,column sep=19pt]
            {\tu} \arrow[r, "{j_Y}"]\arrow[rd, "{j_{[X, Y]}}"']\& {{[Y, Y]}} \arrow[d, "{\Gamma^X_{Y, Y}}"]\\ \&
            {{[[X, Y], [X, Y]]}}
          \end{tikzcd}
        \end{gathered}
      \end{equation}
    \end{minipage}
    \begin{minipage}{.5\textwidth}
      \begin{equation}
        \begin{gathered}
          \begin{tikzcd}[ampersand replacement=\&,row sep=22.5pt,column sep=19pt]
            {{[X, Y]}} \arrow[r, "{\Gamma^\tu_{X, Y}}"]\arrow[rd, "{[i_X, Y]}"']\& {{[[\tu, X], [\tu, Y]]}} \arrow[d, "{[[\tu, X], i_Y]}"]\\ \&
            {{[[\tu, X], Y]}}
          \end{tikzcd}
        \end{gathered}
      \end{equation}
    \end{minipage}
    \par
    \begin{equation}
      \begin{gathered}
        \begin{tikzcd}[ampersand replacement=\&,row sep=24.2pt,column sep=21.5pt]
          {{[W, X]}} \arrow[r, "{\Gamma^U_{W, X}}"]\arrow[dd, "{\Gamma^V_{W, X}}"']\& {{[[U, W], [U, X]]}} \arrow[d, "{\Gamma^{[U, V]}_{[U, W], [U, X]}}"]\\ \&
          {{[[[U, V], [U, W]], [[U, V], [U, X]]]}}
          \arrow[d, "{[\Gamma^U_{V, W}, [[U, V], [U, X]]]}"]\\
          {{[[V, W], [V, X]]}}
          \arrow[r, "{[[V, W], \Gamma^U_{V, X}]}"']\& {{[[V, W], [[U, V], [U, X]]]}}
        \end{tikzcd}
      \end{gathered}
    \end{equation}
  \end{subequations}
  A skew-closed category is said to be
  \begin{enumerate}[label=(N\arabic*),ref=(N\arabic*),leftmargin=1.2cm]
    \item\label{item:N1}
    \emph{left normal} if
    \[
      \widehat{\jmath}_{X,Y}:\cat(X, Y) \to \cat(\tu, [X, Y]), \quad
      f \mapsto [f, Y] \circ j_Y
    \]
    is a natural bijection;

    \item\label{item:N2}
    \emph{right normal} if \(i\) is a natural isomorphism;

    \item\label{item:N3}
    \emph{associative normal} if the  canonical morphism
    \[
      \hat\Gamma_{U,X,Y,Z}\from \int^{V \in \cat} \cat(U, [V, Z]) \times \cat(X, [Y, V]) \to \cat(U, [X, [Y, Z]])
    \]
    is a bijection, where \(\hat\Gamma_{U,X,Y,Z}\) is induced by
    \[
      (f,g) \quad\mapsto\quad
      \left(U \xrightarrow{\ f\ } [V,Z] \xrightarrow{\Gamma^Y_{V,Z}} [[Y,V],[Y,Z]] \xrightarrow{[g,[Y,Z]]} [X,[Y,Z]]\right),
    \]
    for all \(f\in \cat(U, [V, Z])\), \(g\in \cat(X, [Y, V])\).
  \end{enumerate}
  A skew-closed category satisfying left and right normality conditions will be called \emph{unital normal}.
  A skew-closed category satisfying all three normality conditions will be called \emph{normal-closed}.
\end{definition}

A category \(\cat{C}\) is called \emph{\textup(left\textup) skew-monoidal} if it is equipped with a functor \(\otimes\from \cat{C} \times \cat{C} \to \cat{C}\), an object \(\tu \in \cat{C}\),
and not-necessarily-invertible coherence morphisms \(\lambda\from \tu \otimes \blank \to \Id_{ \cat{C}}\), \(\rho\from \Id_{ \cat{C}} \to \blank\otimes \tu \), and \(\alpha\from (\blank \otimes \bblank) \otimes {\equiv} \to \blank \otimes (\bblank \otimes {\equiv})\),
satisfying axioms similar to those of a monoidal category; see~\cite[Definition~2.1]{szlachanyi2012:SkewMonoidalBialgebroid} for the right skew-monoidal case.
The following theorem relates skew-closed and skew-monoidal categories.

\begin{theorem}[{\cite[Theorems 2.10 and 3.8]{UVZ_Eilenberg-Kelly_reloaded}}]\label{thm:bijection_skew_closed_monoidal_structures}
  Let \(\cat\) be a category with a distinguished object \(\tu\), functors \(\blank \otimes \bblank \colon \cat \times \cat \to \cat\) and \([\blank,\bblank] \colon \cat\op \times \cat \to \cat\), and dinatural transformations
  \[
    \coev_Y^X \from Y \to [X, Y\otimes X], \qquad
    \ev_Y^X \from [X,Y]\otimes X \to Y,
  \]
  establishing a family adjunctions \(\{\blank \otimes X   \dashv [X, \blank]\}_{X \in \cat[C]}\) parameterised by the objects \(X \in \cat{C}\).

  Then left skew-monoidal structures \((\alpha, \lambda, \rho)\) on \((\cat, \otimes, \tu)\) are in bijection
  with left skew-closed structures \((\Gamma, i, j)\) on \((\cat, [\blank,\bblank], \tu)\).
  Moreover, the left skew-monoidal structure is left/right/associative normal if and only if the
  left skew-closed structure is left/right/associative normal.
\end{theorem}

To make \cref{thm:bijection_skew_closed_monoidal_structures} explicit, suppose that \(\cat{C}\) is a skew-closed category.
We may define the associated skew-monoidal structure via
\begin{align}\label{eq:mon_closed}
  \begin{gathered}
    \begin{tikzcd}[ampersand replacement=\&,cramped]
      {\tu \otimes X} \&\& { [X,X] \otimes X } \\
      \& { X}
      \arrow["{j_{X} \otimes X}", from=1-1, to=1-3]
      \arrow["{\lambda_X}"', dashed, from=1-1, to=2-2]
      \arrow["{\ev_X^X}", from=1-3, to=2-2]
    \end{tikzcd}
    \qquad
    \begin{tikzcd}[ampersand replacement=\&,cramped]
      {X } \&\& {[\tu, X\otimes \tu]  } \\
      \& {X\otimes \tu}
      \arrow["{\coev_X^\tu}", from=1-1, to=1-3]
      \arrow["{\rho_X}"', dashed, from=1-1, to=2-2]
      \arrow["{i_{{X\otimes \tu}}}", from=1-3, to=2-2]
    \end{tikzcd}\\
    \begin{tikzcd}[ampersand replacement=\&,cramped]
      {(X\otimes Y)\otimes Z} \&\& {([Y\otimes Z, X\otimes (Y\otimes Z)] \otimes Y) \otimes Z } \\
      \&\& {([Y,[Z, X\otimes (Y\otimes Z)]] \otimes Y) \otimes Z      } \\
      {X\otimes (Y\otimes Z)} \&\& {[Z, X\otimes (Y\otimes Z)] \otimes Z}
      \arrow["{(\coev_{X}^{Y \otimes Z} \otimes Y) \otimes Z}", from=1-1, to=1-3]
      \arrow["{\alpha_{X,Y,Z}}"',dashed, from=1-1, to=3-1]
      \arrow["{(\beta_{Y,Z,X\otimes (Y\otimes Z)}\otimes Y)\otimes Z}", from=1-3, to=2-3]
      \arrow["{\ev^{Y}_{[Z,X\otimes (Y\otimes Z)]}\otimes Z}", from=2-3, to=3-3]
      \arrow["{\ev^{Z}_{X\otimes (Y\otimes Z)}}", from=3-3, to=3-1]
    \end{tikzcd}
  \end{gathered}
\end{align}
where we set for all \(X, Y, Z \in \cat[C]\):
\begin{equation}
  \beta_{X, Y, Z} \colon
  [X \otimes Y, Z] \xrightarrow{\Gamma^{Y}_{X \otimes Y, Z}}
  [[Y, X \otimes Y], [Y, Z]] \xrightarrow{ [\coev_X^Y, [Y, Z]] }
  [X, [Y, Z]].
\end{equation}

In the opposite direction, if \(\cat[C]\) is a skew-monoidal category, then we can define the skew-closed structure via
\begin{gather}
  \begin{tikzcd}[ampersand replacement=\&,row sep=22.5pt,column sep=19pt,baseline=(\tikzcdmatrixname-1-1.base)]
{\tu} \arrow[dr, dashed, "{j_{X}}"']\arrow[rr, "{\coev_\tu^X}"]\& \& {[X,\tu \otimes X]} \arrow[dl, "{[\tu, \lambda_{X}]}"]\\
    \& {[X,X]} \&
\end{tikzcd}
  \qquad
  \begin{tikzcd}[ampersand replacement=\&,row sep=22.5pt,column sep=19pt,baseline=(\tikzcdmatrixname-1-1.base)]
{[\tu,X]} \arrow[dr, dashed, "{i_X}"']\arrow[rr, "{\rho_{[\tu,X]}}"]\& \& {[\tu,X] \otimes \tu} \arrow[dl, "{\ev_X^\tu}"]\\
    \& {X} \&
\end{tikzcd} \\
  \begin{tikzcd}[ampersand replacement=\&,row sep=24.2pt,column sep=21.5pt,baseline=(\tikzcdmatrixname-1-1.base)]
{[Y,Z]} \arrow[ddd, dashed, "{\Gamma^X_{Y,Z}}"']\arrow[r, "{\coev^{[X,Y]}_{[Y,Z]}}"]\& {[[X,Y],[Y,Z] \otimes [X,Y]]} \arrow[d, "{[[X,Y],\coev^X_{[Y,Z] \otimes [X,Y]}]}"]\\
    \& {[[X,Y],[X,([Y,Z] \otimes [X,Y]) \otimes X]]} \arrow[d, "{[[X,Y],[X,\alpha_{[Y,Z],[X,Y],X}]]}"]\\
    \& {[[X,Y],[X,[Y,Z] \otimes ([X,Y] \otimes X)]]} \arrow[d, "{[[X,Y],[X,[Y,Z] \otimes \ev^X_Y]]}"]\\
    {[[X,Y],[X,Z]]} \& {[[X,Y],[X,[Y,Z] \otimes Y]]} \arrow[l, "{[[X,Y],[X,\ev^Y_Z]]}"]
\end{tikzcd}
\end{gather}

By the above theorem, the coherences of the tensor product and internal-hom of \(\cat{C}\) completely determine each other, allowing us to define closed monoidal categories equivalently as special cases of monoidal or closed categories.

\begin{definition}\label{def:skew-closed-monoidal-category}
  A left skew-closed category \(\cat{C}\) with internal hom \([\blank,\bblank] \colon \cat\op \times \cat \to \cat\) is called (\emph{left}) \emph{skew-closed monoidal} if there is a functor \(\blank \otimes \bblank \colon \cat \times \cat \to \cat\) and an adjunction \(\blank \otimes X   \dashv [X, \blank]\) for all \(X \in \cat[C]\).

  We refer to \( \cat{C}\) as (\emph{left}) \emph{closed monoidal} if \( \cat{C}\) is normal closed.
\end{definition}
In this paper, we are mainly concerned with left (skew-)closed and left (skew-)monoidal structures, whence we often omit the adjective \emph{left}, unless clarity requires specifying it.

\medskip

To familiarise ourselves with closed categories, let us close this subsection by collecting a number of key examples.

\begin{example}\label{ex:set_is_closed}
The category \(\Set\) of sets and functions is skew-closed.
The internal hom functor is given by \([A, B] = \Set(A, B)\).
The unit object is the unit object of the monoidal structure, i.e.\ a fixed
one-element set \(\term \coloneqq \{\ast\}\).
The natural transformation \(i_A \colon \Set(\{\ast\}, A) \to A\) is the isomorphism given by
\(i_A(f) \defeq f(\ast)\), and the dinatural transformation \(j_A \colon \{\ast\} \to \Set(A, A)\) picks
out the identity, i.e.\ \(j_A(\ast) = \id_A\).
Finally, the transformation \(\Gamma^A_{B, C}\) is given by post-composition, meaning
\begin{equation*}
\Gamma^A_{B, C} \colon \Set(B, C) \to \Set(\Set(A, B), \Set(A, C)),
\quad f \mapsto (g \mapsto f \circ g). \qedhere
\end{equation*}
\end{example}

\begin{example}\label{ex:k-mod_is_closed}
Let \(R\) be a commutative ring with unit \(1\).
Its category of left modules \(\hmodM[R]\) is skew-closed.
The internal hom is given by \([M, N] = \hmodM[R](M, N)\), on which \(R\) acts as \((r .
f)(m) = r f(m)\) for all \(r\in R\).
The unit object is \(R\).
The natural transformation \(i_M \colon \hmodM[R](R, M) \to M\) is the isomorphism
given by \(i_M(f) = f(1)\).
The dinatural transformation \(j_M\from R \to [M, M]\) is again `picking out the identity', i.e.\ it is
the unique \(R\)-module map with \(j_M(1) = \id_M\).
Lastly, the transformation \(\Gamma^M_{N, P}\from [N, P]\to [[M,N],[M,P]] \) is once more given by post-composition: \(\Gamma^M_{N, P} (f) = f \circ -\).
\end{example}

The next example is one of the primary motivations for the present article.
It states that the reconstruction theory of closed monoidal categories is governed by so-called left Hopf monads in the sense of \cite{BruguieresLackVirelizier}.

\begin{example}\label{ex:T-alg_for_T_hopf_is_closed}
  Let \(T\) be a left Hopf monad on the left closed monoidal category \(\cat\).
  In view of \cite[Theorem~3.6]{BruguieresLackVirelizier}, the Eilenberg--Moore category \(\cat^T\) is left closed monoidal.
  The internal hom \([\mb{M},\mb{N}]_T\) for \(T\)-algebras \(\mb{M} = (M,\act_M)\) and \(\mb{N} = (N,\act_N)\) is given by the internal hom \([M,N]\) in \(\cat\) with \(T\)-algebra structure \([M,\act_N] \circ s^{\ell}_{M,N} \circ T[\act_M,N]\), where \(s^{\ell}_{M,N} \colon T[TM,N] \to [M,TN]\) is the left antipode for \(T\), see~\cite[Section~3.3]{BruguieresLackVirelizier}.
  The constraints \(i\), \(j\), and \(\Gamma\) are the same as those of \(\cat\).
\end{example}

Henceforth, we will often denote an algebra \((M,\act_M)\) for a monad \(T\) simply by \(\mb{M}\).

\smallskip

The categories in \cref{ex:set_is_closed,ex:k-mod_is_closed,ex:T-alg_for_T_hopf_is_closed} are, in fact,  normal-closed.
While one can check this directly, it follows most easily from the fact that they are closed monoidal and Theorem~\ref{thm:bijection_skew_closed_monoidal_structures}.
An example of a non normal-closed category is the following.

\begin{example}
  Let \(k\) be a commutative ring and \(R\) be a (not necessarily commutative) \(k\)-algebra such that the multiplication \(\mu\from R \otimes_k R \to R\) is not an isomorphism.
  Consider the category \(\hmodM[R]\) of left \(R\)-modules. For any \(M,N\in \hmodM[R]\), define \([M,N]=\hmodM[k](M,N)\) with the following \(R\)-action.
  \[(r.f)(m) = r \cdot f(m) \qquad \text{for all } f\in \hmodM[k](M,N), r\in R \text{ and }m\in M.\]
  The unit object is given by the regular module \(R\) and the structure maps are
  \begin{align*}
    i_M \colon & \hmodM[k](R,M)\to M, \qquad f \mapsto f(1_R), \\
    j_M\colon & R\to \hmodM[k](M,M), \qquad r \mapsto r.\id_{M}.
  \end{align*}
  Again, \(\Gamma\) is simply post-composition.

  For any \(M, N, P \in \hmodM[R]\), we have an adjunction
  \begin{align*}
    \hmodM[R](.M,\hmodM[k](N,.P))
    & \cong \hmodM[R](.M,\hmodM[R](_{*}R. \otimes_{k} N, {_{*}P}))
      \cong \hmodM[R](({_{*}R}. \otimes_{k} N) \otimes_{R} .M, {_{*}P}) \\
    & \cong \hmodM[R](.M \otimes_{k} N,.P).
  \end{align*}
  Due to Theorem~\ref{thm:bijection_skew_closed_monoidal_structures}, the given skew-closed structure \([\blank,\bblank]\) corresponds to the skew-monoidal structure \(\otimes_{k}\) given by tensoring over \(k\), with the regular left \(R\)-action on the left-hand tensor factor.
  While \(\otimes_{k}\) is associative normal, it is clearly not unital normal.
\end{example}

The next \cref{ex:gabi-algebra} is another primary motivation for the present article, as it shows that the reconstruction theory of closed categories is governed by something more general than left Hopf monads.

\begin{definition}\label{def:gabi-algebra}
  A \emph{gabi-algebra} over a commutative ring \(k\) is a triple \((A, \varepsilon, \delta)\) comprising an augmented algebra \((A, \varepsilon)\) and an algebra morphism \(\delta \from A \to A \otimes A\op\), \(a \mapsto a_{+} \otimes a_{-}\) (with summation understood), satisfying
  \begin{align}
    a_+ \varepsilon(a_{-}) &= a  \tag{GA1}\label{eq:GA1},\\
    a_{+}a_{-} & = \varepsilon(a) 1_{A} \tag{GA2}\label{eq:GA2},\\
    (a_{+})_{+} \otimes (a_{-})_{+} \otimes (a_{-})_{-} \; (a_{+})_{-} &= a_{+} \otimes a_{-} \otimes 1_{A}, \tag{GA3}\label{eq:GA3}
  \end{align}
  for all \(a \in A\).
\end{definition}

Gabi-algebras were introduced and studied in~\cite{Berger-Vercruysse-Saracco} with the aim of developing a Tannakian reconstruction theory in the closed setting.

\begin{example}\label{ex:gabi-algebra}
  Let \(A\) be a gabi-algebra over a commutative ring \(k\).
  The category \({}_A\cM\) of left \(A\)-modules with the skew-closed structure  \(\left({}_A\cM,\cM(\blank,\bblank),k\right)\) lifted from \(k\)-modules is not, in general, normal-closed.
  For example, when \(A\) admits a gabi-algebra structure arising from a one-sided Hopf algebra structure on \(A\), \(\left({}_A\cM, \cM(\blank,\bblank), k\right)\) is skew-closed; see \cite[\S4.5]{Berger-Vercruysse-Saracco}. In fact, \({}_A\cM\) is normal-closed if and only if \(A\) is a Hopf algebra.
\end{example}

\subsection{Closed and lax monoidal functors}\label{sec:closed-functors}

Functors between skew-closed categories can be required to preserve the skew-closed structure in a coherent way.
This leads to the next definition \cite[\S3]{Eilenberg-Kelly}.

\begin{definition}\label{def:closed-functor}
  A functor \(F \colon \cat \to \cat[D]\) between skew-closed  categories is
  \emph{closed} if there is a morphism \(\varphi_0 \colon \tu_{\cat[D]} \to F\tu_{\cat}\) and a
  natural transformation \(\varphi_{X, Y} \colon F([X, Y]_{\cat}) \to {[F X, F Y]}_{\cat[D]}\)
  rendering the following diagrams commutative
  \hspace{-2em}
  \begin{gather}
  \begin{gathered}
    \begin{tikzcd}[ampersand replacement=\&,column sep=50pt]
{F X}
      \& {{[\tu_{\cat[D]}, F X]}_{\cat[D]}}
      \arrow[l, "{i_{F X}}"']\\
      {{F([\tu_{\cat}, X]_{\cat})}} \arrow[r, "{\varphi_{\tu_{\cat}, X}}"']\arrow[u, "{F i_X}"]\&
      {{{[F\tu_{\cat}, F X]}_{\cat[D]}}} \arrow[u, "{{[\varphi_0, F X]}_{\cat[D]}}"']
\end{tikzcd}
    \end{gathered}
    \qquad
    \begin{gathered}
    \begin{tikzcd}[ampersand replacement=\&,column sep=50pt]
{\tu_{\cat[D]}} \arrow[r, "{\varphi_0}"]\arrow[d, "{j_{F X}}"']\& {F \tu_{\cat}} \arrow[d, "{F j_X}"]\\
      {{[F X, F X]_{\cat[D]}}}
      \& {F{[X, X]}_{\cat}}
      \arrow[l, "{\varphi_{X, X}}"]
\end{tikzcd}
    \end{gathered}\\
    \begin{gathered}
    \begin{tikzcd}[ampersand replacement=\&,column sep=40pt]
{{F([X, Y]_{\cat})}} \arrow[r, "{F\Gamma^Z_{X, Y}}"]\arrow[d, "{\varphi_{X, Y}}"']\&
      {{F[[Z, X]_{\cat}, [Z, Y]_{\cat}]_{\cat}}} \arrow[r, "{\varphi_{[Z, X], [Z,Y]}}"]\&
      {{[F[Z, X]_{\cat}, F[Z, Y]_{\cat}]_{\cat[D]}}} \arrow[d, "{[\id, \varphi_{Z, Y}]}"]\\
      {{[F X, F Y]_{\cat[D]}}} \arrow[r, "{\Gamma^{F Z}_{F X, F Y}}"']\&
      {{[[F Z, F X]_{\cat[D]}, [F Z, F Y]_{\cat[D]}]_{\cat[D]}}} \arrow[r, "{[\varphi_{Z, X}, \id]}"']\&
      {{[F[Z, X]_{\cat}, [F Z, F Y]_{\cat[D]}]_{\cat[D]}}}
\end{tikzcd}
    \end{gathered}
  \end{gather}
  If \(\varphi_{0}\) and \((\varphi_{X,Y})_{X,Y \in\, \cat{C}}\) are isomorphisms, then \(\free\) is called  \emph{strong closed}.
  If they are identities, then \(\free\) is called \emph{strict closed}.
\end{definition}

\begin{example}
  The identity functor on a skew-closed category is strict closed.
  If \((F,\varphi_0,\varphi)\) and \((G,\psi_0,\psi)\) are composable closed functors, then \(FG\) is closed with structure morphisms
  \[
    \xi_0 = F \psi_0 \circ \varphi_0 \from \tu \to FG\tu
    \qquad\text{and}\qquad
    \xi = \varphi_{G,G} \circ F \psi \from FG[\blank, \bblank] \to [FG \blank, FG \bblank].\qedhere
  \]
\end{example}

\begin{example}
    Let \(R\) be a commutative ring and let \(\prescript{}{R}{\cat[M]}\) be its category of left modules, equipped with the closed structure from \cref{ex:k-mod_is_closed}. The forgetful functor \(\forget \colon \prescript{}{R}{\cat[M]} \to \mathsf{Ab}\) is closed with \(\mathbb{Z} \xrightarrow{u} \forget(R)\) and \(\forget\left(\prescript{}{R}{\cat[M]}(M,N)\right) \subseteq \mathsf{Ab}(M,N)\), but it is neither strong nor strict closed in general.
\end{example}

\begin{example}\label{ex:forget_hopf_monad_is_strong_closed}
  Let \(T\) be a left Hopf monad on the left closed monoidal category \(\cat\). As in \cref{ex:T-alg_for_T_hopf_is_closed}, \(\cat^T\) can be endowed with a left closed monoidal structure. With respect to this structure, the forgetful functor \(\U^T \colon \cat^T \to \cat\) is strict closed.
\end{example}

\begin{definition}[{\cite[\S4]{Eilenberg-Kelly}}]\label{def:closed_natural_transformation}
    Let \((F,\varphi_0,\varphi)\) and \((G,\psi_0,\psi)\) be closed functors between skew-closed categories \(\cat{C}\) and \(\cat{D}\). A \emph{closed natural transformation} from \(F\) to \(G\) consists of a natural transformation \(\alpha \colon F \to G\) such that
    \begin{equation}
        \begin{gathered}
            \begin{tikzcd}[ampersand replacement=\&]
\& {F(\tu_{\cat{C}})} \arrow[dd, "{\alpha_{\tu_{\cat{C}}}}"]\\
                {\tu_{\cat{D}}} \arrow[ur, "{\varphi_0}"]\arrow[dr, "{\psi_0}"']\& \\
                 \& {G(\tu_{\cat{C}})}
\end{tikzcd}
        \end{gathered}
        \qquad \text{and} \qquad
        \begin{gathered}
            \begin{tikzcd}[ampersand replacement=\&]
{F([X,Y]_{\cat{C}})} \arrow[r, "{\varphi_{X,Y}}"]\arrow[dd, "{\alpha_{[X,Y]_{\cat{C}}}}"']\& {[F X,F Y]_{\cat{D}}} \arrow[d, "{[F X,\alpha_Y]_{\cat{D}}}"]\\
                 \& {[F X,G Y]_{\cat{D}}} \\
                {G([X,Y]_{\cat{C}})} \arrow[r, "{\psi_{X,Y}}"']\& {[G X,G Y]_{\cat{D}}} \arrow[u, "{[\alpha_X,G Y]_{\cat{D}}}"']
\end{tikzcd}
        \end{gathered}
    \end{equation}
    commute for every \(X,Y\) in \(\cat{C}\).
\end{definition}

We now turn to the monoidal side.

\begin{definition}\label{def:lax-mon-fun}
  Consider skew-monoidal categories \( \cat{C}\) and \( \cat{D}\).
  A \emph{lax monoidal functor} consists of a functor \(\free \from \cat{C} \to \cat{D} \),
  a natural transformation \(\phi \from \free(\blank) \otimes_{ \cat{D}} \free(\bblank) \to \free(\blank\otimes_{ \cat{C}}\bblank)\),
  and a morphism \(\phi_{0} \from \tu_{ \cat{D}}\to \free\tu_{ \cat{C}}\),
  satisfying for all \(X,Y,Z \in \cat{C}\) the identities
  \begin{gather*}
    \lambda_{\free X} = \free\lambda_{X} \circ \phi_{\tu_{\cat{C}}, X} \circ (\phi_{0} \otimes_{ \cat{D}} \free X), \qquad
    \free(\rho_{X}) = \phi_{X, \tu_{ \cat{C}}} \circ (\free X \otimes_{ \cat{D}} \phi_{0}) \circ \rho_{\free X},\\
    \free\alpha_{X,Y,Z} \circ \phi_{X\otimes Y, Z} \circ (\phi_{X,Y} \otimes _{ \cat{D}} \free Z) =
    \phi_{X, Y\otimes Z} \circ (\free X \otimes_{ \cat{D}} \phi_{Y, Z}) \circ \alpha_{\free X, \free Y, \free Z}. \qedhere
  \end{gather*}
\end{definition}

From the correspondence between skew-closed and skew-monoidal structures on a category, we get a similar correspondence between closed and lax monoidal structures on a functor between skew-closed monoidal categories.
In fact, \cite[Chapter~II, (3.23), Propositions~4.3 and~4.4]{Eilenberg-Kelly} readily translate to the skew-closed monoidal setting; see~\cite[Remark~19]{Street-skew_closed}.

\begin{proposition}\label{prop:closed-mon-bij}
  Let  \(\free \from \cat{C} \to \cat{D}\) be a functor between  skew-closed monoidal categories.
  Then there is a bijective correspondence between closed and lax monoidal structures on \(\free\).

\end{proposition}
\begin{proof}
  Given a closed structure \(\varphi_{0} \from \tu_{\cat{D}} \to F\tu_{\cat{C}} \) and \(\varphi_{X,Y} \from F([X, Y]_{ \cat{C}}) \to [F X, F Y]_{ \cat{D}}\), the associated lax monoidal structure is constructed by setting \(\phi_0 \defeq \varphi_0\),
  and, for all \(X,Y \in \cat{C}\), defining \(\phi_{X,Y}\) to be the composition
  \[
    \mathscale{0.88}{
    F X \otimes F Y
    \xrightarrow{F \coev^Y_X \otimes F Y} F[Y, X \otimes Y] \otimes F Y
    \xrightarrow{\varphi_{Y,X\otimes Y} \otimes F Y} [F Y, F(X \otimes Y)] \otimes F Y
    \xrightarrow{\ev^{F Y}_{F(X \otimes Y)}} F(X \otimes Y).}
  \]

  The other way around, suppose we start with a lax monoidal structure \((\phi_0, \phi)\) on \(\free\).
  Then we set \(\varphi_0 \defeq \phi_0\) and define \(\varphi_{X,Y}\) to be the composition
  \[
    \mathscale{.88}{
    F[X,Y]
    \xrightarrow{\coev^{F X}_{F[X,Y]}} [F X, F[X,Y] \otimes F X]
    \xrightarrow{[F X, \phi_{[X,Y],X}]} [F X, F([X,Y] \otimes X)]
    \xrightarrow{[F X,F\ev^X_Y]} [F X,F Y].}\qedhere
  \]
\end{proof}

However, it is in general not true that strong or strict closed structures induce strong, respectively strict monoidal ones (or conversely), as the following examples highlight.

\begin{example}
  Let \(B\) be a bialgebra over a commutative ring \(k\).
  The category \(\left({}_B\cM,\otimes,k\right)\) of left \(B\)-modules with the monoidal structure lifted from \(k\)-modules is closed monoidal with respect to the internal hom \([M,N] \defeq {}_B\cM(B \otimes M,N)\), where \(B \otimes M\) has the diagonal left action induced by \(\Delta\).
  The forgetful functor \(\forget \colon {}_B\cM \to \cM\) is strict monoidal, but in general neither strict nor strong closed.
\end{example}

\begin{example}\label{ex:gabi-algebra-forgetful-functor}
  Let \(A\) be a gabi-algebra over a commutative ring \(k\). As in \cref{ex:gabi-algebra}, the category \(\left({}_A\cM,\cM(\blank,\bblank),k\right)\) of left \(A\)-modules is skew-closed with the skew-closed structure lifted from \(k\)-modules. It is, in fact, a skew-closed monoidal category with respect to the monoidal product
  \begin{equation*}
    M \boxtimes N \coloneqq (A \otimes N) \otimes_A M,
  \end{equation*}
  where \(A \otimes N\) is an \(A\)-bimodule with respect to the actions \(a \cdot (b \otimes n) \cdot c = abc_+ \otimes c_-\cdot n\) for all \(a,b,c \in A\), \(n \in N\) (see \cite[Proposition~4.12]{Berger-Vercruysse-Saracco}).
  The forgetful functor \(\forget \colon {}_A\cM \to \cM\) is strict closed, but neither strict nor strong monoidal in general, see \cite[Example~4.4]{Berger-Vercruysse-Saracco}.
\end{example}

\section{Gabi-monads and Tannaka--Krein reconstruction of closed categories}\label{sec:gabi}

In the representation theoretic approach of \cite{Moerdijk2002,McCrudden2002,Bruguieres2007,BruguieresLackVirelizier},
bimonads correspond to lifts of monoidal structures.
We will now develop a notion of gabi-monads that follows the same principle for closed categories.
To do this, we first need to determine which transformations on a monad correspond to lifts of internal homs.

\subsection{Gabi-monads}\label{sec:gabi-monads}

The next result is a special case of \cite[Theorem~A.2]{Berger-Vercruysse-Saracco} (reported below as \cref{prop:lifting-co-and-contra}).

\begin{lemma}[{\cite[Lemma~3.1]{Berger-Vercruysse-Saracco}}]\label{lem:gabi-functor-lift}
  Let \(\cat\) be a category, \([\blank,\bblank] \colon \cat\op \times \cat \to \cat\) a functor, and \(T\) a monad on \(\cat{C}\).
  Then liftings of \([\blank,\bblank]\) to a functor \([\blank,\bblank]_T\from {(\cat{C}^T)}\op \times \cat{C}^T \to \cat{C}^T\) are in bijective correspondence
  with natural transformations \(s \colon T[T\blank, \bblank] \to [\blank, T\bblank]\) rendering the following diagrams commutative for all \(X, Y \in \cat\):
  \begin{equation}\label{eq:lifting_homs}
    \begin{gathered}
      \begin{tikzcd}[ampersand replacement=\&,column sep=35pt,row sep=35pt]
{[TX,Y]} \arrow[dr, "{[\eta_X,\eta_Y]}"']\arrow[r, "{\eta_{[TX,Y]}}"]\& {T[TX,Y]} \arrow[d, "{s_{X,Y}}"]\\
        \& {[X,TY]}
\end{tikzcd}
    \end{gathered}
    \quad
    \begin{gathered}
      \begin{tikzcd}[ampersand replacement=\&,column sep=35pt,row sep=35pt]
{T^2[TX,Y]} \arrow[d, "{T^2[\mu_X,Y]}"']\arrow[r, "{\mu_{[TX,Y]}}"]\& {T[TX,Y]} \arrow[r, "{s_{X,Y}}"]\& {[X,TY]} \\
        {T^2[T^2X,Y]} \arrow[r, "{Ts_{TX,Y}}"']\& {T[TX,TY]} \arrow[r, "{s_{X,TY}}"']\& {[X,T^2Y]} \arrow[u, "{[X,\mu_Y]}"']
\end{tikzcd}
    \end{gathered}
  \end{equation}
\end{lemma}

\begin{example}
    When \(T = H\otimes \blank \from  \Vect_{  \Bbbk} \to \Vect_{  \Bbbk}\) is given by tensoring with a \(  \Bbbk\)-Hopf algebra \(H\), the transformation \(s \colon T[T\blank, \bblank] \to [\blank, T\bblank]\) corresponds to a variant of the adjoint action on \(H\)-linear morphisms.
    In fact, let \(S \from H\to H\) be the antipode of \(H\) and \(\Delta\from H\to H\otimes H\) the comultiplication.
    By using reduced Sweedler notation, we write  \(h_{(1)} \otimes h_{(2)}\eqdef \Delta(h)\) for all \(h \in H\).
    From Example~3.2 of \cite{Berger-Vercruysse-Saracco} one can deduce that
    \begin{equation*}
      s_{X,Y}\from H \otimes \Hom_{  \Bbbk}(H\otimes X, Y) \to \Hom_{ \Bbbk}(X, H\otimes Y), \quad
      h \otimes f \mapsto \big(x \mapsto h_{(1)} \otimes f(S(h_{(2)}) \otimes x) \big). \qedhere
    \end{equation*}
\end{example}

\begin{definition}\label{def:gabi_monad}
  A \emph{gabi-monad} is a triple \((T, s, \act_{\tu})\) consisting of a monad \((T, \mu, \eta)\) on a skew-closed category \(\cat{C}\),  a natural transformation
  \(s \colon T[T\blank, \bblank] \to [\blank, T\bblank]\),
  and a morphism \(\act_\tu \colon T \tu \to \tu\)
  subject to the following axioms:
  \begin{enumerate}[label=(GM\arabic*),leftmargin=1.7cm]
    \item \(\act_\tu \colon T \tu \to \tu\) is a \(T\)-algebra structure,
    \item  \(s \colon T[T\blank, \bblank] \to [\blank, T\bblank]\) makes the diagrams \eqref{eq:lifting_homs} commutative, and
    \item for all \(X,Y \in \cat\) and \((M,\act_M),(P,\act_P) \in \cat^T\), the diagrams  below commute:
    \begin{gather}\label{eq:siotaj}
      \begin{gathered}
        \begin{tikzcd}[ampersand replacement=\&,column sep=35pt,row sep=35pt]
{T[\tu,X]} \arrow[r, "{Ti_X}"]\arrow[d, "{T[\act_\tu,X]}"']\& {TX} \\
          {T[T\tu,X]} \arrow[r, "{s_{\tu,X}}"']\& {[\tu,TX]} \arrow[u, "{i_{TX}}"']
\end{tikzcd}
      \end{gathered}
      \qquad
      \begin{gathered}
        \begin{tikzcd}[ampersand replacement=\&,column sep=35pt,row sep=35pt]
{T\tu} \arrow[r, "{\act_\tu}"]\arrow[d, "{Tj_{TM}}"']\& {\tu} \arrow[r, "{j_M}"]\& {[M,M]} \\
          {T[TM,TM]} \arrow[r, "{s_{M,TM}}"']\& {[M,T^2M]} \arrow[r, "{[M,\mu_M]}"']\& {[M,TM]} \arrow[u, "{[M,\act_M]}"']
\end{tikzcd}
      \end{gathered} \\
      \label{eq:sGamma}
      \begin{gathered}
        \mathscale{.95}{\begin{tikzcd}[ampersand replacement=\&,column sep=55pt,row sep=35pt]
{T[TX,Y]} \arrow[d, "{T\Gamma_{TX,Y}^P}"']\arrow[r, "{s_{X,Y}}"]\& {[X,TY]} \arrow[r, "{\Gamma_{X,TY}^P}"]\& {[[P,X],[P,TY]]} \\
            {T[[P,TX],[P,Y]]} \arrow[r, "{T[s_{P,X},[\act_P,Y]]}"']\& {T[T[TP,X],[TP,Y]]} \arrow[r, "{s_{[TP,X],[TP,Y]}}"']\& {[[TP,X],T[TP,Y]]} \arrow[u, "{[[\act_P,X],s_{P,Y}]}"']
\end{tikzcd}}
      \end{gathered}
    \end{gather}
  \end{enumerate}
\end{definition}

Gabi-monads  generalise the notion of gabi-algebras, developed in \cite{Berger-Vercruysse-Saracco}.

\begin{example}\label{ex:gabi-algebras}
  Consider a gabi-algebra \((A, \varepsilon, \delta)\) over a commutative ring \(k\).
  We may endow \(A\otimes \blank \from \cM \to \cM\) with the structure of a gabi-monad by setting
  \begin{align*}
    \mathfrak{a}_{\tu}& \from A\otimes \tu = A \otimes k \to k= \tu, & \mathfrak{a}_{\tu}(a \otimes 1) & = \varepsilon(a)\\
    s_{X,Y} & \from A\otimes \cM(A\otimes X , Y) \to \cM(X, A \otimes Y),              & s_{X,Y}(a \otimes f)(x)      & = a_{+} \otimes f(a_{-}\otimes x). \qedhere
  \end{align*}
\end{example}
Suppose \((T, s, \act_{\tu})\) is a gabi-monad on a skew-closed category \( \cat{C}\).
For any pair of \(T\)-algebras \((M, \act_{M}) \) and \((N, \act_{N})\), we define a morphism
\(\act_{M}\star \act_{N} \from T[M, N]\to [M,N]\) as the composition
\begin{equation}\label{eq:ActfromS}
  T[M, N] \xrightarrow{T[\act_M, N]} T[TM, N] \xrightarrow{s_{M, N}} [M, TN] \xrightarrow{[M, \act_N]} [M, N].
\end{equation}
The commutativity of Diagram~\eqref{eq:lifting_homs} implies that \(([M,N], \act_{M}\star \act_{N})\)
is a \(T\)-algebra, allowing us to define the functor
\begin{equation}\label{eq:lifted-internal-hom}
  \begin{gathered}
    [\blank, \bblank]_T \from {(\cat{C}^{T})}\op\! \times \cat{C}^{T} \to \cat{C}^{T},\qquad \big((M, \act_{M}),(N, \act_{N})\big) \mapsto \big([M, N], \act_{M}\star \act_{N}).
  \end{gathered}
\end{equation}

The following lemma is essentially a part of \cite[Theorem~3.4]{Berger-Vercruysse-Saracco}, and a detailed proof can be found therein.

\begin{lemma}\label{lem:gabi-monads-lift-skew-closed-structures}
  Let \((T, s, \act_{\tu})\) be a gabi-monad on a skew-closed category \(( \cat{C}, [ \blank , \bblank], \tu, \Gamma, i, j)\).
  Its Eilenberg--Moore category \( \cat{C}^{T}\) becomes skew-closed with the internal hom \([\blank,\bblank]_T\) of Equation~\eqref{eq:lifted-internal-hom}, the unit \((\tu, \act_{\tu})\), and the coherence morphisms \(\Gamma,i,j\) inherited from \( \cat{C}\).
\end{lemma}
\begin{proof}
  A direct computation shows that the commutativity of the Diagrams~\eqref{eq:siotaj} and \eqref{eq:sGamma} entails that the constraints \(\Gamma,i,j\) lift to \(\cat^T\).
\end{proof}

By rephrasing~\cite[Theorem~3.4]{Berger-Vercruysse-Saracco}, we obtain that gabi-monads are Tannaka-dual to skew-closed structures.

\begin{theorem}\label{thm:gabi-bijection}
  Let \(T\) be a monad on a skew-closed category \(\cat{C}\).
  There exists a bijective correspondence between gabi-monad structures on \(T\)
  and skew-closed structures on \(\cat{C}^T\) such that \(\U^T\from \cat{C}^T \to \cat{C}\) is strict closed.
\end{theorem}
\begin{proof}
  The fact that \(T\) is a gabi-monad if and only if there exists a skew-closed structure on \(\cat{C}^T\) such that \(\U^T\from \cat{C}^T \to \cat{C}\) is strict closed is proven in \cite[Theorem~3.4]{Berger-Vercruysse-Saracco}.
  We recall the constructions here, for the convenience of the reader.

  Suppose that \(\cat^T\) is skew-closed in such a way that \(\U^T\) is strict closed.
  Then \(\tu\) admits a \(T\)-algebra structure \(\act_\tu \colon T\tu \to \tu\) and, for all \((M, \act_M), (N, \act_N) \in \cat^T\), the object \([M, N]\) is equipped with some action of \(T\),
  which we denote by \(\act_M \star \act_N \colon T[M, N] \to [M, N]\).
  For all \(X, Y \in \cat\) we can thus define a natural transformation \(s_{X,Y}\from T [ TX, Y] \to [X,TY]\) as the composition
  \begin{equation}\label{eq:SfromAct}
    T [TX, Y] \xrightarrow{T[TX, \eta_Y]} T[TX, TY] \xrightarrow{\mu_X \star \mu_Y} [TX, TY] \xrightarrow{[\eta_X, TY]} [X, TY]
  \end{equation}
  and a direct computation shows that Diagram~\eqref{eq:lifting_homs} commutes.
  Moreover, the fact that for \((M,\act_M)\), \((N,\act_N)\), \((P,\act_P)\) in \(\cat^T\), the coherence maps  \(i_M,j_M,\Gamma^M_{N,P}\) are morphisms of \(T\)-algebras entails that \eqref{eq:siotaj} and \eqref{eq:sGamma} commute.

  Conversely, if \((T, s, \act_{\tu})\) is a gabi-monad, \( \cat{C}^{T}\) can be endowed with the skew-closed structure of Lemma~\ref{lem:gabi-monads-lift-skew-closed-structures}.
  In this case, \(\U^{T}\from \cat{C}^{T}\to \cat{C}\) is strict closed.

  These constructions are easily seen to be inverse to each other.
  For example, suppose that \((T, s,\act_{\tu}) \) is a gabi-monad.
  The map \(s'\) induced from the closed structure on \( \cat{C}^{T}\) is
  \[
    \begin{tikzcd}[ampersand replacement=\&,column sep=60pt]
      {T[TX,Y]} \& \& \& {[X,TY]} \\
      {T[TX,TY]} \& {T[T^2X,TY]} \& {[TX, T^2Y]} \& {[TX,TY]}
	   \arrow["{s'_{X,Y}}",dashed, from=1-1, to=1-4]
      \arrow["{T[TX,\eta_Y]}"', from=1-1, to=2-1]
      \arrow["{T[\mu_X,TY]}"', from=2-1, to=2-2]
      \arrow["{s_{TX,TY}}"', from=2-2, to=2-3]
      \arrow["{[TX,\mu_Y]}"', from=2-3, to=2-4]
      \arrow["{[\eta_X,TY]}"', from=2-4, to=1-4]
    \end{tikzcd}
  \]
  This is seen to be equal to \(s\) by the commutativity of the following diagram:
  \[\mathscale{0.9}{
    \begin{tikzcd}[ampersand replacement=\&,column sep=large]
      {T[TX,Y]} \& {T[TX,TY]} \& {T[T^2X,TY]} \& {[TX, T^2Y]} \& {[TX,TY]} \\
      \&\& {T[TX,TY]} \\
      {[X,TY]} \&\&\& {[X,T^2Y]} \& {[X, TY]}
      \arrow["{T[TX,\eta_Y]}", from=1-1, to=1-2]
      \arrow["{s_{X,Y}}"', from=1-1, to=3-1]
      \arrow["{T[\mu_X,TY]}", from=1-2, to=1-3]
      \arrow[equals, from=1-2, to=2-3]
      \arrow["{s_{TX,TY}}", from=1-3, to=1-4]
      \arrow["{T[T\eta_{X},TY]}", from=1-3, to=2-3]
      \arrow["{[TX,\mu_Y]}", from=1-4, to=1-5]
      \arrow["{[\eta_X,T^2Y]}", from=1-4, to=3-4]
      \arrow["{[\eta_X,TY]}", from=1-5, to=3-5]
      \arrow["{s_{X,TY}}"{description}, from=2-3, to=3-4]
      \arrow["{[X,T\eta_{Y}]}", from=3-1, to=3-4]
      \arrow["{[X,\mu_Y]}", from=3-4, to=3-5]
      \arrow[equals, from=3-1, to=3-5, rounded corners,
      to path={(\tikztostart.south) -- ([yshift=-.7cm]\tikztostart.center) -- ([yshift=-.7cm]\tikztotarget.center) \tikztonodes -- (\tikztotarget.south)}]
      \arrow["s'", color=red, from=1-1, to=3-5, rounded corners,
        to path={(\tikztostart.north) -- ([yshift=.7cm]\tikztostart.center) -- ([yshift=3.7cm, xshift=1.5cm]\tikztotarget.center)\tikztonodes -- ([xshift=1.5cm]\tikztotarget.center)  -- (\tikztotarget.east)}]
    \end{tikzcd}
    }
  \]
  The other direction is similar.
\end{proof}

\begin{remark}
  There exist gabi-monads \(T\) on closed monoidal categories \(\cat[C]\) for which \(\cat[C]^T\) is just skew-closed with respect to the lifted internal hom, as can be inferred from \cite[\S4.5]{Berger-Vercruysse-Saracco}. This is due to the fact that two of the normality conditions, namely \ref{item:N1} and \ref{item:N3}, are external conditions---they are bijections in \(\Set\)---and so they are not automatically satisfied by the lifted internal hom. This leads to the forthcoming \cref{def:normal_gabi_monad}.
\end{remark}

\begin{definition}\label{def:normal_gabi_monad}
    A gabi-monad \((T, s, \act_{\tu})\) on a closed category \(\cat\) is said to be \emph{normal}, or a \emph{normal gabi-monad}, if the skew-closed structure on \(\cat^T\) for which the forgetful functor \(\cat^T \to \cat\) is strict closed is even normal-closed.
\end{definition}

According to Theorem~\ref{thm:bijection_skew_closed_monoidal_structures}, the closed and monoidal coherence maps share the same ``strength''. In particular, invertibility of the unitors and associator on the monoidal side is equivalent to the normality conditions of the closed side being satisfied.
Applied to left Hopf monads, this provides us with our first examples of normal gabi-monads.

\begin{example}\label{ex:Hopf-monads-are-normal-gabi}
  Let \(H\) be a left Hopf monad in the sense of~\cite{BruguieresLackVirelizier} on a left closed monoidal category \(\cat{C}\).
  By Example~\ref{ex:T-alg_for_T_hopf_is_closed}, its Eilenberg--Moore category \( \cat{C}^{H}\) is closed monoidal with the same coherence maps.
  As the monoidal coherences of \(\cat{C}\) are invertible, \cref{thm:bijection_skew_closed_monoidal_structures} implies that \(H\) is a normal gabi-monad.
\end{example}

In~\cite{Berger-Vercruysse-Saracco}, it was shown that a converse of this observation holds in a ring-theoretic context: any normal gabi-algebra over a commutative base ring is a Hopf algebra.
This is not the case anymore for gabi-monads.
In~\cref{sec:gabi-not-hopf}, we will study normal gabi-monads that are not left Hopf.

\subsection{Gabi-monads from adjunctions}\label{sec:gabi-monads-from-adjunctions}

Consider an ordinary adjunction \(\adj{\free}{\forget}{\cat{C}}{\cat{D}}\) between two skew-closed categories \(\cC\) and \(\cat{D}\). Suppose that \(\forget\) is strong closed with structure morphisms \(\varphi_0 \colon \tu_\cC \to \forget\left(\tu_\cD\right)\) and \(\varphi_{A,B} \colon \forget[A,B]_\cD \to [\forget A,\forget B]_\cC\), for objects \(A,B\) in \(\cD\) and set \(T \coloneqq \forget\free\).

Since \(\forget\) is strong closed, \(\varphi_0\) is an isomorphism, and so \(\tu_{\cat{C}}\) naturally becomes a \(T\)-algebra by transport of structure, that is, with respect to
\begin{equation}\label{eq:gabi-def-1}
    \act_{\tu_{\cat{C}}} \colon T\tu_{\cat{C}} \xrightarrow{T\varphi_0} T\forget\tu_{\cat{D}} \xrightarrow{\forget\epsilon_{\tu_{\cat{D}}}} \forget\tu_{\cat{D}} \xrightarrow{\varphi_0^{-1}} \tu_{\cat{C}}.
\end{equation}
Similarly, also \([\forget A,\forget B]_\cC\) becomes a \(T\)-algebra for every \(A,B\) in \(\cD\) by transport of structure, that is to say, with respect to
\begin{equation}\label{eq:gabi-def-2}
    \act_{[\forget A,\forget B]} \colon T[\forget A,\forget B] \xrightarrow{T\varphi_{A,B}^{-1}} T\forget[A,B] \xrightarrow{\forget\epsilon_{[A,B]}} \forget[A,B] \xrightarrow{\varphi_{A,B}} [\forget A,\forget B],
\end{equation}
natural in \(A,B\) in \(\cat{D}\).
As a consequence, \([TX,TY]_\cC\) also becomes a \(T\)-algebra for every \(X\) and \(Y\) in \(\cC\), with respect to \(\act_{[TX,TY]}\) natural in \(X,Y\) in \(\cat{C}\). This allows us to consider a natural transformation \(s_{X,Y} \colon T[TX,Y] \to [X,TY]\) as in \eqref{eq:SfromAct}, i.e.
\begin{equation}\label{eq:Sclosed}
  \begin{gathered}
    \begin{tikzcd}[ampersand replacement=\&,cramped,column sep=60pt]
{T[TX,Y]} \arrow[d, dotted, "{s_{X,Y}}"']\arrow[r, "{T[TX,\eta_Y]}"]\& {T[TX,TY]} \arrow[r, "{T\varphi_{\free X,\free Y}^{-1}}"]\arrow[d, "{\act_{[TX,TY]}}"]\& {\forget\free\forget[\free X,\free Y]} \arrow[d, "{\forget\epsilon_{[\free X,\free Y]}}"]\\
      {[X,TY]} \& {[TX,TY]} \arrow[l, "{[\eta_X,TY]}"]\& \arrow[l, "{\varphi_{\free X,\free Y}}"]\forget[\free X,\free Y],
\end{tikzcd}
  \end{gathered}
\end{equation}
and a morphism \(\act_{[M,N]} \coloneqq \act_M \star \act_N \colon T[M,N] \to [M,N]\) as in \eqref{eq:ActfromS}, natural in \((M,\act_M), (N,\act_N)\) in \(\cat{C}^T\). Note that \(\mu_X \star \mu_Y = \act_{[TX,TY]}\) for every \(X,Y\) in \(\cat{C}\).

Remark also that, by definition of \(s\) and naturality of \(\eta\),
\begin{equation}\label{eq:s_and_eta}
    s_{X,Y} \circ \eta_{[TX,Y]} = [\eta_X,\eta_Y]
\end{equation}
for all \(X,Y\) in \(\cat{C}\). Furthermore,
\cref{fig:GammaTalgs} shows how, in the present setting, \(\Gamma_{TX,TY}^{TP}\) becomes a morphism of \(T\)-algebras for every \(X,Y,P\) in \(\cat\). That is to say, it shows that
\begin{equation}\label{eq:GammaTalgsFree}
  \begin{aligned}
    \Gamma_{TX,TY}^{TP} & \circ \act_{[TX,TY]} = \left(\act_{[TP,TX]} \star \act_{[TP,TY]}\right) \circ T \Gamma_{TX,TY}^{TP} \\
    & \stackrel{\eqref{eq:ActfromS}}{=} [[TP,TX],\act_{[TP,TY]}] \circ s_{[TP,TX],[TP,TY]} \circ T\left[\act_{[TP,TX]},[TP,TY]\right] \circ T\Gamma_{TX,TY}^{TP}.
  \end{aligned}
\end{equation}
These relations shall be used to prove the next \cref{prop:gabi-moand-VS-strong-closed}.

\begin{proposition}\label{prop:gabi-moand-VS-strong-closed}
    Consider an ordinary adjunction \(\adj{\free}{\forget}{\cat{C}}{\cat{D}}\) between two skew-closed categories \(\cat{C}\) and \(\cat{D}\).
    There is a bijective correspondence between strong closed structures on the right adjoint \(\forget\) on one side and pairs consisting of a gabi-monad structure on \(T \defeq \forget\free\) and a strong closed structure on the comparison functor \(K \colon \cat[D] \to \cat[C]^T\) on the other. That is to say,
    \[\left\{\begin{array}{cc}\text{strong closed structures } \\ \text{on } \forget\end{array}\right\} \leftrightarrow \left\{\begin{array}{cc}\text{gabi-monad structures on } T \\ \text{and strong closed structures on } K\end{array}\right\}.\]
    If either structure is given, then \(\U^T \circ K = \forget\) holds as closed functors.
\end{proposition}

\begin{proof}
    Suppose that \((\varphi_0,\varphi)\) is a strong closed structure on \(\forget\), to begin with.
    Define \(\act_{\tu_{\cat{C}}}\) to be the composition \eqref{eq:gabi-def-1} and \(s_{X,Y}\) to be the composition \eqref{eq:Sclosed}.
    Then, the commutativity of
    \begin{equation*}
        \begin{tikzcd}[ampersand replacement=\&]
            {T[TX,Y]} \&\&\&\& {T[TX,TY]} \\
            {[TX,Y]} \&\& {[TX,TY]} \\
            {[X,TY]} \&\&\&\& {[TX,TY]}
            \arrow[""{name=0, anchor=center, inner sep=0}, "{T[TX,\eta_Y]}", from=1-1, to=1-5]
            \arrow[""{name=1, anchor=center, inner sep=0}, "{\act_{[TX,TY]}}", from=1-5, to=3-5]
            \arrow["{\eta_{[TX,Y]}}", from=2-1, to=1-1]
            \arrow["{[TX, \eta_{Y}]}",from=2-1, to=2-3]
            \arrow["{[\eta_X,\eta_Y]}"', from=2-1, to=3-1]
            \arrow["{\eta_{[TX,TY]}}"{description},from=2-3, to=1-5]
            \arrow[equals, from=2-3, to=3-5]
            \arrow[""{name=2, anchor=center, inner sep=0}, "{[\eta_X,TY]}", from=3-5, to=3-1]
            \arrow["{\eta\ \mathsf{nat}}"', draw=none, from=2-3, to=0]
            \arrow["{[\blank, \bblank]\ \mathsf{functor}}"', draw=none, from=2-3, to=2]
            \arrow["{T\mbox{-}\mathsf{alg}}"', draw=none, from=2-3, to=1]
            \arrow["{s_{X,Y}}"{description}, from=1-1, to=3-1, rounded corners,
            to path={
            (\tikztostart.north)
            -- ([yshift=.7cm]\tikztostart.center)
            -- ([xshift=10cm,yshift=.7cm]\tikztostart.center)
            \tikztonodes
            -- ([xshift=10cm,yshift=-.7cm]\tikztotarget.center)
            -- ([yshift=-.7cm]\tikztotarget.center)
            -- (\tikztotarget.south)}]
        \end{tikzcd}
    \end{equation*}
    and of \cref{fig:gabi-moand-VS-strong-closed-1} ensures that the diagrams from \eqref{eq:lifting_homs} are commutative.
    The commutativity of the leftmost diagram in \eqref{eq:siotaj} follows from \cref{fig:badmagic},
    the commutativity of the rightmost diagram in \eqref{eq:siotaj} follows from the commutativity of \cref{fig:gabi-moand-VS-strong-closed-2},
    and \cref{fig:gabi-moand-VS-strong-closed-3} witnesses the commutativity of \eqref{eq:sGamma}. Therefore \((T,s,\act_\tu)\) is a gabi-monad.

    To show that the comparison functor \(K\from \cat{D}\to \cat{C}^T\), \(D \mapsto (\forget D, \forget\epsilon_{D})\) is strong closed,
    it is enough to prove that \(\varphi\) and \(\varphi_0\) are morphisms of \(T\)-algebras.
    For \(\varphi_{A,B}\), with \(A,B\) objects in \( \cat{D}\), this follows from the commutativity of the below diagram.
    \[
        \mathscale{.8}{\begin{tikzcd}[ampersand replacement=\&,cramped,column sep=13pt,row sep=20pt]
            {\forget\free[\forget A,\forget B]} \&\& {\forget\free[\forget\free\forget A,\forget B]} \& {[\forget A,\forget\free\forget B]} \&\& {[\forget A,\forget B]} \\
            \& {\forget\free[\forget A,\forget\free\forget B]} \& {\forget\free[\forget\free\forget A,\forget\free\forget B]} \& {[\forget\free\forget A,\forget\free\forget B]} \& {[\forget A,\forget\free\forget B]} \\
            \&\& {\forget\free\forget[\free\forget A,\free\forget B]} \& {\forget[\free\forget A,\free\forget B]} \\
            \& {\forget\free\forget[A,\free\forget B]} \&\&\& {\forget[A,\free\forget B]} \\
            {\forget\free\forget[A,B]} \&\&\& {\forget[A,B]} \&\& {[\forget A,\forget B]}
            \arrow[""{name=0, anchor=center, inner sep=0}, "{{\forget\free[\forget\epsilon_A,\id]}}", from=1-1, to=1-3]
            \arrow["{{\forget\free[\id,\eta_{\forget B}]}}"{description}, from=1-1, to=2-2]
            \arrow["{{\forget\free\varphi^{-1}_{A,B}}}"', from=1-1, to=5-1]
            \arrow["{{s_{\forget A,\forget B}}}", from=1-3, to=1-4]
            \arrow["{{\forget\free[\id,\eta_{\forget B}]}}"', from=1-3, to=2-3]
            \arrow["{\mathsf{def}\ s}"{description}, draw=none, from=1-3, to=2-4]
            \arrow["{{[\id,\forget\epsilon_{B}]}}", from=1-4, to=1-6]
            \arrow[equals, from=1-4, to=2-5]
            \arrow["{\mathsf{nat}\ \varphi}"{description}, draw=none, from=1-6, to=4-5]
            \arrow[""{name=1, anchor=center, inner sep=0}, "{\forget\free[\forget\epsilon_A,\id]}"'{yshift=-4pt}, from=2-2, to=2-3]
            \arrow["{{\forget\free\varphi^{-1}_{A,\free\forget B}}}"', from=2-2, to=4-2]
            \arrow["{{\forget\free\varphi^{-1}_{\free\forget A,\free\forget B}}}", from=2-3, to=3-3]
            \arrow["{{[\eta_{\forget A},\id]}}"{description}, from=2-4, to=1-4]
            \arrow[""{name=2, anchor=center, inner sep=0}, "{{\forget\epsilon_{[\free\forget A,\free\forget B]}}}"'{yshift=-4pt}, from=3-3, to=3-4]
            \arrow["{{\varphi_{\free\forget A,\free\forget B}}}"', from=3-4, to=2-4]
            \arrow[""{name=3, anchor=center, inner sep=0}, "{\forget\free\forget[\epsilon_A,\id]}"{description}, from=4-2, to=3-3]
            \arrow[""{name=4, anchor=center, inner sep=0}, "{{\forget\epsilon_{[A,\free\forget B]}}}"', from=4-2, to=4-5]
            \arrow[""{name=5, anchor=center, inner sep=0}, "{\forget\free\forget[A,\epsilon_B]}"{description}, from=4-2, to=5-1]
            \arrow["{{\varphi_{A,\free\forget B}}}"', from=4-5, to=2-5]
            \arrow[""{name=6, anchor=center, inner sep=0}, "{\forget[\epsilon_A,\id]}"{description}, from=4-5, to=3-4]
            \arrow["{\forget[A,\epsilon_B]}"{description}, from=4-5, to=5-4]
            \arrow[""{name=7, anchor=center, inner sep=0}, "{{\forget\epsilon_{[A,B]}}}"', from=5-1, to=5-4]
            \arrow["{{\varphi_{A,B}}}"', from=5-4, to=5-6]
            \arrow[equals, from=5-6, to=1-6]
            \arrow["\equiv"{description}, draw=none, from=0, to=2-2]
            \arrow["{\genfrac{}{}{0pt}{}{\mathsf{nat}\ \varphi}{+\mathsf{adj}}}"{description}, draw=none, from=1-1, to=5]
            \arrow["{\mathsf{nat}\ \varphi^{-1}}"{description}, draw=none, from=1, to=4-2]
            \arrow["{\mathsf{nat}\ \epsilon}"{description,yshift=-4pt}, draw=none, from=2, to=4]
            \arrow["{\mathsf{nat}\ \epsilon}"{description}, draw=none, from=4, to=7]
            \arrow["{\genfrac{}{}{0pt}{}{\mathsf{nat}\ \varphi}{+\mathsf{adj}}}"{description}, draw=none, from=6, to=2-5]
            \arrow["{\forget\epsilon_A \star \forget\epsilon_B}", from=1-1, to=1-6, rounded corners,
            to path={
            (\tikztostart.north)
            -- ([yshift=.7cm]\tikztostart.center)
            -- ([yshift=.7cm]\tikztotarget.center)
            \tikztonodes
            -- (\tikztotarget.north)}]
        \end{tikzcd}}
    \]
    Moreover, the definition of \(\mathfrak{a}_{\tu}\) implies that \(\varphi_0\from \tu \to \forget \tu\) is a morphism of algebras:
    \[
        \act_{\forget\tu} \circ T \varphi_0
        = \forget \epsilon_{\tu} \circ T \varphi_0
        = \varphi_0 \circ \varphi_0^{-1} \circ \forget \epsilon_{\tu} \circ T \varphi_0
        = \varphi_0 \circ \act_{\tu}.
    \]
    Since \(\U^T\) is strict closed in this case, we can set \(\psi = \varphi\) and \(\psi_0=\varphi_0\), and \((\U^T,\id,\id)  \circ (K,\psi_0,\psi) = (\forget,\varphi_0,\varphi)\) holds as closed functors.

    In the opposite direction, if \((T,s,\act_\tu)\) is a gabi-monad and \((\psi_0,\psi)\) is a strong closed structure on \(K \colon \cat[D] \to \cat[C]^T\), then setting \(\varphi = \psi\) and \(\varphi_0 = \psi_0\) yields a strong closed structure on \(\forget = \U^T \circ K\) for which the equality holds as closed functors.

    It remains to verify that the two constructions are inverse to each other. It is self-evident that if we start from a strong closed structure \((\varphi_0,\varphi)\) on \(\forget\), construct \((T,s,\act_\tu)\) and \((\psi_0,\psi)\) as above, and then endow \(\forget\) with the strong closed structure given by the composite \(\U^T \circ K\), we recover the starting strong closed structure \((\varphi_0,\varphi)\).
    Conversely, suppose that we begin with \((T,s,\act_\tu)\) and \((\psi_0,\psi)\).
    Then, the new algebra structure on \(\tu_{\cat[C]}\) is given by
    \[
      T\tu_{\cat{C}} \xrightarrow{T\psi_0} T\forget\tu_{\cat{D}} \xrightarrow{\forget\epsilon_{\tu_{\cat{D}}}} \forget\tu_{\cat{D}} \xrightarrow{\psi_0^{-1}} \tu_{\cat{C}}
    \]
    which coincides with \(\act_\tu\) because \(\psi_0\) is a morphism of algebras \((\tu_{\cat[C]},\act_{\tu_{\cat[C]}}) \to (\forget\tu_{\cat[D]},\forget\epsilon_{\tu_{\cat[D]}})\). Similarly, the new natural transformation \(s\) is given by
    \[[\eta_X,TY] \circ \psi_{\free X,\free Y} \circ \forget\epsilon_{[\free X,\free Y]} \circ T\psi_{\free X,\free Y}^{-1} \circ T[TX,\eta_Y].\]
    Since \(\psi_{A,B} \colon (\forget[A,B]_{\cat[D]},\forget\epsilon_{[A,B]_{\cat[D]}}) \to [(\forget A,\forget\epsilon_A),(\forget B,\forget\epsilon_B)]_T\) is a morphism of \(T\)-algebras for every \(A,B\) in \(\cat[D]\),
    \[
      \act_{[TX,TY]} \circ T\psi_{\free X,\free Y} = \psi_{\free X,\free Y} \circ \forget\epsilon_{[\free X,\free Y]},
    \]
    where
    \[\act_{[TX,TY]} = [TX,\mu_Y] \circ s_{TX,TY} \circ T[\mu_X,TY].\]
    Therefore, the new natural transformation \(s\) is given by
    \begin{align*}
        [\eta_X,TY] \circ & \psi_{\free X,\free Y} \circ \forget\epsilon_{[\free X,\free Y]} \circ T\psi_{\free X,\free Y}^{-1} \circ T[TX,\eta_Y] = [\eta_X,TY] \circ \act_{[TX,TY]} \circ T[TX,\eta_Y] \\
        & = [\eta_X,TY] \circ [TX,\mu_Y] \circ s_{TX,TY} \circ T[\mu_X,TY] \circ T[TX,\eta_Y] = s_{X,Y},
    \end{align*}
    whence we recover the one we started with.
  \end{proof}

\subsection{Closed Monads}\label{sec:closed-monads}

\begin{definition}\label{def:closed-monad}
  An adjunction \(\free \colon \cC  \rightleftarrows \cD \colon \forget\) between skew-closed categories is \emph{closed} if \((\free,\varphi_0,\varphi)\) and \((\forget,\psi_0,\psi)\) are closed functors and \(\eta,\epsilon\) are closed natural transformations.

  A monad \((T, \mu, \eta) \from \cat{C}\to \cat{C}\) on a skew-closed category \(\cat{C}\) is called \emph{closed}
  if there exist \(\xi\from T[\blank, \bblank] \to [T\blank, T\bblank]\) and \(\xi_0\from \tu \to T\tu\)
  such that \((T, \xi_0,\xi)\) is a closed functor,
  and \(\mu\) and \(\eta\) are closed natural transformations.
\end{definition}

Clearly, the monad of a closed adjunction is a closed monad.
However, closed monads do not necessarily yield closed adjunctions.

\begin{remark}
  Spelt out explicitly, the conditions from \cref{def:closed_natural_transformation} on the multiplication and unit on the monad \(T\)
  in \cref{def:closed-monad} state that the following diagrams commute, for all \(X,Y \in \cat{C}\):
  \begin{equation}\label{eq:closed-monad-mu}
    \begin{gathered}
      \begin{tikzcd}[ampersand replacement=\&]
        {T^2[X,Y]} \& {T[X,Y]} \& {[TX,TY]} \\
        {T[TX,TY]} \& {[T^2X, T^2Y]} \& {[T^2X,TY]}
        \arrow["{\mu_{[X,Y]}}", from=1-1, to=1-2]
        \arrow["{T\xi_{X,Y}}"', from=1-1, to=2-1]
        \arrow["{\xi_{X,Y}}", from=1-2, to=1-3]
        \arrow["{[\mu_X,TY]}", from=1-3, to=2-3]
        \arrow["{\xi_{TX,TY}}"', from=2-1, to=2-2]
        \arrow["{[T^2 X, \mu_Y]}"', from=2-2, to=2-3]
      \end{tikzcd}
    \end{gathered}
    \qquad
    \begin{gathered}
      \begin{tikzcd}[ampersand replacement=\&,cramped]
{\tu} \arrow[d, "{\xi_0}"']\arrow[r, "{\xi_0}"]\& {T\tu} \arrow[d, "{T\xi_0}"]\\
        {T\tu} \& {T^2\tu} \arrow[l, "{\mu_\tu}"]
\end{tikzcd}
    \end{gathered}
  \end{equation}
  \begin{equation}\label{eq:closed-monad-eta}
    \begin{tikzcd}[ampersand replacement=\&]
      {[X,Y]} \& {T[X,Y]} \&\&\& \tu \\
      {[X,TY]} \& {[TX,TY]} \&\& \tu \& T\tu
      \arrow["{\eta_{[X,Y]}}", from=1-1, to=1-2]
      \arrow["{[X,\eta_Y]}"', from=1-1, to=2-1]
      \arrow["{\xi_{X,Y}}", from=1-2, to=2-2]
      \arrow[equals, from=2-4, to=1-5]
      \arrow["{\xi_0}"', from=2-4, to=2-5]
      \arrow["{[\eta_X,TY]}", from=2-2, to=2-1]
      \arrow["{\eta_\tu}", from=1-5, to=2-5]
    \end{tikzcd}
  \end{equation}
  Notice that the right-hand side of \eqref{eq:closed-monad-mu} is redundant, since it follows from the right-hand side of \eqref{eq:closed-monad-eta}.
\end{remark}

\begin{proposition}\label{prop:normal-closed-is-gabi}
  Let \(\cat{C}\) be a skew-closed category, with \((T, \xi_0,\xi)\) a closed monad on \(\cat{C}\).
  If \(\xi_0\) is an isomorphism, then \(T\) is a gabi-monad.
\end{proposition}
\begin{proof}
  Define the action of the unit to be the inverse of \(\xi_0\): \(\act_\tu\from T \tu \xrightarrow{\ \xi_0^{-1}\ } \tu\) and set
  \begin{equation}
    s_{X,Y}\from T[TX, Y] \xrightarrow{\xi_{TX,Y}} [T^2X, TY] \xrightarrow{[\eta_{TX}, TY]} [TX, TY] \xrightarrow{[\eta_X, TY]} [X, TY].
  \end{equation}
  By naturality of \(\eta\), this is the same as the composition
  \begin{equation}
    T[TX, Y] \xrightarrow{\xi_{TX,Y}} [T^2X, TY] \xrightarrow{[T\eta_{X}, TY]} [TX, TY] \xrightarrow{[\eta_X, TY]} [X, TY].
  \end{equation}
  We will often use this fact in the sequel.

  The fact that the internal hom of \(\cat{C}\) lifts as a functor---that is, that \cref{eq:lifting_homs} is satisfied---follows by the commutativity of the following two diagrams:
  \[
    \begin{tikzcd}[ampersand replacement=\&]
      {[TX, Y]} \& {T[TX,Y]} \& {[T^2X,TY]} \\
      {[X,TY]} \&\& {[TX,TY]}
      \arrow["{\eta_{[TX,Y]}}", from=1-1, to=1-2]
      \arrow["{[\eta_X,\eta_Y]}"', from=1-1, to=2-1]
      \arrow[""{name=0, anchor=center, inner sep=0}, "{[TX,\eta_Y]}"{description}, from=1-1, to=2-3]
      \arrow["{\xi_{TX,Y}}", from=1-2, to=1-3]
      \arrow["{[\eta_{TX}, TY]}", from=1-3, to=2-3]
      \arrow["{[\eta_X,TY]}", from=2-3, to=2-1]
      \arrow["\eqref{eq:closed-monad-eta}"{description}, draw=none, from=0, to=1-3]
      \arrow["\equiv"{description}, draw=none, from=2-1, to=0]
    \end{tikzcd}
  \]
  \[
    \mathscale{.8}{
      \begin{tikzcd}[ampersand replacement=\&]
        {T^2[TX,Y]} \& {T[TX,Y]} \& {[T^2X,TY]} \&\&\& {[TX,TY]} \\
        {T^2[T^2X,Y]} \& {T[T^2X,TY]} \& {[T^3X,TY]} \& {[T^2X,TY]} \& {[TX,TY]} \& {[X,TY]} \\
        {T[T^3X,TY]} \&\& {[T^3X,T^2Y]} \& {[T^2X,T^2Y]} \& {[TX,T^2Y]} \& {[X,T^2Y]} \\
        {T[T^2X,TY]} \& {T[TX,TY]} \& {[T^2X,T^2Y]} \&\&\& {[TX,T^2Y]}
        \arrow["{\mu_{[TX,Y]}}", from=1-1, to=1-2]
        \arrow["{T^2[\mu_X,Y]}"', from=1-1, to=2-1]
        \arrow[""{name=0, anchor=center, inner sep=0}, "{T\xi_{TX,Y}}"{description}, from=1-1, to=2-2]
        \arrow["{\xi_{TX,Y}}", from=1-2, to=1-3]
        \arrow["\eqref{eq:closed-monad-mu}"{description}, draw=none, from=1-2, to=2-2]
        \arrow[""{name=1, anchor=center, inner sep=0}, "{[\eta_{TX}, TY]}", from=1-3, to=1-6]
        \arrow["{[\mu_{TX},TY]}"', from=1-3, to=2-3]
        \arrow[""{name=2, anchor=center, inner sep=0}, curve={height=-10pt}, equals, from=1-3, to=2-4]
        \arrow["{[\eta_X,TY]}", from=1-6, to=2-6]
        \arrow["{T \xi_{T^2X, Y}}"', from=2-1, to=3-1]
        \arrow[""{name=3, anchor=center, inner sep=0}, "{T[T\mu_X,TY]}"{description}, from=2-2, to=3-1]
        \arrow["{\xi_{T^2X,TY}}"{description}, from=2-2, to=3-3]
        \arrow[""{name=4, anchor=center, inner sep=0}, curve={height=-20pt}, equals, from=2-2, to=4-1]
        \arrow["{[\eta_{T^2X},TY]}"', from=2-3, to=2-4]
        \arrow["{[\eta_{TX},TY]}"', from=2-4, to=2-5]
        \arrow["{[\eta_X,TY]}"', from=2-5, to=2-6]
        \arrow["\equiv"{description}, draw=none, from=2-5, to=3-4]
        \arrow[""{name=5, anchor=center, inner sep=0}, "{T[\eta_{T^2X},TY]}"', from=3-1, to=4-1]
        \arrow["{[T^3X,\mu_Y]}"{description}, from=3-3, to=2-3]
        \arrow["{[\eta_{T^2X},T^2Y]}"', from=3-3, to=3-4]
        \arrow["{[T\eta_{TX}, T^2Y]}"', from=3-3, to=4-3]
        \arrow[""{name=6, anchor=center, inner sep=0}, "{[\eta_{TX},T^2Y]}"', from=3-4, to=3-5]
        \arrow["{[\eta_X,T^2Y]}"', from=3-5, to=3-6]
        \arrow["{[X,\mu_Y]}"', from=3-6, to=2-6]
        \arrow["{T[\eta_{TX},TY]}"', from=4-1, to=4-2]
        \arrow["{\mathsf{nat}\ \xi}"{description}, draw=none, from=4-2, to=2-2]
        \arrow["{\xi_{TX,TY}}"', from=4-2, to=4-3]
        \arrow[""{name=7, anchor=center, inner sep=0}, "{[\eta_{TX},T^2Y]}"', from=4-3, to=4-6]
        \arrow["{[\eta_X, T^2Y]}"', from=4-6, to=3-6]
        \arrow["\equiv"{description}, draw=none, from=1, to=2-5]
        \arrow["{\mathsf{nat}\ T\xi}"{description}, draw=none, from=3, to=0]
        \arrow["{\mathsf{monad}}"{description}, draw=none, from=2-3, to=2]
        \arrow["{\mathsf{monad}}"{description}, draw=none, from=5, to=4]
        \arrow["\mathsf{nat}\ \eta"{description}, draw=none, from=6, to=7]
        \arrow["{Ts_{TX,Y}}"{description}, from=2-1, to=4-2, rounded corners, color=red,
        to path={
          (\tikztostart.west)
          -- ([xshift=-2.1cm]\tikztostart.center)
          -- ([xshift=-5.7cm,yshift=-.7cm]\tikztotarget.center)
          \tikztonodes
          -- ([yshift=-.7cm, xshift=-0.1cm]\tikztotarget.center)
          -- ([xshift=-0.1cm]\tikztotarget.south)}]
        \arrow["{s_{X,Y}}"{description}, from=4-2, to=3-6, rounded corners, color=red,
        to path={
          ([xshift=.1cm]\tikztostart.south)
          -- ([xshift=.1cm, yshift=-0.7cm]\tikztostart.center)
          -- ([xshift=1.4cm,yshift=-2.2cm]\tikztotarget.center)
          \tikztonodes
          -- ([xshift=1.4cm]\tikztotarget.center)
          -- (\tikztotarget.east)}]
        \arrow["{s_{X,Y}}"{description}, from=1-2, to=2-6, rounded corners, color=red,
        to path={
          (\tikztostart.north)
          -- ([yshift=0.7cm]\tikztostart.center)
          -- ([xshift=1.4cm,yshift=2.2cm]\tikztotarget.center)
          \tikztonodes
          -- ([xshift=1.4cm]\tikztotarget.center)
          -- (\tikztotarget.east)}]
      \end{tikzcd}}
  \]

  It is left to verify that the coherence maps also lift from \(\cat{C}\) to \(\cat{C}^T\).
  \Cref{eq:siotaj} follows by the commutativity of
  \[
    \begin{tikzcd}[ampersand replacement=\&]
      {T[\tu,X]} \&\&\& TX \\
      \& {[T\tu,TX]} \&\& {[\tu,TX]} \\
      \\
      {T[T\tu,X]} \& {[T^2\tu,TX]} \&\& {[T\tu,TX]}
      \arrow["{Ti_X}", from=1-1, to=1-4]
      \arrow["{\xi_{\tu,X}}"{description}, from=1-1, to=2-2]
      \arrow[""{name=0, anchor=center, inner sep=0}, "{T[\act_\tu,X]}"', shift right=2, curve={height=6pt}, from=1-1, to=4-1]
      \arrow["{T\ \mathsf{closed}}"{description}, draw=none, from=2-2, to=1-4]
      \arrow[""{name=1, anchor=center, inner sep=0}, "{[\xi_0,TX]}"{description}, from=2-2, to=2-4]
      \arrow["{i_{TX}}"', from=2-4, to=1-4]
      \arrow[""{name=2, anchor=center, inner sep=0}, "{T[\xi_0,X]}"', shift right=2, curve={height=6pt}, from=4-1, to=1-1]
      \arrow["{\xi_{T\tu,X}}"', from=4-1, to=4-2]
      \arrow["{[T\xi_0,TX]}"', from=4-2, to=2-2]
      \arrow[""{name=3, anchor=center, inner sep=0}, "{[\eta_{T\tu},TX]}"', from=4-2, to=4-4]
      \arrow["{[\eta_\tu,TX]}"', from=4-4, to=2-4]
      \arrow["{\mathsf{def}}"{description}, draw=none, from=0, to=2]
      \arrow["{\genfrac{}{}{0pt}{}{\xi_0=\eta_\tu}{\eta\ \mathsf{nat}}}"{description}, draw=none, from=1, to=3]
      \arrow["{\mathsf{nat}\ \xi}"{description}, draw=none, from=2, to=4-2]
      \arrow["{s_{\tu,Y}}"{description}, from=4-1, to=2-4, rounded corners, color=red,
      to path={
        (\tikztostart.south)
        -- ([yshift=-0.7cm]\tikztostart.center)
        -- ([xshift=1.4cm,yshift=-2.9cm]\tikztotarget.center)
        \tikztonodes
        -- ([xshift=1.4cm]\tikztotarget.center)
        -- (\tikztotarget.east)}]
    \end{tikzcd}
  \]
  and
  \[
    \begin{tikzcd}[ampersand replacement=\&,column sep=35pt]
      T\tu \& \tu \&\&\& {[X,X]} \\
      \&\&\& {[X,TX]} \\
      \&\&\&\& {[X,TX]} \\
      {T[TX,TX]} \& {[T^2X,T^2X]} \& {[TX,T^2X]} \& {[X,T^2X]}
      \arrow[""{name=0, anchor=center, inner sep=0}, "{\act_\tu}"{description}, shift left=2, curve={height=-6pt}, from=1-1, to=1-2]
      \arrow[""{name=1, anchor=center, inner sep=0}, "{Tj_{TX}}"', from=1-1, to=4-1]
      \arrow[""{name=2, anchor=center, inner sep=0}, "{\xi_0}"{description}, shift left=2, curve={height=-6pt}, from=1-2, to=1-1]
      \arrow[""{name=3, anchor=center, inner sep=0}, "{j_X}", from=1-2, to=1-5]
      \arrow[""{name=4, anchor=center, inner sep=0}, "{j_{T^2X}}", from=1-2, to=4-2]
      \arrow["{[X,\eta_X]}"{description}, from=1-5, to=2-4]
      \arrow[""{name=5, anchor=center, inner sep=0}, equals, from=2-4, to=3-5]
      \arrow["{[X,\eta_{TX}]}"', from=2-4, to=4-4]
      \arrow[""{name=6, anchor=center, inner sep=0}, "{[X, \act_X]}"', from=3-5, to=1-5]
      \arrow["{\xi_{TX,TX}}"', from=4-1, to=4-2]
      \arrow["{[\eta_{TX},T^2X]}"', from=4-2, to=4-3]
      \arrow["{[\eta_X,T^2X]}"', from=4-3, to=4-4]
      \arrow["{[X, \mu_X]}"', from=4-4, to=3-5]
      \arrow["{\mathsf{def}}"{description}, draw=none, from=0, to=2]
      \arrow["{T\ \mathsf{closed}}"{description}, draw=none, from=1, to=4]
      \arrow["{j\ \mathsf{dinat}}", draw=none, from=3, to=4-2]
      \arrow["{\mathsf{act}}"{description}, draw=none, from=2-4, to=6]
      \arrow["{\mathsf{monad}}"{description}, draw=none, from=4-4, to=5]
    \end{tikzcd}
  \]

  Finally, \cref{eq:sGamma} follows by the commutativity of \cref{fig:closed-monad-lift-gamma}.
\end{proof}

\begin{remark}\label{rem:rainbow}
  Suppose that \(T\) is a closed monad on a closed monoidal category \(\cat{C}\).
  Then by \cref{prop:closed-mon-bij}, \(T\) is a lax monoidal functor.
  In fact, \(T\) is a lax monoidal monad;
  that is, the unit and multiplication of the monad are transformations of lax monoidal functors.

  Suppose that \(\cat{C}\) admits coreflexive equalisers.
  Essentially by \cite{Kock-StrongFunctors}, structures of a lax monoidal monad on \(T\) are in bijection with structures of a \emph{commutative strong monad} on \(T\). Thus, in view of \cite{Kock-ClosedEM}, \(\cat[C]^T\) canonically becomes a closed category, \(\free^T\) and \(\U^T\) canonically become closed functors, and \(\eta\) and \(\epsilon\) canonically become closed natural transformations. Here, the \(\cat[C]\)-functor structure (strength) \(\mathsf{st}_{A,B} \colon [A,B] \to [TA,TB]\) allows us to define the internal hom as the equaliser
  \[\begin{tikzcd}[ampersand replacement=\&,column sep=40pt]
{\overline{[\boldsymbol{A},\boldsymbol{B}]}} \arrow[r, "{e_{A,B}}"]\& {[A,B]} \arrow[rr, shift left=+0.5ex, "{[\act_A,B]}"]\arrow[rr, shift right=0.5ex, "{[TA,\act_B] \,\circ\, \mathsf{st}_{A,B}}"']\&\& {[TA,B].}
\end{tikzcd}
  \]
  The morphisms \(e_{A,B}\) and \(\eta_\tu\) induce the closed structure on \(\U^T\).
  At the same time, since \(T\) is lax monoidal, the Kleisli category \(\cat{C}_T\) is canonically monoidal.
  Suppose that \(\cat{C}\) additionally admits reflexive coequalisers,
  which are preserved by \(T\) and the tensor product on \(\cat{C}\).
  Then \(\cat{C}^T\) is a monoidal category such that the unique embedding \(\iota\from \cat{C}_T \to \cat{C}^T\)
  and \(\free^T\) are strong monoidal, see~\cite[Theorem~4.14]{aguiar18:monad} and~\cite[Section~3.2]{stroinski2024:reconstr}.
  If \(\phi \colon T(\blank) \otimes T(\bblank) \to T(\blank \otimes \bblank)\) is the lax monoidal structure, then the tensor product of \(T\)-algebras \(\boldsymbol{A},\boldsymbol{B}\) is given by the Linton coequaliser
  \[
    \begin{tikzcd}[ampersand replacement=\&]
{T(TA \otimes TB)} \arrow[rrr, shift right=0.5ex, "{T(\act_A \otimes \act_B)}"']\arrow[rrr, shift left=+0.5ex, "{\mu_{A \otimes B} \, \circ \, T(\phi_{A,B})}"]\& \& \& {T(A \otimes B)} \arrow[r, "{q}"]\& {A \otimes ^{\phi} B.}
\end{tikzcd}
  \]
  As \(\free^T\) is strong, the adjunction \(\free^T \dashv \U^T\) is lax monoidal with structures \((\phi_0,\phi)\) on \(\free^T\) and \((\psi_0,\psi)\) on \(\U^T\).
  In particular, \(\xi_0 = \U^T\phi_0 \circ \psi_0\), and since \(\phi_0\) is invertible, \(\xi_0\) is an isomorphism if and only if \(\psi_0\) is.
\end{remark}

\begin{corollary}\label{cor:nice-closed-gabi}
  We keep the notation of \cref{rem:rainbow}. Let \((T,\xi_0,\xi)\) be a closed right exact monad on an abelian closed monoidal category \(\cat{C}\).
  Then \(T\) admits a gabi-monad structure for which \(\xi_0\) is a morphism of \(T\)-algebras if and only if
  \(\psi_0\from \tu \to U^T\tu\) is an isomorphism.

  Moreover, if the monoidal unit of \( \cat{C}\) is a  generator, \(T\) is cocontinuous, and \( \cat{C}\) is cocomplete, \(T\) being gabi implies \(T\cong \Id\).
\end{corollary}
\begin{proof}
  By \cref{rem:rainbow}, \(\psi_0\) is invertible if and only if \(\xi_0\) is.
  By \cref{prop:normal-closed-is-gabi}, see \eqref{eq:closed-monad-eta}, the gabi-monad structure satisfies \(\xi_0=\eta_\tu\),
  and hence \(\xi_0\) is a morphism of \(T\)-algebras.

  Conversely, suppose such a gabi-monad structure exists for which \(\xi_0\) is a morphism of algebras.
  Then we immediately have \(\act_\tu\circ\xi_0=\id_\tu\),
  and one furthermore calculates
  \[
    \xi_0\circ\act_\tu
    = \mu_\tu\circ T\xi_0
    \stackrel{\eqref{eq:closed-monad-eta}}{=} \mu_\tu\circ T\eta_\tu
    = \id_{T\tu}.
  \]

  Now, suppose \(\tu\) is a generator and fix \(X \in \cat{C}\).
  There exists an exact sequence
  \begin{equation*}
    P_{0}\xrightarrow{f} P_{1}\xrightarrow{g} X \xrightarrow{} 0,
  \end{equation*}
  with both \(P_{0}, P_{1}\in \cat{C}\) of the form \(P_{i} = \oplus_{N_{i}} \tu\).
  Now consider the diagram below:
  \[
    \begin{tikzcd}[ampersand replacement=\&,cramped]
      {P_0} \&\& {P_1} \&\& M \&\& 0 \&\& 0 \\
      \\
      {T(P_0)} \&\& {T(P_1)} \&\& {T(M)} \&\& 0 \&\& 0
      \arrow["f", from=1-1, to=1-3]
      \arrow["{\eta_{P_0}}"', from=1-1, to=3-1]
      \arrow["g", from=1-3, to=1-5]
      \arrow["{\eta_{P_1}}"', from=1-3, to=3-3]
      \arrow[from=1-5, to=1-7]
      \arrow["{\eta_M}"', from=1-5, to=3-5]
      \arrow[from=1-7, to=1-9]
      \arrow["{\eta_0}"', from=1-7, to=3-7]
      \arrow["{\eta_0}"', from=1-9, to=3-9]
      \arrow["{T(f)}"', from=3-1, to=3-3]
      \arrow["{T(g)}"', from=3-3, to=3-5]
      \arrow[from=3-5, to=3-7]
      \arrow[from=3-7, to=3-9]
    \end{tikzcd}
  \]
  Its rows are exact and \(\eta_{P_{i}} = \oplus_{i \in N_{i}} \eta_{\tu}\) are isomorphisms.
  Thus, the 5-lemma implies that \(\eta_{M}\from \Id(M) \to T(M)\) is an isomorphism, natural in \(M\).
\end{proof}

\begin{example}
    Let \(\cat{C}\) be a cartesian closed category and \(E\) an object of \(\cat{C}\). Then \(E\) is a cocommutative comonoid and hence \(\blank \times E\) is a symmetric colax monoidal comonad.
    Therefore, \([E,\blank]\) is a symmetric lax monoidal monad on \(\cat{C}\) (as the right adjoint of \(\blank \times E\)), often called in the literature the \emph{reader monad}, \emph{environment monad}, or \emph{function monad}, see for example \cite{wadler93:monads}.
    The unit at any \(X\) in \(\cat{C}\) is the unique morphism
    \[
      \eta_X \colon X \to [E,X]
    \]
    that is the transpose of the projection \(\pi_X \colon X \times E \to X\).
    The multiplication at any \(X\) in \(\cat{C}\) is the unique morphism
    \[
      \mu_X \colon [E,[E,X]] \to [E,X]
    \]
    that, under the internal tensor--hom adjunction, corresponds to \([E \times E,X] \xrightarrow{[\Delta_E,X]} [E,X]\),
    where \(\Delta_E\from E \to E \times E\) is the diagonal at \(E\).
    The lax monoidal structure is given by
    \[
      \phi_0 \colon \star \cong [E,\star],
      \qquad\text{and}\qquad
      \phi_{X,Y} \colon [E,X] \times [E,Y] \cong [E,X \times Y],
    \]
    because \([E,\blank]\) preserves limits.
    Then, the reader monad \([E,\blank]\) is a commutative monad in the sense of \cite{Kock-Monads}, by \cite[Theorem 2.3]{Kock-StrongFunctors}, with \(\cat{C}\)-functor structure given by
    \[\Gamma^E_{X,Y} \colon [X,Y] \to [[E,X],[E,Y]]\]
    for all \(X,Y\) in \(\cat{C}\).
    At the same time, since \([E,\blank]\) is lax monoidal, it is a closed monad with respect to the structure morphisms
    \[\xi_0 \colon \star \cong [E,\star]\]
    (because \([E,\blank]\) preserves limits) and \(\xi_{X,Y} \colon [E,[X,Y]] \to [[E,X],[E,Y]]\) given by
    \[
        \begin{tikzcd}[ampersand replacement=\&,cramped,column sep=80pt,row sep=40pt]
{[E,[X,Y]]} \arrow[d, "{\coev_{[E,[X,Y]]}^{[E,X]}}"']\arrow[r, "{\xi_{X,Y}}"]\& {[[E,X],[E,Y]]} \\
            {[[E,X],[E,[X,Y]]\times[E,X]]} \arrow[r, "{[[E,X],\phi_{[X,Y], X}]}"']\& {[[E,X],[E,[X,Y]\times X]]} \arrow[u, "{[[E,X],[E,\ev_Y^X]]}"']
\end{tikzcd}
    \]
    by \cref{prop:closed-mon-bij}.
    Thus, \([E,\blank]\) is a gabi-monad by \cref{prop:normal-closed-is-gabi} and the Eilenberg--Moore category \(\cat{C}^{[E,\blank]}\) of ``\(E\)-contramodules'' (see \cite[Chapter III, \S5]{Eilenberg-Moore}) is a skew-closed category in such a way that the forgetful functor \(\cat{C}^{[E,\blank]} \to \cat{C}\) is strict closed. In general, this skew-closed structure is not left normal.

    Suppose, for example, that \(\cat[C] = \Set\) with its cartesian closed structure.
    Then
    \begin{align*}
      \eta_X\from X \to [E,X], \ x \mapsto (e \mapsto x), \quad
      \mu_X\from [E,[E,X]]\to [E,X], \ \varphi \mapsto (e \mapsto \varphi(e)(e)) = \check{\varphi} \Delta_E,
    \end{align*}
    where \(\check{\varphi}(x,y) \defeq \varphi(x)(y)\).
    If \(E = \{0,1\}\), then it is well known that \(\Set^{[E,-]}\) is the category of \emph{rectangular bands}: semigroups \(S\) such that \(sts = s\) for every \(s,t\in S\).
    Indeed, an object in \(\Set^{[E,-]}\) is a set \(S\) equipped with a map \(m\from [E,S] \cong S \times S \to S\),
    such that \(m \circ \eta_S = \id_S\) and \(m \circ \mu_S = m \circ [E,m]\).
    Note that the unit and multiplication may be rewritten as
    \(s \mapsto (s,s)\) and \(((s,t),(x,y)) \mapsto (s,y)\), respectively.
    If we write \(m(s,t) \defeq st\), then the compatibility conditions become \(ss = s\) and \((st)(xy) = sy\).
    As \(m\) is associative by
    \[
      s(tx) = (ss)(tx) = sx = (st)(xx) = (st)x,
    \]
    these conditions are equivalent to \(S\) being a rectangular band; see for example~\cite{kimura58}.

    Rectangular bands form a skew-closed category in the obvious way: \(\term = \{*\}\) is trivially a rectangular band and \([S,S'] = \Set(S,S')\) has the multiplication \((f\cdot g)(s) = f(s)g(s)\) for all \(s \in S\), \(f,g \in \Set(S,S')\).
    With this structure on \(\Set(S,S')\),
    as well as the idempotency of the multiplication, so that e.g.\ \(\Set^{[E,\blank]}(\term,\Set(S,S')) = \Set(S,S')\),
    we have that
    \begin{align*}
      i_S & \colon \Set(\term,S) \to S, &  f & \mapsto f(*), \\
      j_S & \colon \term \to \Set(S,S), & * & \mapsto \id_S, \\
      \Gamma_{S',S''}^S & \colon \Set(S',S'') \to \Set(\Set(S,S'),\Set(S,S'')), & f & \mapsto (g \mapsto f \circ g),\\
      \widehat{\jmath}_{S,S'} & \colon \Set^{[E,\blank]}(S,S') \to \Set^{[E,\blank]}(\term,\Set(S,S')) = \Set(S,S'), & f & \mapsto f,
    \end{align*}
    are morphisms of semigroups.
    It is immediate that \(\widehat{\jmath}\) is not surjective;
    for any non-commutative rectangular band \(S\), take \(S' = S\op\),
    the rectangular band with the same underlying set and the opposite multiplication.\footnote{%
      A commutative rectangular band \(S\) is trivial: given \(s,t\in S\) we have \(s= sts =ts\) and so \(s=t\).

      Moreover, for any set \(X\) with at least 2 elements, we can form a non-commutative rectangular band by defining the multiplication \(x\cdot y \eqdef y\).%
    }
    Then the identity map is not a map of rectangular bands and cannot be in the image of \(\widehat{\jmath}\).

    Note that the closed structure on \(\cat[C]^{[E,\blank]}\) coming from the gabi-monad structure on \([E,\blank]\) does not coincide with the one presented in \cite{Kock-ClosedEM}, where the internal hom is realised as the equaliser
    \[\begin{tikzcd}[ampersand replacement=\&,column sep=40pt]
{\overline{[\boldsymbol{A},\boldsymbol{B}]}} \arrow[r]\& {[A,B]} \arrow[rr, shift left=+0.5ex, "{[\act_A,B]}"]\arrow[rr, shift right=0.5ex, "{[[E,A],\act_B] \, \circ \, \Gamma^E_{A,B}}"']\&\& {[[E,A],B].}
\end{tikzcd} \qedhere
    \]
\end{example}

The following result is a direct consequence of \cite[Theorem~2.2]{Kelly-doctrines}.

\begin{proposition}\label{prop:windybridge}
  Suppose that \((T,\mu,\eta,s,\act_\tu)\) is a gabi-monad on a skew-closed category \((\cat[C],[\blank,\bblank],\tu,\Gamma,i,j)\).
  Denote by \(\free^T \colon \cC  \rightleftarrows \cC^T \colon \U^T\) the Eilenberg–Moore adjunction.
  Then the following are equivalent:
  \begin{enumerate}[label=\textnormal{(\alph*)},leftmargin=0.8cm]
    \item\label{item:closed_adjunction_2}
    The following natural arrows are isomorphisms in \(\cat[C]^T\) and \(\cat[C]\), respectively:
    \begin{align*}
      \phi_0 &\colon \free^T\tu = (T\tu,\mu_\tu) \xrightarrow{=} \free^T\U^T(\tu,\act_\tu) \xrightarrow{\epsilon_{\tu} = \act_\tu} (\tu,\act_\tu) \\
      \phi_{X,M} &\colon \U^T[\free^TX,\mb{M}]_T \xrightarrow{=} [\U^T\free^TX,M] = [TX,M] \xrightarrow{[\eta_X,M]} [X,M]
    \end{align*}
    \item\label{item:closed_adjunction_1}
    There exist \(\psi_0 \colon (\tu,\act_\tu) \to \free^T(\tu)=(T(\tu),\mu_\tu)\) and
    \[
      \psi_{X,Y} \colon (T[X,Y],\mu_{[X,Y]}) = \free^T[X,Y] \to [\free^TX,\free^TY]_T = ([TX,TY],\act_{[TX,TY]})
    \]
    in \(\cat[C]^T\), the latter natural in \(X,Y\) in \(\cat[C]\), such that \((\free^T,\psi_0,\psi)\) is a closed functor from \(\cat[C]\) to \(\cat[C]^T\) and the adjunction \(\free^T \colon \cat[C]  \rightleftarrows \cat[C]^T \cocolon \U^T\) is closed.
  \end{enumerate}
\end{proposition}

\Cref{prop:when-free-is-closed} below offers an analogue of the interplay between opmonoidal monads and opmonoidal adjunctions, which govern reconstruction results for opmonoidal and Hopf monads, in the case of gabi-monads.

\begin{proposition}\label{prop:when-free-is-closed}
  Suppose that \((T,\mu,\eta,s,\act_\tu)\) is a gabi-monad on a skew-closed category \((\cat[C],[\blank,\bblank],\tu,\Gamma,i,j)\).
  Denote by \(\free^T \colon \cC  \rightleftarrows \cC^T \colon \U^T\) the Eilenberg–Moore adjunction.
  \begin{enumerate}[label=\textnormal{(\arabic*)},ref=(\arabic*),leftmargin=0.8cm]
    \item\label{item:FreeClosed1} If either of the equivalent conditions of \cref{prop:windybridge} holds,
    then there exist morphisms \(\xi_0 \colon \tu \to T\tu\) and \(\xi_{X,Y}\colon T[X,Y] \to [TX,TY]\) in \(\cat[C]^T\),
    the latter natural in \(X,Y\) in \(\cat[C]\), such that \((T,\xi_0,\xi)\) is a closed monad.
    \item\label{item:FreeClosed2} If \(T\) can be equipped with the structure morphisms \(\xi_0\) and \(\xi\) of a closed monad, which are morphisms in \(\cat{C}^T\), then \((\free^T,\xi_0,\xi)\) is a closed functor.
  \end{enumerate}
\end{proposition}

\begin{proof}
  \ref{item:FreeClosed1} Suppose that \ref{item:closed_adjunction_1} in \cref{prop:windybridge} holds.
  Then \(T\) becomes a closed functor via \(\xi_0 \colon \tu \xrightarrow{=} \U^T(\tu,\act_\tu) \xrightarrow{\U^T\psi_0} \U^T(T(\tu),\mu_\tu) = T(\tu)\) and
    \[
      \xi_{X,Y} \colon T[X,Y] = \U^T\free^T[X,Y] \xrightarrow{\U^T\psi_{X,Y}} \U^T[\free^TX,\free^TY]_T = [TX,TY],
    \]
    as a composition of closed functors (see \cite[Theorem~3.1, p.~434]{Eilenberg-Kelly}).
    It is clear that they are morphisms of \(T\)-algebras. Moreover, the adjunction \(\free^T \colon \cat[C]  \rightleftarrows \cat[C]^T \cocolon \U^T\) is closed if and only if the following diagrams commute
    \begin{gather*}
        \begin{gathered}
            \begin{tikzcd}[ampersand replacement=\&,cramped]
{\tu} \arrow[dr, draw=none, "{(i)}"{description}]\arrow[r, "{=}"]\arrow[d, "{=}"']\& {\tu} \arrow[d, "{\eta_\tu}"]\\
                {\U^T(\tu,\act_\tu)} \arrow[r, "{\U^T(\psi_0)}"']\& {\U^T\free^T(\tu)}
\end{tikzcd}
        \end{gathered}
        \qquad
        \begin{gathered}
            \begin{tikzcd}[ampersand replacement=\&,cramped]
{[X,Y]} \arrow[ddrr, draw=none, "{(ii)}"{description}]\arrow[dd, "{\eta_{[X,Y]}}"']\arrow[rr, "{=}"]\& \& {[X,Y]} \arrow[d, "{[X,\eta_Y]}"]\\
                 \& \& {[X,\U^T\free^TY]} \\
                {\U^T\free^T[X,Y]} \arrow[r, "{\U^T\psi_{X,Y}}"']\& {\U^T[\free^TX,\free^TY]_T} \arrow[r, "{=}"']\& {[\U^T\free^TX,\U^T\free^TY]} \arrow[u, "{[\eta_X,\U^T\free^TY]}"']
\end{tikzcd}
        \end{gathered} \\
        \begin{gathered}
            \begin{tikzcd}[ampersand replacement=\&,cramped]
{(T\tu,\mu_\tu)} \arrow[dr, draw=none, "{(iii)}"{description}]\arrow[r, "{=}"]\& {\free^T\U^T(\tu,\act_\tu)} \arrow[d, "{\epsilon_{\tu} = \act_\tu}"]\\
                {(\tu,\act_\tu)} \arrow[u, "{\psi_0}"]\arrow[r, "{=}"']\& {(\tu,\act_\tu)}
\end{tikzcd}
        \end{gathered}
        \qquad
        \begin{gathered}
            \begin{tikzcd}[ampersand replacement=\&,cramped]
{\free^T\U^T[\mb{M},\mb{N}]_T} \arrow[ddrr, draw=none, "{(iv)}"{description}]\arrow[dd, "{\epsilon_{[\mb{M},\mb{N}]_T}}"']\arrow[r, "{=}"]\& {\free^T[M,N]} \arrow[r, "{\psi_{M,N}}"]\& {[\free^TM,\free^TN]_T} \arrow[d, "{[\free^TM,\epsilon_{\mb{N}}]_T}"]\\
                 \& \& {[\free^TM,\mb{N}]_T} \\
                {[\mb{M},\mb{N}]_T} \arrow[rr, "{=}"']\& \& {[\mb{M},\mb{N}]_T} \arrow[u, "{[\epsilon_{\mb{M}},\mb{N}]_T}"']
\end{tikzcd}
        \end{gathered}
    \end{gather*}
    In fact, the commutativity of \((iii)\) follows from the commutativity of \((i)\) by applying \(\U^T\):
    if \((i)\) commutes, we have \(\U^T \psi_0 = \eta_{\tu}\), and so applying \(\U^T\) to \((iii)\) gives
    \[
      \U^T(\epsilon_\tu \circ \psi_0) = \U^T\epsilon_\tu \circ \U^T\psi_0 = \U^T\epsilon_\tu \circ \eta_{\tu} = \U^T\epsilon_\tu \circ \eta_{\U^T\tu} = \id_\tu,
    \]
    hence \((iii)\) commutes because \(\U^T\) is faithful.
    The commutativity of \((i)\) and \((ii)\) expresses the fact that \(\eta \colon \Id \to T\) is a closed natural transformation. At the same time, the commutativity of \((i)\) entails that \(\U^T\psi_0 = \eta_\tu\) and so the next square on the left-hand side commutes, while \((iv)\) can be used to show the commutativity of the right-hand side diagram:
    \[
    \begin{gathered}
            \begin{tikzcd}[ampersand replacement=\&,cramped]
{\tu} \arrow[d, "{\U^T\psi_0}"']\arrow[r, "{\U^T\psi_0}"]\& {T\tu} \arrow[d, "{T\U^T\psi_0}"]\arrow[dl, "{=}"']\\
                {T\tu} \& {T^2\tu} \arrow[l, "{\mu_\tu}"]\arrow[ul, draw=none, "{(i)}"{description}, pos=0.3]
\end{tikzcd}
        \end{gathered}
        \qquad
        \begin{gathered}
            \begin{tikzcd}[ampersand replacement=\&,cramped,column sep=40pt,row sep=40pt]
{T^2[X,Y]} \arrow[dd, "{\U^T\epsilon_{\free^T[X,Y]}}"']\arrow[r, "{T\U^T\psi_{X,Y}}"]\& {T[TX,TY]} \arrow[dr, draw=none, "{(iv)}"{description}]\arrow[d, "{\U^T\epsilon_{[\free^TX,\free^TY]_T}}"']\arrow[r, "{\U^T\psi_{TX,TY}}"]\& {[T^2X,T^2Y]} \arrow[d, "{[T^2X,\U^T\epsilon_{\free^TY}]}"]\\
                 \& {[TX,TY]} \arrow[r, "{\U^T[\epsilon_{\free^TX},\free^TY]_T}"']\& {[T^2X,TY]} \\
                {T[X,Y]} \arrow[urr, draw=none, "{\U^T \text{ closed}}"', pos=0.6]\arrow[uur, draw=none, "{\epsilon \text{ nat}}"]\arrow[ur, "{\U^T\psi_{X,Y}}"']\arrow[rr, "{\U^T\psi_{X,Y}}"']\& \& {[TX,TY]} \arrow[u, "{[\U^T\epsilon_{\free^TX},TY]}"']
\end{tikzcd}
        \end{gathered}
    \]
    that is, \(\mu\) is a closed natural transformation.

    \ref{item:FreeClosed2} In the opposite direction, suppose that
    \((T,\xi_0,\xi)\) is a closed structure on \(T\) such that \(\xi_0\) and \(\xi_{X,Y}\), for \(X,Y\) in \(\cat[C]\), are morphisms in \(\cat[C]^T\), and \(\eta\) and \(\mu\) are closed transformations.
    Since \(\U^T\) is strict closed, we can set \(\psi_0 \coloneqq \xi_0\) and \(\psi_{X,Y} \coloneqq \xi_{X,Y}\) for all \(X,Y\) in \(\cat[C]\).
    Then, all the diagrams
    \begin{gather*}
        \begin{tikzcd}[ampersand replacement=\&,cramped,column sep=50pt]
{(\tu,\act_\tu)} \arrow[r, "{\psi_0}"]\arrow[d, "{j^T_{\free^TX}}"']\& {\free^T \tu} \arrow[d, "{\free^T j_X}"]\\
            {{[\free^TX, \free^TX]_{T}}}
            \& {\free^T{[X, X]}}
            \arrow[l, "{\psi_{X,X}}"]
\end{tikzcd}
            \qquad
            \begin{tikzcd}[ampersand replacement=\&,cramped,column sep=50pt]
{\free^TX}
            \& {{[(\tu,\act_\tu), \free^TX]}_{T}} \arrow[l, "{i^T_{\free^TX}}"']\\
            {{{\free^T[\tu, X]}}} \arrow[r, "{\psi_{\tu, X}}"']\arrow[u, "{\free^T i_X}"]\&
            {{{[\free^T\tu, \free^TX]}_{T}}} \arrow[u, "{{[\psi_0, \free^TX]}_{T}}"']
\end{tikzcd}
            \\
            \begin{tikzcd}[ampersand replacement=\&,cramped,column sep=37.5pt]
{{{\free^T[X, Y]}}} \arrow[r, "{\free^T\Gamma^Z_{X, Y}}"]\arrow[d, "{\psi_{X,Y}}"']\&
            {{\free^T[[Z, X], [Z, Y]]}} \arrow[r, "{\psi_{[Z, X], [Z,Y]}}"]\&
            {{[\free^T[Z, X], \free^T[Z, Y]]_{T}}} \arrow[d, "{[\id, \psi_{Z, Y}]_T}"]\\
            {{[\free^TX, \free^TY]_{T}}} \arrow[r, "{\Gamma^{\free^TZ}_{\free^TX, \free^TY}}"']\&
            {{[[\free^TZ, \free^TX]_{T}, [\free^TZ, \free^TY]_{T}]_{T}}} \arrow[r, "{[\psi_{Z, X}, \id]_T}"']\&
            {{[\free^T[Z, X], [\free^TZ, \free^TY]_{T}]_{T}}}
\end{tikzcd}
    \end{gather*}
    commute because they commute after applying \(\U^T\): they become the diagrams expressing the fact that \((T,\xi_0,\xi)\) is a closed functor. Thus, \((\free^T,\psi_0,\psi)\) is a closed functor.
\end{proof}

\section{Normal gabi-monads and left Hopf monads}\label{sec:normal-gabi}

\Cref{ex:Hopf-monads-are-normal-gabi} showed that a left Hopf monad on a closed monoidal category is automatically a normal gabi-monad.
In this section, we concern ourselves with studying the converse of that observation.
To that end, we first classify lifts of the tensor--hom adjunction from a closed monoidal base category \( \cat{C} \) to the Eilenberg--Moore category \( \cat{C}^{T}\) of a monad \(T\) in  terms of maps that ``intertwine'' monads with (co)monads, see \cref{thm:all-liftings-data}.
This allows us to derive, in \cref{thm:lift-to-hopf-monads}, necessary and sufficient criteria for a gabi-monad \(T\) to be a left Hopf monad in terms of the invertibility of an arrow associated to \(s \from T[T \blank , \bblank] \to [ \blank , T \bblank]\).
In \cref{cor:gabi-in-the-com-case}, we apply this result to derive a shortened proof that a normal gabi-algebra over a commutative base ring is Hopf.
Examples of normal gabi-monads that are not Hopf are then discussed in~\cref{sec:gabi-not-hopf}.

\subsection{Lifting functors}\label{sec:lifting-functors}

Throughout this section, we fix monads \(R\), \(S\), and \(T\) on the categories \( \cat{C}\), \(\cat{D}\), and \( \cat{E}\).
\medskip

First, we recall the concept of ``extending'' functors and natural transformations to  Eilenberg--Moore categories following~\cite{Beck,Street-formal_theory}.

\begin{definition}
  A \emph{lift}  of \(F \from \cat{C} \to \cat{D}\) is a functor \(\widehat{F}\from \cat{C}^R \to \cat{D}^S\)
  satisfying \(\U^S \widehat{F} = F \U^R\).
\end{definition}

In order to study lifts of the tensor-hom adjunction, we also need to lift arrows between functors to Eilenberg--Moore categories.

\begin{definition}\label{def:morphisms-of-lifts}
  Let \(F, H \from \cat{C} \to \cat{D}\) be functors with lifts \(\widehat{F}, \widehat{H}\from \cat{C}^{R} \to \cat{D}^{S}\) and \(\alpha \from F\to H\) a natural transformation.
  A \emph{lift} of \(\alpha\) is a natural transformation  \(\widehat{\alpha} \from \widehat{F}\to \widehat{H}\) satisfying
  \begin{equation*}
    \U^{S}(\widehat{\alpha}_{\boldsymbol{M}})= \alpha_{\U^{R}(\boldsymbol{M})} \qquad \text{ for all } \boldsymbol{M} \in \cat{C}^{R}. \qedhere
  \end{equation*}
\end{definition}

Suppose \(\boldsymbol{M}\in \cat{C}^{R}\) is an \(R\)-algebra.
In order to  lift \(F \from \cat{C} \to \cat{D}\) to a functor \(\widehat{F}\from \cat{C}^{R} \to \cat{D}^{S}\), we need to impose an \(S\)-algebra structure on \(F M\).
One possible solution is to demand the existence of a suitably compatible map \(\chi \from SF \to F R\) that allows us to pull back the action.
In the definition below, we follow the terminology of~\cite[\S3.5]{BruguieresLackVirelizier}.

\begin{definition}\label{def:distributive-law}
  A \emph{lifting datum} for \(F\) along \(R\) and \(S\) is a natural arrow \(\chi \from SF \to F R\) rendering the below diagrams commutative:
    \begin{equation}\label{eq:monadic_distibutive_law}
        \begin{gathered}
            \begin{tikzcd}[ampersand replacement=\&,cramped]
\& {F} \arrow[dl, "{\eta^SF}"']\arrow[dr, "{F\eta^R}"]\& \\
                 {SF} \arrow[rr, "{\chi}"']\& \& {F R}
\end{tikzcd}
        \end{gathered}
        \qquad
        \qquad
        \begin{gathered}
            \begin{tikzcd}[ampersand replacement=\&,cramped]
{SSF} \arrow[r, "{S \chi}"]\arrow[d, "{\mu^S F}"']\& {SF R} \arrow[r, "{\chi R}"]\& {F RR} \arrow[d, "{F\mu^R}"]\\
                {SF} \arrow[rr, "{\chi}"']\& \& {F R}
\end{tikzcd}
        \end{gathered}\qedhere
    \end{equation}
\end{definition}

In \cite[\S1]{Street-formal_theory}, the pair \((F,\chi)\) is also called a \emph{monad functor}.
When \(F\) is an endofunctor and \(S=T\), \(\chi\) is also referred to as a \emph{distributive law} for the monad \(T\) over the endofunctor \(F\), see~\cite[\S1]{Beck}.
\smallskip

The next result is discussed for example in~\cite{Street-formal_theory}.
\begin{lemma}\label{lem:lifting-co-and-co}
  Let \(H \from \cat{C} \to \cat{D}\) be a functor.
  There is a bijection
  \begin{equation}\label{eq:lifting-condition}
    \begin{aligned}
      \left\{\text{lifts of }H\text{ to a functor }\cat{C}^R \to \cat{D}^S\right\} &\to \left\{\text{lifting data } \chi\from SH\to HR\right\}, \\
      \widehat{H} &\mapsto \chi, \\
      \big((M, \act_{M}) \mapsto (HM, H\act_{M} \circ \chi_{M})\big) &\mapsfrom \chi,
    \end{aligned}
  \end{equation}
  where \(\chi_C\) is given for all \(C \in \cat{C}\) by
  \[
    SHC \xrightarrow{SH \eta^R_C} SHRC = \U^S\free^S\U^S\widehat{H}\free^RC \xrightarrow{\U^S \epsilon^S \widehat{H} \free^R C} \U^S \widehat{H} \free^R C = H R C.
  \]

  Further, a natural transformation \(\alpha \colon H \to H'\) lifts to a natural transformation \(\widehat{\alpha} \colon \widehat{H} \to \widehat{H'}\) if and only if it makes the following diagram commute:
  \begin{equation}\label{eq:lifting-alpha}
  \begin{gathered}
    \begin{tikzcd}[ampersand replacement=\&,cramped]
{SH} \arrow[r, "{\chi}"]\arrow[d, "{S\alpha}"']\& {HR} \arrow[d, "{\alpha R}"]\\
        {SH'} \arrow[r, "{\chi'}"']\& {H'R}
\end{tikzcd}
  \end{gathered}
  \end{equation}
\end{lemma}

Due to the contravariance of the internal-hom \([ \blank , \bblank]_{\cat{C}}\) of a (skew) closed category \( \cat{C}\) in its first variable, its lifts are parameterised by a slightly modified notion of lifting data.

\begin{definition}[{\cite[Definition~A.1]{Berger-Vercruysse-Saracco}}]\label{def:comonad-monad-morphism}
  A \emph{mixed lifting datum} for \(\G \from \cat{D}\op\times \cat{E}\to \cat{C}\) along \((S\op,T)\) and \(R\) consists of a natural transformation \(\nu \colon R\G(S\blank,\bblank) \to \G(\blank,T\bblank)\)
  such that the following diagrams
  \begin{gather}
    \label{eq:comonad-monad-morphism-1}\begin{tikzcd}[ampersand replacement=\&,cramped,sep=small]
      \& {R\G(SD,E)} \\
      {\G(SD,E)} \\
      \& {\G(D,TE)}
      \arrow["{\nu_{D,E}}", from=1-2, to=3-2]
      \arrow["{\eta^R_{\G(SD,E)}}", from=2-1, to=1-2]
      \arrow["{\G(\eta^S_D, \eta^T_E)}"', from=2-1, to=3-2]
    \end{tikzcd}
    \\
   \label{eq:comonad-monad-morphism-2}\begin{tikzcd}[ampersand replacement=\&,column sep=3.2em]
    	{R^2\G(SD,E)} \& {R^2\G(S^2D,E)} \& {R\G(SD,TE)} \& {\G(D,T^2E)} \\
    	{R\G(SD,E)} \&\&\& {\G(D,TE)}
    	\arrow["{{R^2\G(\mu^S_D,E)}}", from=1-1, to=1-2]
    	\arrow["{{\mu^R_{\G(SD,E)}}}"', from=1-1, to=2-1]
    	\arrow["{{R\nu_{SD,E}}}", from=1-2, to=1-3]
    	\arrow["{{\nu_{D,TE}}}", from=1-3, to=1-4]
    	\arrow["{{\G(D,\mu^T_E)}}", from=1-4, to=2-4]
    	\arrow["{{\nu_{D,E}}}"', from=2-1, to=2-4]
    \end{tikzcd}
  \end{gather}
  commute for all \(D\) in \(\cat[D]\) and \(E\) in \(\cat[E]\).
\end{definition}

We can think of a mixed lifting datum as a lifting datum that intertwines a comonad and monad structure on the domain with a monad structure on its codomain.
This differs from the \emph{mixed distributive laws} discussed, for example, in~\cite[\S5.3]{Wisbauer}, which, for a monad \(P\) and a comonad \(Q\) on the same category, are natural transformations \(\chi \from PQ \to QP\) such that \((Q,\chi)\) is a monad functor and \((P, \chi)\) is a comonad functor.

\begin{remark}
  Given an endofunctor \(R\) on a category \(\cat{C}\), let us denote its opposite endofunctor on \(\cat{C}\op\) by \(R\op\).
  Then \(R\) is a monad on \(\cat[C]\) if and only if \(R\op\) is a comonad on \(\cat[C]\op\).
  The Eilenberg--Moore category \((\cat[C]\op)^{R\op}\) of coalgebras for the comonad \(R\op\) is canonically isomorphic to the opposite \({(\cat[C]^R)}\op\) of the category of algebras for the monad \(R\).
\end{remark}

The next result is essentially~\cite[Theorem~A.2]{Berger-Vercruysse-Saracco}.

\begin{proposition}\label{prop:lifting-co-and-contra}
  Let \(\G \from \cat[D]\op\times\cat[E] \to\cat[C]\) be a functor.
  \begin{enumerate}[label=\textnormal{(\arabic*)},leftmargin=0.8cm]
    \item Any mixed lifting datum \(\nu \colon R\G(S\blank,\bblank) \to \G(\blank,T\bblank)\) induces a lift  of \(\G\)
    \begin{equation*}
      \widehat{\G} \from {\left({\cat[D]\op}\right)^{S\op}} \times \cat[E]^{T} \to\cat[C]^{R}, \quad \big((M, \act^{S}_{M}), (N, \act^{T}_{N}) \big) \mapsto (\G(M,N), \act_{\G(M,N)}),
    \end{equation*}
   where \(\act_{\G(M,N)}\) is given by
    \begin{equation}\label{eq:lift-via-mixed-lifting-datum}
      \begin{tikzcd}[ampersand replacement=\&,cramped]
        {R\G(M,N)} \&\& {R\G(SM,N)} \&\& {\G(M,TN)} \&\& {\G(M,N).}
        \arrow["{R\G(\act^{S}_M,N)}", from=1-1, to=1-3]
        \arrow["{\nu_{M,N}}", from=1-3, to=1-5]
        \arrow["{\G(M,\act^{T}_N)}", from=1-5, to=1-7]
      \end{tikzcd}
    \end{equation}
    \item Conversely, any lift \(
    \widehat{\G} \from {\left({\cat[D]\op}\right)^{S\op}} \times \cat[E]^{T} \to\cat[C]^{R}\) of \(\G\) gives rise to a mixed lifting datum \(\nu \from R\G(S \blank ,\bblank)\to \G(\blank, T\bblank)\) via
    \begin{equation}
      \begin{tikzcd}[ampersand replacement=\&,cramped, column sep=small]
        {R\G(SM,N)} \&\&\& {R\G(SM,TN)} \&\&\& {\G(SM,TN)} \&\& {\G(M,TN).}
        \arrow["{R\G(SM,\eta^T_N)}", from=1-1, to=1-4]
        \arrow["{\act^R_{\G(SM,TN)}}", from=1-4, to=1-7]
        \arrow["{\G(\eta^S_M,TN)}", from=1-7, to=1-9]
      \end{tikzcd}
    \end{equation}
  \end{enumerate}
  This yields\! a bijection between mixed lifting data \(\nu\colon\!R\G(S\blank,\bblank) \to \G(\blank,T\bblank)\)
  and lifts of \(\G\).
\end{proposition}

The following definition is inspired by \cite[Remark~3.3]{Berger-Vercruysse-Saracco}.
\begin{definition}\label{def:param-lift-dat}
  Consider a functor \(\G \from \cat[D]\op\times\cat[E] \to\cat[C]\).
  A \emph{parametric lifting datum} consists of natural transformations
  \[
    {\{\,\chi^{\boldsymbol{M}}_E \colon R\G(M,E) \to \G(M,TE)\,\}}_{\boldsymbol{M} \, = \, (M,\act_M) \, \in \, \cat{D}^S, E \, \in \, \cat{E}}
  \]
  such that
  \begin{equation}\label{eq:param-lift}
    \begin{tikzcd}[ampersand replacement=\&]
      {\G(M,E)} \& {R\G(M,E)} \\
      \& {\G(M,TE)}
      \arrow["{\eta^R_{\G(M,E)}}", from=1-1, to=1-2]
      \arrow["{\G(M,\eta^T_E)}"', from=1-1, to=2-2]
      \arrow["{\chi^{\boldsymbol{M}}_E}", from=1-2, to=2-2]
    \end{tikzcd} \quad
    \begin{tikzcd}[ampersand replacement=\&]
      {R^2 \G(M,E)} \&\& {R\G(M,E)} \\
      {R\G(M,TE)} \& {\G(M,T^2E)} \& {\G(M,TE)}
      \arrow["{\mu^R_{\G(M,E)}}", from=1-1, to=1-3]
      \arrow["{R\chi^{\boldsymbol{M}}_E}"', from=1-1, to=2-1]
      \arrow["{\chi^{\boldsymbol{M}}_E}", from=1-3, to=2-3]
      \arrow["{\chi^{\boldsymbol{M}}_{TE}}"', from=2-1, to=2-2]
      \arrow["{\G(M,\mu^T_E)}"', from=2-2, to=2-3]
    \end{tikzcd}
  \end{equation}
  commute for all \(\boldsymbol{M} \in \cat{D}^S\) and \(E \in \cat{E}\).
\end{definition}

Any functor \(\cat[D]\op \times \cat[E] \to \cat[C]\) corresponds to a functor \(\cat[D]\op \to \big[\cat[E],\cat[C]\big]\), and therefore any lifting to a functor \({(\cat[D]^S)}\op \times \cat[E]^T \to \cat[C]^R\) can be uniquely identified with a functorial parametric lifting \({(\cat[D]^S)}\op \to \left[\cat[E]^T,\cat[C]^R\right]\).
The next proposition formalises this fact, along the lines of \cite[Remarks 3.3 and 3.5]{Berger-Vercruysse-Saracco}.

\begin{lemma}\label{lem:mixed-lifting-param-lifting}
    Let \(\G \from \cat[D]\op\times\cat[E] \to\cat[C]\) be a functor.
    \begin{enumerate}[label=\textnormal{(\arabic*)},leftmargin=0.8cm]
        \item Any mixed lifting datum \(\nu \colon R\G(S\blank,\bblank) \to \G(\blank,T\bblank)\) induces the (natural in \(\boldsymbol{M}\)) parametric lifting datum
        \begin{equation}\label{eq:parametric-lifting-datum-via-mixed-lifting-datum}
            \chi^{\boldsymbol{M}} \colon R\G(M,\blank) \xrightarrow{R\G(\act_M,\blank)} R\G(SM,\blank) \xrightarrow{\nu_{M,\blank}} \G(M,T\blank).
        \end{equation}

        \item Conversely, any parametric lifting datum \(\chi^{\boldsymbol{M}} \colon R\G(M,\blank) \to \G(M,T\blank)\) natural in \(\boldsymbol{M} \in \cat[D]^S\) gives rise to the mixed lifting datum
        \begin{equation}
          \nu_{D,E} \colon R\G(SD,E) \xrightarrow{\chi^{\free^S(D)}_E} \G(SD,TE) \xrightarrow{\G(\eta^S_D,TE)} \G(D,TE).
        \end{equation}

    \end{enumerate}
    These constructions yield a bijection between mixed lifting data \(\nu \colon R\G(S\blank,\bblank) \to \G(\blank,T\bblank)\)
    and natural parametric lifting data \(\chi^{\boldsymbol{M}}\).
\end{lemma}

\begin{proof}
  It is not hard to see that the two correspondences do assign to any mixed lifting datum, a natural parametric lifting datum and \emph{vice versa},
  that \eqref{eq:param-lift} is equivalent to \eqref{eq:comonad-monad-morphism-1} and \eqref{eq:comonad-monad-morphism-2},
  and that they are inverse to each other.
  For example, the equivalence of the second diagram in \eqref{eq:param-lift} with \eqref{eq:comonad-monad-morphism-2} follows by the following two diagrams:
  \[
    \begin{tikzcd}[ampersand replacement=\&,cramped]
      \& {R^2\G(SM,E)} \&\& {R\G(M,TE)} \\
      {R^2\G(M,E)} \&\& {R^2\G(S^2M,E)} \\
      \& {R^2\G(SM,E)} \&\& {R\G(SM,TE)} \\
      {R\G(M,E)} \& {R\G(SM,E)} \& {\G(M,TE)} \& {\G(M,T^2E)}
      \arrow["{R\nu_{M,E}}"{description}, from=1-2, to=1-4]
      \arrow["{R^2\G(S\act^S_M,E)}"{description}, from=1-2, to=2-3]
      \arrow["{R\G(\act^S_M,TE)}"{description}, from=1-4, to=3-4]
      \arrow["{R^2\G(\act^S_M,E)}"{description}, from=2-1, to=1-2]
      \arrow["{R^2\G(\act^S_M,E)}"{description}, from=2-1, to=3-2]
      \arrow["{\mu^R_{\G(M,E)}}"', from=2-1, to=4-1]
      \arrow["{R\nu_{SM,E}}", from=2-3, to=3-4]
      \arrow["{R^2\G(\mu^S_M,E)}"{description}, from=3-2, to=2-3]
      \arrow["{\mu^R_{\G(SM,E)}}"', from=3-2, to=4-2]
      \arrow["{\nu_{M,TE}}", from=3-4, to=4-4]
      \arrow["{R\G(\act^S_M,E)}"', from=4-1, to=4-2]
      \arrow["{\nu_{M,E}}"', from=4-2, to=4-3]
      \arrow["{\G(M,\mu^T_E)}", from=4-4, to=4-3]
      \arrow["\mathsf{alg}"{description}, draw=none, from=2-1, to=2-3]
      \arrow["\mathsf{nat}\ \nu"{description}, draw=none, from=2-3, to=1-4]
      \arrow["\mathsf{nat}\ \mu", draw=none, from=4-1, to=3-2]
      \arrow["{\eqref{eq:comonad-monad-morphism-2}}", draw=none, from=2-3, to=4-3]
      \arrow["{R\chi^{\boldsymbol{M}}_E}", from=2-1, to=1-4, rounded corners,
      to path={-- ([yshift=2.5cm]\tikztostart.center) -- ([yshift=1cm]\tikztotarget.center) \tikztonodes -- (\tikztotarget.north)}]
      \arrow["{\chi^{\boldsymbol{M}}_{TE}}", from=1-4, to=4-4, rounded corners,
      to path={-- ([xshift=2cm]\tikztostart.center) -- ([xshift=2cm]\tikztotarget.center) \tikztonodes -- (\tikztotarget.east)}]
      \arrow["{\chi^{\boldsymbol{M}}_{E}}", from=4-1, to=4-3, rounded corners,
      to path={-- ([yshift=-1cm]\tikztostart.center) -- ([yshift=-1cm]\tikztotarget.center) \tikztonodes -- (\tikztotarget.south)}]
    \end{tikzcd}
  \]
  \[
    \begin{tikzcd}[ampersand replacement=\&,cramped,column sep=24pt,row sep=30pt]
      {R^2\G(S^2D,E)} \& {R\G(S^2D,TE)} \&\&\& \\
      {R^2\G(SD,E)} \&\& {R\G(SD,TE)} \& {\G(SD,T^2E)} \& {\G(D,T^2E)} \\
      {R\G(SD,E)} \&\&\& {\G(SD,TE)} \& {\G(D,TE)}
      \arrow["{R\chi^{\free^S(SD)}_E}", from=1-1, to=1-2]
      \arrow["{R\G(\eta^S_{SD},TE)}"{description}, from=1-2, to=2-3]
      \arrow["{R^2\G(\mu^S_D,E)}"{description}, from=2-1, to=1-1]
      \arrow["{R\chi^{\free^S(D)}_E}"{description}, from=2-1, to=2-3]
      \arrow["{\mu^R_{\G(SD,E)}}"{description}, from=2-1, to=3-1]
      \arrow["{\chi^{\free^S(D)}_{TE}}", from=2-3, to=2-4]
      \arrow["{\G(\eta^S_D,T^2E)}", from=2-4, to=2-5]
      \arrow["{\G(SD,\mu^T_E)}"{description}, from=2-4, to=3-4]
      \arrow["{\G(D,\mu^T_E)}"{description}, from=2-5, to=3-5]
      \arrow["{\chi^{\free^S(D)}_E}"', from=3-1, to=3-4]
      \arrow["{\G(\eta^S_D,TE)}"', from=3-4, to=3-5]
      \arrow["\mathsf{nat}\ \chi"{description}, draw=none, from=2-1, to=1-2]
      \arrow["{\eqref{eq:param-lift}}"{description}, draw=none, from=2-1, to=3-4]
      \arrow["\equiv"{description}, draw=none, from=3-4, to=2-5]
      \arrow["{R\nu_{SD,E}}", from=1-1, to=2-3, rounded corners,
      to path={-- ([yshift=1cm]\tikztostart.center) -- ([xshift=-0.1cm,yshift=2.75cm]\tikztotarget.center) \tikztonodes -- ([xshift=-0.1cm]\tikztotarget.north)}]
      \arrow["{\nu_{D,TE}}", from=2-3, to=2-5, rounded corners,
      to path={-- ([yshift=1cm]\tikztostart.center) -- ([yshift=1cm]\tikztotarget.center) \tikztonodes -- (\tikztotarget.north)}]
      \arrow["{\nu_{D,E}}"', from=3-1, to=3-5, rounded corners,
      to path={-- ([yshift=-1cm]\tikztostart.center) -- ([yshift=-1cm]\tikztotarget.center) \tikztonodes -- (\tikztotarget.south)}]
    \end{tikzcd}
  \]
\end{proof}

Summing up, we have the following result.

\begin{theorem}\label{thm:all-liftings-data}
  Let \(R \in \End( \cat[C])\), \( S \in \End( \cat[D])\), and \(T \in \End(\cat[E])\) be monads, and suppose that \(\G \from \cat[D]\op\times\cat[E] \to\cat[C]\) is a functor.
  The assignments of \cref{prop:lifting-co-and-contra,lem:mixed-lifting-param-lifting} establish a bijective correspondence between
  \begin{enumerate}[label=\textnormal{(\arabic*)},leftmargin=0.8cm]
    \item lifts \(\widehat{\G} \from {\left(\cat[D]^S\right)\op} \times \cat[E]^{T} \to \cat[C]^{R}\) of \(\G\);
    \item mixed lifting data \(\nu \colon R\G(S\blank,\bblank) \to \G(\blank,T\bblank)\);
    \item parametric lifting data \(\left\{\chi^{\boldsymbol{M}} \colon R\G(M,\blank) \to \G(M,T\blank)\right\}\) natural in \(\boldsymbol{M}\) in \(\cat[D]^S\).
  \end{enumerate}

  Further, \(\alpha \colon \G(\blank,\bblank) \to \G'(\blank,\bblank)\) lifts to a natural transformation \(\widehat{\alpha} \colon \widehat{\G}(\blank,\bblank) \to \widehat{\G'}(\blank,\bblank)\)
  if and only if \(\alpha_{M,\blank} \colon \G(M,\blank) \to \G'(M,\blank)\) satisfies \eqref{eq:lifting-alpha} for every \(\boldsymbol{M}\) in \(\cat[D]^S\).
  Put differently, one requires commutativity of either of the following diagrams:
  \[
    \begin{gathered}
      \begin{tikzcd}[ampersand replacement=\&,column sep=large,row sep=large]
        {R\G(SD,E)} \& {\G(D,TE)} \\
        {R\G'(SD,E)} \& {\G'(D,TE)}
        \arrow["{\nu_{D,E}}", from=1-1, to=1-2]
        \arrow["{R\alpha_{SD,E}}"', from=1-1, to=2-1]
        \arrow["{\alpha_{D,TE}}", from=1-2, to=2-2]
        \arrow["{\nu'_{D,E}}"', from=2-1, to=2-2]
      \end{tikzcd}
    \end{gathered}
    \qquad\iff\qquad
    \begin{gathered}
      \begin{tikzcd}[ampersand replacement=\&,column sep=large,row sep=large]
        {R\G(M,E)} \& {\G(M,TE)} \\
        {R\G'(M,E)} \& {\G'(M,TE)}
        \arrow["{\chi^{\boldsymbol{M}}_E}", from=1-1, to=1-2]
        \arrow["{R\alpha_{M,E}}"', from=1-1, to=2-1]
        \arrow["{\alpha_{M,TE}}", from=1-2, to=2-2]
        \arrow["{{\chi'}^{\boldsymbol{M}}_E}"', from=2-1, to=2-2]
      \end{tikzcd}
    \end{gathered}
  \]
\end{theorem}

\subsection{Lifting parametric adjunctions}\label{sec:lifting-parametric-adjunctions}

Our goal is now to classify lifts of adjunctions of the form \(\cat{E}(\free(M,N),P) \cong \cat{C}(M,\forget(N,P))\) to Eilenberg--Moore categories.
This requires us to first study lifts of ordinary adjunctions.

\begin{definition}\label{def:lift-of-adjunction}
  Let \(S\) and \(T\) be monads on \( \cat{C}\) and \(\cat{D}\), respectively and let \(\adj{\free}{\forget}{\cat{C}}{\cat{D}}\) be an adjunction with unit \(\eta \from \Id_{ \cat{C}} \to \forget\free\) and counit \(\epsilon\from \free\forget \to \Id_{ \cat{D}}\).
  A \emph{lift of the adjunction} is a quadruple \((\widehat{\free}, \widehat{\forget}, \widehat{\eta}, \widehat{\epsilon})\) comprising an  adjunction \(\adj{\widehat{\free}}{\widehat{\forget}}{ \cat{C}^{S}}{ \cat{D}^{T}}\) with unit \(\widehat{\eta} \from \Id_{\cat{C}^{S}}\to \widehat{\forget}\widehat{\free}\) and counit \(\widehat{\epsilon} \from  \widehat{\free}\widehat{\forget}\to \Id_{ \cat{D}^{T}}\) such that \(\widehat{\free}\), \(\widehat{\forget}\), \(\widehat{\eta}\), \(\widehat{\epsilon}\) lift \(\free\),\(\forget\), \(\eta\), \(\epsilon\), respectively.
\end{definition}

As explained in the following definition and proposition,
in case we have a lift \((\widehat{\free},\widehat{\forget}, \widehat{\eta}, \widehat{\epsilon})\) of the adjunction \((\free, \forget, \eta, \epsilon)\), the triangle identities ensure that the lifting datum of \(\free\) is obtained from the lifting datum of \(\forget\) combined with the unit and counit of the adjunction.

\begin{definition}\label{def:mate}
  Consider an adjunction \(\adj{\free}{\forget}{ \cat{C}}{ \cat{D}}\) and endofunctors \(R \from \cat{C}\to \cat{C}\) as well as \(S \from \cat{D} \to \cat{D}\).
  The \emph{mate} of a natural transformation \(\xi \from R\forget\to \forget S\) is the natural transformation \(\xi^{\flat}\from \free R\to S\free\) given by
  \[
    \begin{tikzcd}[ampersand replacement=\&,cramped]
      {\free R(M)} \&\& {\free R\forget\free(M)} \&\& {\free\forget S\free(M)} \&\& {S\free(M).}
      \arrow["{\free R\eta_M}", from=1-1, to=1-3]
      \arrow["{\free\xi_{\free(M)}}", from=1-3, to=1-5]
      \arrow["{\epsilon_{S\free(M)}}", from=1-5, to=1-7]
    \end{tikzcd} \qedhere
  \]
\end{definition}

The next result follows immediately from \cite[Theorem~3.13]{BruguieresLackVirelizier}.
\begin{proposition}\label{prop:lifting-adjunctions}
  Consider monads \(R\) on \(\cat{C}\) and \(S\) on \(\cat{D}\), and let \(\adj{\free}{\forget}{\cat{C}}{\cat{D}}\) be an adjunction with unit \(\eta \from \Id_{ \cat{C}} \to \forget\free\) and counit \(\epsilon\from \free\forget \to \Id_{ \cat{D}}\).
  The bijection of \cref{lem:lifting-co-and-co} applied to \(\free\) and \(\forget\) induces bijections between
  \begin{enumerate}[label=\textnormal{(\arabic*)},leftmargin=0.8cm]
      \item lifts of \((\free, \forget, \eta, \epsilon)\);
      \item lifting data \(\zeta \colon S\free \to \free R\) and \(\xi \colon R\forget \to \forget S\) such that \eqref{eq:lifting-alpha} is satisfied by \(\eta\) and \(\epsilon\), with \(\zeta=(\xi^{\flat})^{-1}\);
      \item lifting data \(\zeta \colon S\free \to \free R\) which are natural isomorphisms;
      \item lifting data \(\xi \from R\forget \to \forget S\) such that its mate \(\xi^{\flat}\) is a natural isomorphism.
  \end{enumerate}
\end{proposition}

We now want to study adjunctions in two variables of the form
\[
  \cat[E]\left(\free(C,D),E\right) \cong \cat[C](C,\forget(D,E)),
\]
via their pointwise properties, by looking at their units and counits.

\begin{definition}\label{def:parametric-adjunction}
  A \emph{parametric adjunction} is a quadruple \((\free,\forget, \eta, \epsilon)\) consisting of functors \(\free \from \cat{C}\times \cat{D} \to \cat{E}\) and \(\forget \from \cat{D}\op \times \cat{E} \to \cat{C}\) and morphisms
  \begin{equation*}
    \eta^{(D)}_{C} \from C \to \forget(D,\free(C,D)), \qquad \epsilon^{(D)}_{E} \from \free(\forget(D,E),D) \to E, \qquad\quad C\in \cat{C}, D\in \cat{D}, E \in \cat{E},
  \end{equation*}
  natural in \(C\) and \(E\), and dinatural in \(D\) such that for all \(D\in \cat{D}\) we have an adjunction \(\free( \blank ,D) \dashv \forget(D, \blank)\) with unit \((\eta^{(D)}_{C})_{C \in \cat{C}}\) and counit \((\epsilon^{(D)}_{E})_{E \in \cat{E}}\).
\end{definition}

As discussed for example in \cite[\S~IV.7]{Maclane}, parametric adjunctions are determined by their pointwise properties.

\begin{proposition}\label{prop:parameter-theorem}
  Let \( \cat{C}, \cat{D}, \cat{E}\) be categories and let \(\forget \from \cat{D}\op \times \cat{E} \to \cat{C}\) be a functor.
  Suppose furthermore that for all \(D\in \cat{D}\) the functor \(\forget_{D}\eqdef \forget(D, \blank) \from \cat{E}\to \cat{C}\) has a left adjoint  \(\free_{D}\from \cat{C}\to \cat{E}\) with unit \(\eta_{D,\blank} \from \Id_{ \cat{C}} \to \forget_{D}\free_{D} \) and counit \(\epsilon_{D,\blank} \from \free_{D}\forget_{D}\to \Id_{ \cat{E}}\).
  Then there is a unique parametric adjunction \((\free,\forget, \eta, \epsilon)\) with \(\free\from \cat{C}\times \cat{D} \to \cat{E}\)
  such that \(\free(\blank ,D) = \free_{D}\), \(\eta_{D,\blank} = \eta^{(D)}\), and \(\epsilon_{D,\blank} = \epsilon^{(D)}\) for all \(D \in \cat{D}\).
\end{proposition}

In the setting of \cref{prop:parameter-theorem}, for any given arrow \(g\from D \to D'\) in \(\cat{D}\) and object \(C \in \cat{C}\),
the morphism \(\free(C,g) \from \free(C,D) \to \free(C,D')\) in \(\cat[E]\) is uniquely determined by the following diagram:
\begin{equation}\label{eq:param-induced-morphism}
  \begin{tikzcd}[ampersand replacement=\&,cramped]
    {\free(C,D)} \&\& {\free(\forget(D',\free(C,D')),D)} \\
    \\
    {\free(C,D')} \&\& {\free(\forget(D,\free(C,D')),D)}
    \arrow["{\free\left(\eta^{(D')}_C,D\right)}", from=1-1, to=1-3]
    \arrow["{\free(C,g)}"', dashed, from=1-1, to=3-1]
    \arrow["{\free(\forget(g,\free(C,D')),D)}", from=1-3, to=3-3]
    \arrow["{\epsilon^{(D)}_{\free(C,D')}}", from=3-3, to=3-1]
  \end{tikzcd}
\end{equation}

Lifts of parametric adjunctions are also determined by pointwise lifts.
Given a parametric adjunction \((\free,\forget, \eta, \epsilon)\) and an object \(D \in \cat{D}\), we write \(\free_D \defeq \free(\blank,D)\) and \(\forget_D \defeq \forget(D,\blank)\) as in \cref{prop:parameter-theorem}.

\begin{lemma}\label{lem:lifts-of-parametric-adjunction}
  Suppose that \(R\), \(S\), and \(T\) are monads on \(\cat{C}\), \(\cat{D}\), and \(\cat{E}\), respectively, and that \(\free \from \cat{C}\times \cat{D} \to \cat{E}\) and \(\forget \from \cat{D}\op \times \cat{E} \to \cat{C}\) assemble into a  parametric adjunction \((\free, \forget, \eta, \epsilon)\).
  Furthermore, let \(\widehat{\forget} \from \left(\cat{D}^{S}\right)\op\times \cat{E}^{T} \to \cat{C}^{R}\) be a lift of \(\forget\).
  The following are equivalent:
  \begin{enumerate}[label=\textnormal{(\arabic*)},leftmargin=0.8cm]
    \item\label{item:pointwise-lift-of-adjs-1} there is a (necessarily unique) lift \(\big(\widehat{\free}, \widehat{\forget}, \widehat{\eta}, \widehat{\epsilon}\big)\) of the parametric adjunction \((\free,\forget, \eta, \epsilon)\);

    \item\label{item:pointwise-lift-of-adjs-2} there are families of lifts \(\big(\widehat{\free}_{\boldsymbol{D}}, \widehat{\forget}_{\boldsymbol{D}}, \widehat{\eta}^{\boldsymbol{D}}, \widehat{\epsilon}^{\boldsymbol{D}}\big)_{\boldsymbol{D}\in \cat{D}^{S}}\) of \((\free_{D}, \forget_{D}, \eta^{(D)}, \epsilon^{(D)})_{D\in \cat{D}}\).
  \end{enumerate}
\end{lemma}

\begin{proof}
  First, we note that \ref{item:pointwise-lift-of-adjs-1} \(\Rightarrow\) \ref{item:pointwise-lift-of-adjs-2}.
  Conversely,  suppose that we have a pointwise lift \((\widehat{\free}_{\boldsymbol{D}}, \widehat{\forget}_{\boldsymbol{D}}, \widehat{\eta}^{\boldsymbol{D}}, \widehat{\epsilon}^{\boldsymbol{D}})_{\boldsymbol{D}\in \cat{D}^{S}}\) of the adjunction \((\free_{D}, \forget_{D}, \eta^{(D)}, \epsilon^{(D)})_{D\in \cat{D}}\).
  By \cref{prop:parameter-theorem}, there is a unique way of assembling the \(\widehat{\free}_{\boldsymbol{D}}\)'s into a functor \(\widehat{\free}\from \cat{C}^{R}\times \cat{D}^{S}\to \cat{E}^{T}\) such that \((\widehat{\free}, \widehat{\forget}, \widehat{\eta}, \widehat{\epsilon})\) is a parametric adjunction and \(\widehat{\free}( \blank , \boldsymbol{D}) = \widehat{\free}_{\boldsymbol{D}}\) for all \(\boldsymbol{D}\in \cat{D}^{S}\).
  To show that \(\widehat{\free}\) is a lift of \(\free\), we fix morphisms \(f \in \cat{C}^{R}(\boldsymbol{M}, \boldsymbol{N})\) and \(g \in \cat{D}^{S}(\boldsymbol{P}, \boldsymbol{Q})\) and  compute
  \begin{align*}
    \U^{T}\widehat{\free}(f,g)
    & = \U^{T}(\widehat{\free}(f, \boldsymbol{Q})\widehat{\free}(\boldsymbol{M}, g)) = \U^{T}(\widehat{\free}_{\boldsymbol{Q}}f)\U^{T}(\widehat{\free}(\boldsymbol{M}, g)) \\
    \overset{\eqref{eq:param-induced-morphism}}&{=} \free_Q(\U^Rf) \U^{T}(
      \widehat{\epsilon}^{\boldsymbol{P}}_{\widehat{\free}(\boldsymbol{M}, \boldsymbol{Q})} \widehat{\free}(\widehat{\forget}(g, \widehat{\free}(\boldsymbol{M}, \boldsymbol{Q})), \boldsymbol{P}) \widehat{\free}(\widehat{\eta}^{\boldsymbol{Q}}_{\boldsymbol{M}}, \boldsymbol{P})) \\
    & = \free(\U^{R}f, Q) \epsilon^{(P)}_{\free(M,Q)}
      \free(\forget(\U^{S}g,\free(M,Q)),P)
      \free(\eta^{(Q)}_{M},P) \\
    \overset{\eqref{eq:param-induced-morphism}}&{=} \free(\U^{R}f, Q)\free(M,\U^{S}g) = \free(\U^{R}f, \U^{S}g).
  \end{align*}
  Therefore \(\widehat{\free}\) is a lift of \(\free\) and \((\widehat{\free}, \widehat{\forget}, \widehat{\eta}, \widehat{\epsilon})\) lifts the parametric adjunction \((\free,\forget, \eta, \epsilon)\).
\end{proof}

\begin{lemma}\label{lem:lifting-of-parametric-adjunctions}
  Consider monads \(R\), \(S\), and \(T\) on \(\cat{C}\), \(\cat{D}\), and \(\cat{E}\) respectively, and
  a parametric adjunction \((\free,\forget, \eta, \epsilon)\) where
  \(\free\from \cat{C}\times \cat{D}\to \cat{E}\) and \(\forget \from \cat{D}\op\times \cat{E} \to \cat{C}\).
  There are bijections between
  \begin{enumerate}[label=\textnormal{(\arabic*)},leftmargin=0.8cm]
    \item\label{item:lifting-of-parametric-adjunctions-1} lifts of \((\free,\forget, \eta, \epsilon)\) to a parametric adjunction \((\widehat{\free}, \widehat{\forget}, \widehat{\eta}, \widehat{\epsilon})\) with \(\widehat{\free}\from \cat{C}^{R}\times \cat{D}^{S}\to \cat{E}^{T}\) and \(\widehat{\forget} \from {\cat{D}^{S}}\op\times \cat{E}^{T} \to \cat{C}^{R}\);

    \item\label{item:lifting-of-parametric-adjunctions-2} mixed lifting data \(\nu \from R\forget(S \blank , \bblank) \to \forget( \blank, T \bblank)\) such that for all \(\boldsymbol{D} \in \cat{D}^{S}\)  the mate of
    \begin{equation*}
     \left(\chi^{\boldsymbol{D}}_E = \nu_{D,E} \circ R\forget(\act^{S}_{D}, E) \from R\forget(D, E) \to \forget(D, TE)\right)_{E \in \cat{E}}
    \end{equation*}
    is invertible;

    \item\label{item:lifting-of-parametric-adjunctions-3} parametric lifting data \(\left\{\chi^{\boldsymbol{D}} \colon R\forget_D \to \forget_DT\right\}_{\boldsymbol{D} \in \cat{D}^{S}}\), natural in \(\boldsymbol{D} \in \cat{D}^{S}\), such that for every \(\boldsymbol{D}\) the mate of \(\chi^{\boldsymbol{D}}\) is invertible.
  \end{enumerate}
\end{lemma}
\begin{proof}
  By \cref{thm:all-liftings-data}, mixed lifting data \(\nu\), are in bijection with parametric lifting data \((\chi^{\boldsymbol{D}})_{\boldsymbol{D} \in \cat{D}^{S}}\), where \(\chi^{\boldsymbol{D}}= \nu_{D, \blank} \circ R\forget(\act^S_D, \blank)\).
  Hence \ref{item:lifting-of-parametric-adjunctions-2} and \ref{item:lifting-of-parametric-adjunctions-3} describe the same data.

  Moreover, \cref{thm:all-liftings-data} shows that parametric lifting data correspond bijectively to lifts of \(\forget\) and it remains to show that the conditions stated in \ref{item:lifting-of-parametric-adjunctions-3} are necessary and sufficient to lift the  whole parametric adjunction.

  Fix a lift \(\widehat{\forget} \from {(\cat{D}^{S})}\op\times \cat{E}^{T} \to \cat{C}^{R}\) of \(\forget\), and let \((\chi^{\boldsymbol{D}}\colon R\forget_D\to \forget_DT)_{\boldsymbol{D} \in \cat{D}^{S}}\) be the corresponding parametric lifting datum.
  Applying \cref{prop:lifting-adjunctions} to the adjunction \(\free_D\dashv \forget_D\) shows that \(\widehat{\forget}_{\boldsymbol{D}}\) is the right adjoint in a lifted adjunction \(\widehat{\free}_{\boldsymbol{D}}\dashv \widehat{\forget}_{\boldsymbol{D}}\)
  if and only if the mate of \(\chi^{\boldsymbol{D}}\) is invertible.
  Note that the lifted adjunction is unique, because the lifting datum determines both lifted functors and the lifted unit and counit.
  Hence, there are pointwise lifted adjunctions \({(\widehat{\free}_{\boldsymbol{D}}, \widehat{\forget}_{\boldsymbol{D}}, \widehat{\eta}^{\boldsymbol{D}}, \widehat{\epsilon}^{\boldsymbol{D}})}_{\boldsymbol{D}\in\cat{D}^{S}}\),
  which, by \cref{lem:lifts-of-parametric-adjunction}, are equivalent to a lift \((\widehat{\free},\widehat{\forget},\widehat{\eta},\widehat{\epsilon})\) of the original parametric adjunction.
\end{proof}

Applying \cref{lem:lifting-of-parametric-adjunctions} to the tensor--hom adjunction of a closed monoidal category allows us to determine necessary and sufficient conditions for a gabi-monad to be Hopf.
Note in particular that no normality assumption is needed in the statement; it is instead implied by the invertibility of the parametric mate.

\begin{theorem}\label{thm:lift-to-hopf-monads}
  Let \(\cat[C]\) be a closed monoidal category and \((T, s , \mathfrak{a}_{\tu})\) a gabi-monad on \(\cat[C]\).
  Then \(T\) is a left Hopf monad if and only if for every object \(\boldsymbol{M} \in \cat[C]^T\) the mate of
  \begin{equation*}
    \left(\gamma^{\boldsymbol{M}}_X \eqdef s_{M,X} \circ T[\act_{M}, X] \from T[M,X] \to [M,TX]\right)_{X \in \cat{C}},
  \end{equation*}
  namely,
  \[
    TX\otimes M \xrightarrow{T \coev^{M}_X \otimes M} T[M,X\otimes M]\otimes M \xrightarrow{\gamma^{\boldsymbol{M}}_{X\otimes M}\otimes M} [M,T(X\otimes M)]\otimes M \xrightarrow{\ev_{T(X \otimes M)}^M} T(X\otimes M),
  \]
  is invertible.
\end{theorem}
\begin{proof}
  Theorem~7.1 of \cite{Moerdijk2002} and Theorem~3.6 of \cite{BruguieresLackVirelizier} show that \(T\) is left Hopf if and only if \( \cat{C}^{T}\) is closed monoidal and the forgetful functor \(\U^{T}\from \cat{C}^{T}\to \cat{C}\) is strict closed monoidal.
  \Cref{lem:lifting-of-parametric-adjunctions} shows that the tensor--hom adjunction of \(\cat{C}\) lifts to \(\cat{C}^{T}\) if and only if the mate of \((\gamma^{\boldsymbol{M}}_X \from T[M, X] \to[M, TX])_{X \in \cat{C}}\) is invertible for every \(\boldsymbol{M} \in \cat[C]^T\).

  First, suppose that \(T\) is left Hopf.
  Then \(\cat{C}^{T}\) is closed monoidal and \(\U^T\) is strict closed monoidal.
  In particular, the tensor--hom adjunction lifts to \(\cat{C}^{T}\),
  so \cref{lem:lifting-of-parametric-adjunctions} implies that all the above mates are invertible.

  Conversely, let all mates of \(\gamma^{\boldsymbol{M}}_X\) be invertible, so that the tensor--hom adjunction lifts as a parametric adjunction.
  By \cref{thm:bijection_skew_closed_monoidal_structures}, the tensor product determines a skew-monoidal structure on \(\cat{C}^{T}\)
  with associator and unitors given by the lifted skew-closed coherence morphisms and the unit and counit of the lifted adjunction, explicitly given by \eqref{eq:mon_closed}.
  The associator and unitors are thus lifted from \(\cat{C}\).
  Since \(\U^T\) is conservative and the coherence morphisms in \(\cat{C}\) are invertible,
  the associator and unitors in \(\cat{C}^{T}\) are invertible as well, making \(\cat{C}^T\) monoidal (not just skew-monoidal), and \(U^T\) strict monoidal.
  Hence, the lifted skew-closed structure is normal-closed by the normality statement in \cref{thm:bijection_skew_closed_monoidal_structures},
  \(\cat{C}^{T}\) is closed monoidal, and \(\U^T\) is strict closed monoidal.
  In other words, \(T\) is left Hopf.
\end{proof}

\cref{thm:lift-to-hopf-monads} allows us to refine the results from \cite[\S5.3]{Berger-Vercruysse-Saracco}; in particular, the relationship between normal gabi-algebras, Hopf algebras, and the invertibility of the plus-minus map.
Recall the definition of a gabi-algebra from \cref{def:gabi-algebra}.

\begin{corollary}\label{cor:gabi-in-the-com-case}
  Let \(k\) be a commutative ring.
  Given a \(k\)-algebra \(A\), we have a bijective correspondence between the following sets of data
  \begin{enumerate}[label=\textnormal{(\alph*)},leftmargin=0.8cm]
    \item\label{item:gabi-Hopf1} Hopf algebra structures on \(A\);
    \item\label{item:gabi-Hopf2} normal gabi-algebra structures on \(A\);
    \item\label{item:gabi-Hopf3} gabi-algebra structures on \(A\) for which the canonical map
    \begin{equation}
      \beta \colon A \otimes A \to A \otimes A, \qquad a\otimes b \mapsto a_+ \otimes a_-b
    \end{equation}
    is invertible.
  \end{enumerate}
\end{corollary}

\begin{proof}
    The bijective correspondence between \ref{item:gabi-Hopf1} and \ref{item:gabi-Hopf2} is the content of \cite[Theorem 5.12]{Berger-Vercruysse-Saracco}, and any Hopf algebra \(A\) is a gabi-algebra with invertible \(\beta\colon a \otimes b \mapsto a_{(1)} \otimes S(a_{(2)})b\).

    Then let \((A,\varepsilon,\delta)\) be a gabi-algebra. In \cref{thm:lift-to-hopf-monads}, take \(\cat{C} = \cM\), the category of \(k\)-modules, and \(T = A \otimes \blank\), with \(\act_\tu(\blank \otimes 1_k) = \varepsilon\) and
    \[s_{V,W} \colon A \otimes \Hom(A \otimes V,W) \to \Hom(V, A \otimes W), \qquad a \otimes f \mapsto \left(v \mapsto a_+ \otimes f(a_- \otimes v)\right)\]
    as in \cref{ex:gabi-algebras}.
    For every \(\boldsymbol{M} = (M,\act_M)\) in \(\prescript{}{A}{\cM}\) and \(W\) in \(\cM\), the associated lifting datum
    \[\gamma^{\boldsymbol{M}}_W \coloneqq s_{M,W}\circ (A \otimes \Hom(\act_M,W)) \colon A \otimes \Hom(M,W) \to \Hom(M,A \otimes W)\]
    is given by
    \[\gamma^{\boldsymbol{M}}_W(a \otimes f)(m) = a_+ \otimes f(a_- \cdot m)\]
    for every \(a \in A\), \(m \in M\), \(f \in \Hom(M,W)\).
    The mate of \(\gamma^{\boldsymbol{M}}_W\) is the natural transformation
    \[\tau^{\boldsymbol{M}}_V \colon (A \otimes V) \otimes M \to A \otimes (V \otimes M), \qquad a \otimes v \otimes m \mapsto a_+ \otimes v \otimes a_-\cdot m,\]
    (notation as in \cite[Remark 3.9]{Berger-Vercruysse-Saracco}).
    Now, by \cref{thm:lift-to-hopf-monads}, \(A \otimes \blank\) is a left Hopf monad if and only if \(\tau^{\boldsymbol{M}}\) is a natural isomorphism for every \((M,\act_M)\) in \(\prescript{}{A}{\cM}\). Since \(k\) is a projective generator in \(\cM\) and \((A,m_A)\) is a projective generator in \(\prescript{}{A}{\cM}\), by taking \(V = k\) and \(\boldsymbol{M} = (A,m_A)\), we conclude that \(A \otimes -\) is a Hopf monad if and only if the canonical map
    \[\beta \colon A \otimes A \to A \otimes A, \qquad a\otimes b \mapsto a_+ \otimes a_-b\]
    from \cite[\S5, (5.1)]{Berger-Vercruysse-Saracco} is invertible.

    By adapting \cite[Proposition 5.4]{BruguieresLackVirelizier} and \cite[\S3.4]{Berger-Vercruysse-Saracco}, Hopf and gabi-algebra structures on an algebra \(A\) are in bijection with left Hopf monad structures and gabi-monad structures on \(A \otimes \blank\), respectively.
    Hence, \ref{item:gabi-Hopf1} and \ref{item:gabi-Hopf3} are in bijective correspondence.
\end{proof}

Summing up, an algebra \(A\) is a Hopf algebra if and only if \(A \otimes \blank\) is a left Hopf monad, if and only if \(A \otimes \blank\) is a gabi-monad for which the \(\tau^{\boldsymbol{M}}\)'s are invertible, if and only if \(A\) is a gabi-algebra whose \(\beta\) is invertible. Thus, a gabi-algebra \(A\) is a Hopf algebra if and only if \(\beta\) is invertible. Although \cite[\S5.3]{Berger-Vercruysse-Saracco} seems to suggest that invertibility of \(\beta\) is not enough, the last claim can also be proved by an elementary argument. Indeed, write \(\beta^{-1}(a \otimes b) = a_{(1)} \otimes a_{(2)}b\), so that
\[
  a_{+(1)} \otimes a_{+(2)}a_- = a \otimes 1 = a_{(1)+} \otimes a_{(1)-}a_{(2)}
\]
for every \(a \in A\). Then
\[a = a_{+(1)} \varepsilon( a_{+(2)}a_-) = a_{+(1)} \varepsilon( a_{+(2)})\varepsilon(a_-)  \stackrel{\eqref{eq:GA1}}{=} a_{(1)} \varepsilon( a_{(2)}) \]
and
\[a = a_{(1)+} a_{(1)-}a_{(2)} \stackrel{\eqref{eq:GA2}}{=} \varepsilon(a_{(1)})a_{(2)}.\]
Therefore, one can apply \cite[Theorem 5.3]{Berger-Vercruysse-Saracco} to conclude that \(\Delta \colon a \mapsto a_{(1)} \otimes a_{(2)}\) is coassociative and counital, making \(A\) a Hopf algebra.

\section{Dyadic gabi-monads and \(*\)-autonomous categories}\label{sec:dyadic-gabi}

The notion of a \(*\)-autonomous category originally appeared in \cite{Barr-autonomous} in the symmetric closed monoidal setting and was later extended to the non-symmetric case in~\cite{Barr-nonsymmetric-autonomous}.
This concept was later revisited by Boyarchenko and Drinfeld~\cite{Boyarchenko-Drinfeld}, who called such categories \emph{Grothendieck--Verdier} due to examples coming from Verdier duality.
We shall always work with the non-symmetric version, unless otherwise stated.

In a \(*\)-autonomous category, it is possible to recover the tensor product purely from the internal-hom functors; see \cref{cor:tp-in-gv}.
This further simplifies the conditions for being Hopf, see \cref{thm:dyadic-gabi}.

\begin{definition}
  Let \(\cat[C]\) be a monoidal category.
  An object \(\dobj\) is said to be a \emph{dualising object} if for every \(X\) in \(\cat[C]\),
  the functor \(\cat[C](\blank \otimes X, \dobj)\) is representable by some \(DX \in \cat[C]\),
  and the resulting functor \(D \from \cat[C]\op \to \cat[C]\) is an equivalence of categories.
  That is, we have the following natural isomorphisms:
  \[
    \cat[C](\blank \otimes X, \dobj) \cong \cat[C](\blank, DX),
    \qquad\qquad
    \cat[C](X \otimes \blank, \dobj) \cong \cat[C](X, D\blank) \cong \cat[C](\blank, D^{-1}X).
  \]
  A \emph{\textup(non-symmetric\textup) \(*\)-autonomous} structure on \(\cat[C]\)
  is the datum of a dualising object \(\dobj \in \cat[C]\).
\end{definition}

Prominent examples of \(*\)-autonomous categories include rigid monoidal categories with the monoidal unit as dualising object.
Here, the dualising functor is monoidal, which is peculiar to this class of examples and does not extend to all \(*\)-autonomous categories.
However, what does extend is that \(*\)-autonomous categories are left and right closed.

\begin{proposition}
  Every non-symmetric \(*\)-autonomous category is left and right closed, with internal homs given by
  \({[X, \blank]} = D(X \otimes D^{-1}(\blank))\), right adjoint to \(\blank \otimes X\), and \({\{X, \blank\}} = D^{-1}(D(\blank) \otimes X) \), right adjoint to \(X \otimes \blank\), for all \(X \in \cat[C]\).
\end{proposition}
\begin{proof}
  One computes
  \begin{align*}
    \cat[C](X \otimes Y, Z)
    & \cong \cat[C](DZ,D(X \otimes Y))
    \cong \cat[C](DZ \otimes X \otimes Y,\dobj)
    \cong \cat[C](DZ \otimes X,DY) \\
    & \cong \cat[C](Y,D^{-1}(DZ \otimes X)) = \cat[C](Y,\{X,Z\}) \\
    \cat[C](X \otimes Y, Z)
    & \cong \cat[C](X \otimes Y \otimes D^{-1}Z, \dobj)
    \cong \cat[C](X, D(Y \otimes D^{-1}Z))
    \cong \cat[C](X, {[Y,Z]}).\qedhere
  \end{align*}
\end{proof}

A monoidal category which is left and right closed is called \emph{biclosed monoidal}.

\begin{corollary}\label{cor:tp-in-gv}
  The tensor product of a non-symmetric \(*\)-autonomous category \(\cat{C}\) satisfies
  \[
    D^{-1}[X,DY] \cong X \otimes Y \cong D\{Y,D^{-1}X\}.
  \]
\end{corollary}

Before we can discuss gabi-monads that lift \(*\)-autonomous structures, we need to study how the two internal homs interact.

\begin{lemma}\label{lem:biclosed-is-double-the-fun}
  Let \( \cat{C}\) be a biclosed monoidal category with unit \(\tu\) and left, respectively right internal homs \([ \blank , \bblank], \{ \blank , \bblank \} \from \cat{C}\op\times \cat{C} \to \cat{C}\).
  Then, there is a family of isomorphisms \((\xi_{X,Y,Z} \from  [X, \{Y, Z\}] \to \{Y, [X, Z]\})_{X,Y,Z \in \cat{C}}\), natural in all three variables, given by
  \begin{equation}\label{eq:interchange-between-internal-homs}
    \begin{tikzcd}[ampersand replacement=\&,cramped]
      {[X, \{Y, Z\}]  } \&\& {\{Y,Y \otimes [X,\{Y,Z\}]\}} \\
      \\
      \&\& {\{Y,[X,Y \otimes [X,\{Y,Z\}] \otimes X]\}} \\
      \\
      {\{Y, [X, Z]\}} \&\& {\{Y,[X,Y \otimes \{Y,Z\}]\}}
      \arrow["{\coev^{r,Y}_{[X, \{Y, Z\}]}}", from=1-1, to=1-3]
      \arrow["{\xi_{X,Y,Z}}"', dashed, from=1-1, to=5-1]
      \arrow["{\{Y, \coev^{\ell,X}_{Y\otimes [X,\{Y,Z\}]}\}}", from=1-3, to=3-3]
      \arrow["{\{Y,[X,Y\otimes \ev^{\ell,X}_{\{Y,Z\}}]\}}", from=3-3, to=5-3]
      \arrow["{\{Y,[X,\ev^{r,Y}_Z]\}}", from=5-3, to=5-1]
    \end{tikzcd}
  \end{equation}
\end{lemma}
\begin{proof}
  Diagram~\eqref{eq:interchange-between-internal-homs} is obtained by considering for each \(A \in \cat{C}\) the chain of adjunctions
  \begin{equation*}
    \cat{C}(A, [X, \{Y,Z\}])
    \cong \cat{C}(A \otimes X, \{Y,Z\})
    \cong \cat{C}(Y \otimes A \otimes X, Z)
    \cong \cat{C}(A, \{Y, [X,Z]\}).\qedhere
  \end{equation*}
\end{proof}

Given a monad \(T\) on a  biclosed monoidal category \(\cat{C}\), there are two different closed structures one may consider lifting in the sense of \cref{thm:gabi-bijection}:
the adjoint closed structure of \(\blank \otimes X\), and that of \(X \otimes \blank\).
In these cases, we speak of \emph{left}, \emph{right} gabi-monads.

\begin{definition}\label{def:double-gabi-monad}
  A \emph{dyadic gabi-monad} on a biclosed monoidal category \(\cat{C}\) is a quadruple \((T,s,t, \act_{\tu})\) such that \((T,s,\act_{\tu})\) and \((T,t,\act_{\tu})\) are left and right gabi-monads, respectively, and
  the following diagram commutes for every \(X,Y,Z\) in \(\cat{C}\):
  \begin{equation}\label{eq:dyadic}
    \begin{tikzcd}[ampersand replacement=\&,column sep=3.5em]
      {T[TX,\{TY,Z\}]} \& {T\{TY,[TX,Z]\}} \\
      {[X,T\{TY,Z\}]} \& {\{Y,T[TX,Z]\}} \\
      {[X,\{Y,TZ\}]} \& {\{Y,[X,TZ]\}}
      \arrow["{{T\xi_{TX,TY,Z}}}", from=1-1, to=1-2]
      \arrow["{{s_{X,\{TY,Z\}}}}"', from=1-1, to=2-1]
      \arrow["{{t_{Y,[TX,Z]}}}", from=1-2, to=2-2]
      \arrow["{{[X,t_{Y,Z}]}}"', from=2-1, to=3-1]
      \arrow["{{\{Y,s_{X,Z}\}}}", from=2-2, to=3-2]
      \arrow["{{\xi_{X,Y,TZ}}}"', from=3-1, to=3-2]
    \end{tikzcd}
  \end{equation}

  A dyadic gabi-monad is \emph{normal} if its left and right gabi-monad structures are normal.
\end{definition}

\begin{proposition}\label{prop:dgas-recover-nice-biclosed-structures}
  Consider a monad \(T\) on a biclosed monoidal category \( \cat{C}\).
  There is a bijection between
  \begin{enumerate}[label=\textnormal{(\arabic*)},leftmargin=0.8cm]
    \item dyadic gabi-monad structures on \(T\) and
    \item pairs \(\big([ \blank , \bblank]_{T}, \{ \blank , \bblank \}_{T}\big)\) of skew-closed structures on \( \cat{C}^{T}\), such that:
    \begin{enumerate}[label=\textnormal{(\alph*)}]
      \item \(\U^{T}\) is strict closed with respect to both skew-closed structures,
      \item \([ \blank , \bblank]_{T}\) and \(\{ \blank , \bblank \}_{T}\) share the same closed unit \((\tu, \act_\tu)\), and
      \item there is a natural isomorphism
      \begin{equation*}
        \boldsymbol{\xi}_{\boldsymbol{X}, \boldsymbol{Y}, \boldsymbol{Z}}\from
        [\boldsymbol{X}, \{\boldsymbol{Y}, \boldsymbol{Z}\}_{T}]_{T} \to
        \{\boldsymbol{Y}, [\boldsymbol{X}, \boldsymbol{Z}]_{T}\}_{T},
      \end{equation*}
      with \(\U^{T}(\boldsymbol{\xi}_{\boldsymbol{X}, \boldsymbol{Y}, \boldsymbol{Z}}) = \xi_{X,Y,Z}\).
    \end{enumerate}
  \end{enumerate}

  Under this correspondence, normal dyadic gabi-monad structures on \(T\) correspond to
  normal closed structures \([ \blank , \bblank]_{T}\) and \(\{ \blank , \bblank \}_{T}\).
\end{proposition}

\begin{proof}
  Applying~\cref{thm:gabi-bijection} twice shows that there is a bijection between
  \begin{enumerate}[leftmargin=*]
    \item quadruples \((T,s,t,\act_{\tu})\) such that \((T,s,\act_{\tu})\) is a left gabi-monad and \((T,t, \act_{\tu})\) is a right gabi-monad, and
    \item  pairs \(\big([ \blank , \bblank]_{T}, \{ \blank , \bblank \}_{T}\big)\) of skew-closed structures on \( \cat{C}^{T}\) that share a closed unit and \(\U^{T}\) is strict closed with respect to both of them.
  \end{enumerate}
  Hence it remains to show that the canonical isomorphism \(\xi_{X,Y,Z}\from [X, \{Y,Z\}] \to \{Y, [X,Z]\}\) \eqref{eq:interchange-between-internal-homs} lifts to \( \cat{C}^{T}\)
  if and only if Equation~\eqref{eq:dyadic} holds, which is a straightforward check.
  For instance, \cref{fig:dyadic-gabimonad-lift-xi} shows how \eqref{eq:dyadic} descends from the fact that \(\xi\) is a morphism of \(T\)-algebras.
\end{proof}

The next lemma is a reformulation,
in a setup more convenient for our purposes,
of some known necessary and sufficient conditions for a closed category to be \(*\)-autonomous.
See \cite[Definition~2.4 and Section~6]{Barr-nonsymmetric-autonomous}.
The proof is added for the sake of completeness.

\begin{lemma}\label{lem:grothendieck-verdier-from-biclosed}
  Let \( \cat{C}\) be a category with two normal-closed structures \([ \blank , \bblank], \{\blank, \bblank\} \from \cat{C}\op\times \cat{C} \to \cat{C}\) that share a closed unit \(\tu \in \cat{C}\) and satisfy \([X, \{Y, Z\}] \cong \{Y, [X, Z]\}\) naturally for all \(X, Y, Z \in \cat{C}\).
  For every \(\dobj\in \cat{C}\) the following are equivalent:
  \begin{enumerate}[label=\textnormal{(\arabic*)},leftmargin=0.8cm]
    \item\label{item:GVbiclosed1} \([ \blank ,\dobj]\from \cat{C}\op\to \cat{C}\) is an equivalence of categories with quasi-inverse \(\{ \blank , \dobj\} \from \cat{C}\op\to \cat{C}\).
    \item\label{item:GVbiclosed2} \((\cat{C},\dobj)\) is a non-symmetric \(*\)-autonomous category with internal hom \([ \blank , \bblank]\) and \(\{ \blank , \bblank\}\).
  \end{enumerate}
\end{lemma}
\begin{proof}
  We begin by showing \(\ref{item:GVbiclosed1} \Rightarrow \ref{item:GVbiclosed2}\).
  Let us write  \(D= [ \blank ,\dobj] \) and \(R =\{ \blank ,\dobj\}\).
  Note that
  \begin{equation}\label{eq:D_closed}
      \{DY,DX\} = \{DY,[X,d]\} \cong [X,\{DY,d\}] \cong [X,Y]
  \end{equation}
  and that by left normality \ref{item:N1} we obtain
    \begin{equation}\label{eq:swap_int_homs}
      \cat{C}(X, [Y, Z])
      \cong \cat{C}(\tu, \{X, [Y,Z]\})
      \cong \cat{C}(\tu, [Y,\{X,Z\}])
      \cong \cat{C}(Y,\{X,Z\}).
    \end{equation}

  Define \(X\otimes Y \eqdef R[X, DY]\).
  Then
  \begin{align*}
    \cat{C} (X\otimes Y, Z)
    & = \cat{C}(R[X,DY],Z)
      \cong \cat{C}(DZ, [X, DY])
      \stackrel{\eqref{eq:swap_int_homs}}{\cong} \cat{C}(X,\{DZ, DY\}) \stackrel{\eqref{eq:D_closed}}{\cong} \cat{C}(X, [Y,Z]).
  \end{align*}
  Therefore, we have for every \(Y\in \cat{C}\) an adjunction \( \blank \otimes Y \dashv[Y, \blank]\), turning \( \cat{C}\) into a closed monoidal category by \cref{thm:bijection_skew_closed_monoidal_structures}.
  In addition, \(\cat[C](X \otimes Y,\dobj) \cong \cat[C](X, [Y,\dobj]) = \cat[C](X, DY)\).
  Moreover, by left normality again
  \[
    \cat{C}(X \otimes Y, Z)
    \cong \cat{C}(X, [Y, Z])
    \cong \cat{C}(\tu, \{X, [Y, Z]\})
    \cong \cat{C}(\tu, [Y, \{X, Z\}])
    \cong \cat{C}(Y, \{X, Z\})
  \]

  If \(( \cat{C}, \dobj)\) is \(*\)-autonomous, it is in particular biclosed monoidal and hence \(\ref{item:GVbiclosed2} \Rightarrow \ref{item:GVbiclosed1}\).
\end{proof}

Let \(\cat{C}\) be a biclosed monoidal category.
A Hopf monad structure \((\varphi_0,\varphi,s^\ell,s^r)\) on a monad \(T\) on \(\cat{C}\) induces left and right gabi-monad structures \((T,s^\ell,\varphi_0)\) and \((T,s^r,\varphi_0)\), where \(s^\ell\) and \(s^r\) are the binary left and right antipodes and \(\varphi_0 \colon T\tu \to \tu\) is the compatibility with the unit object of the bimonad structure.
By \eqref{eq:ActfromS} and \cref{thm:gabi-bijection}, these are precisely the gabi-monad structures determined by the left and right closed structures induced on \(\cat{C}^{T}\).

\begin{theorem}\label{thm:dyadic-gabi}
  Let \(T\) be a monad on a \(*\)-autonomous category \((\cat{C},\dobj)\) and suppose that
  \(\mb{\dobj}=(\dobj,\act_{\dobj})\) is a \(T\)-algebra with underlying object \(d\).
  Then \(T\) admits a Hopf monad structure if and only if it admits a normal dyadic gabi-monad structure.

  In either case, \(\mb{\dobj}\) is dualising in \(\cat{C}^{T}\),
  and \(\U^T\) strictly preserves the \(*\)-autonomous structure.
\end{theorem}
\begin{proof}
  If \(T\) is a Hopf monad, then the result follows immediately from \cite[Theorem~3.6]{BruguieresLackVirelizier}, \cref{prop:dgas-recover-nice-biclosed-structures}, and \cite[Theorem~5.9]{hasegawa-lemay2018:LinearHopf}.

  Suppose that \((T,s,t,\act_{\tu})\) is a normal dyadic gabi-monad.
  To keep our argument lucid, we note that due to \cref{cor:tp-in-gv}, we can replace the tensor product of \( \cat{C}\) with the monoidally equivalent tensor product
  \begin{equation*}
    X\otimes Y \eqdef R[X, DY],\qquad X, Y \in \cat{C}.
  \end{equation*}

  Using \cref{prop:dgas-recover-nice-biclosed-structures}, let \([\blank,\bblank]_T\) and \(\{\blank,\bblank\}_T\) be the two induced closed structures on \(\cat{C}^T\).
  Define functors
  \[
    D_T=[\blank,\mb{\dobj}]_T \from (\cat{C}^T)\op \to \cat{C}^T
    \qquad\text{and}\qquad
    R_T=\{\blank,\mb{\dobj}\}_T\from (\cat{C}^T)\op \to \cat{C}^T.
  \]
  Left normality of the two closed structures and the lifted interchange give natural
  bijections
  \begin{align*}
    \cat{C}^T(\mb{X},D_T\mb{Y})
    & \cong \cat{C}^T(\mb{\tu},\{\mb{X},[\mb{Y},\mb{\dobj}]_T\}_T)
      \cong \cat{C}^T(\mb{\tu},[\mb{Y},\{\mb{X},\mb{\dobj}\}_T]_T)
      \cong \cat{C}^T(\mb{Y},R_T\mb{X})\\
    & \cong {(\cat{C}^T)}\op(R_T\op\mb{X},\mb{Y}),
  \end{align*}
  and so \(R_T\op \dashv D_T\).
  The unit and counit of this adjunction are natural transformations
  \((\blank) \longrightarrow D_TR_T\) in \(\cat{C}^T\)
  and \(R_TD_T \longrightarrow (\blank)\) in \((\cat{C}^T)\op\).
  Since \(\U^{T}\) strictly preserves both closed structures and the interchange,
  it sends these to the unit and counit of the adjunction \(R\op = \{\blank,d\}\op \dashv [\blank, d] = D\).
  Since these maps are isomorphisms and \(\U^T\) is conservative, \(D_T\) and \(R_T\) must be equivalences of categories as well.
  By \cref{lem:grothendieck-verdier-from-biclosed}, we can thus define a monoidal structure on \(\cat{C}^T\) by
  \[
    \mb{X} \otimes_T \mb{Y} \defeq R_T[\mb{X},D_T\mb{Y}]_T.
  \]
  Since \(\U^T\) is strict closed and \(\U^T \mb{\dobj} = \dobj\), we have
  \begin{align*}
    \U^T(\mb{X} \otimes_T \mb{Y})
    & = \U^TR_T[\mb{X},D_T\mb{Y}]_T
      = \U^T\{[\mb{X},D_T\mb{Y}]_T,\mb{\dobj}\}_T
      = \{[X,[Y,\dobj]],\dobj\} \\
    & = R[X,DY]
      = X \otimes Y
      = \U^T \mb{X} \otimes \U^T \mb{Y}.
  \end{align*}
  From this, we conclude that the tensor--hom adjunction of \( \cat{C}\) lifts to \( \cat{C}^{T}\)
  and so invoking \cref{thm:bijection_skew_closed_monoidal_structures} shows that the coherence morphisms of \( \cat{C}\) lift as well,
  turning \(U^{T}\) into a strict monoidal functor.
  In particular, \(\U^T\) is strict monoidal and so, due to~\cite[Theorem~3.6]{BruguieresLackVirelizier}, \(T\) admits a Hopf monad structure.
  By \cref{eq:ActfromS,thm:gabi-bijection}, the induced gabi-monad structures are the original ones, so this Hopf monad structure induces \((T,s,t,\act_{\tu})\).
  This concludes the proof.
\end{proof}

\section{Examples of (normal) gabi-monads that are not Hopf}\label{sec:gabi-not-hopf}

We will now study skew-closed full subcategories \(\cat{D}\) of skew-closed monoidal categories \(\cat{C}\).
In case the canonical inclusion \(\iota \from \cat{D} \to \cat{C}\) has a left adjoint, \(\cat{D}\) can be identified with the algebras \(\cat{C}^{T}\) of a gabi-monad \(T\) on \(\cat{C}\), see \cref{thm:recognition-of-reflectiveness} and \cref{lem:recognition-of-idempotent-gabi-monads}.
In this setting, \(\cat{D} \simeq \cat{C}^{T}\) even admits a skew-closed monoidal structure, but the tensor product differs substantially from that of \(\cat{C}\), providing us with ample examples of gabi-monads that are not Hopf, including normal ones and ones whose Eilenberg--Moore category is skew-closed monoidal.
In the following, to better fit the examples we are going to discuss, we use the term \emph{full subcategory} to denote a pair \((\cat{D}, \iota)\) of a category \( \cat{D}\) and a fully faithful functor \(\iota \from \cat{D} \to \cat{C}\).

\subsection{Idempotent monads}\label{sec:idempotent-monads}

The following exposition of reflective subcategories and idempotent monads is based on \cite[Chapter 3.5]{borceux1994:HandbookCategoricalAlgebra1} and \cite[Chapter 4.2]{borceux1994:HandbookCategoricalAlgebra2}, respectively.

\begin{definition}\label{def:reflective-subcategory}
  A full subcategory \(\iota \from \cat[D]\to \cat[C] \) is \emph{reflective} if the canonical inclusion \(\iota \from \cat[D] \to \cat[C]\) admits a left adjoint \(\free\from \cat[C] \to \cat[D]\) called the \emph{reflector} of \(\iota\).
  We shall often denote reflective subcategories by \((\cat{D}, \iota, \free)\).
\end{definition}

The category \(\mathsf{Ab}\) of abelian groups is a reflective subcategory of the category \(\mathsf{Grp}\) of all groups.
The reflector is given by mapping a group \(\forget\) to its abelianisation \(\forget/[\forget,\forget]\).
This is an idempotent operation in the sense that \(A \cong A/[A,A]\) for all \(A\in \mathsf{Ab}\).
Equivalently, we observe that the counit of this reflector--inclusion adjunction is an isomorphism.

\begin{definition}\label{def:idemptotent-monad}
A monad \((T,\mu, \eta)\) on a category \( \cat[C]\) is \emph{idempotent} if \(\mu \from T^{2}\to T\) is a natural isomorphism.
\end{definition}

Since the multiplication of a monad associated with an adjunction originates from the counit of the adjunction itself, reflector-inclusion adjunctions give rise to idempotent monads. The following result, formalising this observation, corresponds to Corollary~4.2.4 of \cite{borceux1994:HandbookCategoricalAlgebra2}.

\begin{proposition}\label{thm:recognition-of-reflectiveness}
  Let \(\cat{C}\) be a category.
  There is a bijection
  \begin{equation}\label{eq:recog}
    \begin{aligned}
      \{\text{reflective subcategories } (\cat{D}, \iota)\}/{\sim} & \to\{{\text{\(\cat{C}^T\) for an idempotent monad \(T\) on }} \cat{C}\}/{\sim} \\
      ( \cat{D}, \iota) & \mapsto \cat{C}^{\iota \free}, \qquad \text{ where } \free \text{ is the reflector of } \iota, \\
      ( \cat{C}^{T}, \U^{T}) & \mapsfrom \cat{C}^T,
    \end{aligned}
  \end{equation}
  where both sides are taken up to equivalence of categories.
\end{proposition}

Let \(\cat{C}\) be a skew-closed category.
A subcategory \((\cat{D}, \iota)\) of \(\cat{C}\) is called \emph{skew-closed} if \(\cat{D}\) is skew-closed and \(\iota \from \cat{D} \to \cat{C}\) is strong closed.
Part of the forthcoming \cref{lem:recognition-of-idempotent-gabi-monads}, that is, that reflective subcategories of skew-monoidal categories are skew-monoidal, already appears in \cite{Lack-Street}.
However, its proof becomes more elementary in the context we are interested in, where the focus is on closed structures, hence it is reported for the sake of completeness.

\begin{theorem}\label{lem:recognition-of-idempotent-gabi-monads}
  Let \(\cat{C}\) be a skew-closed category.
  The bijection of \cref{eq:recog} extends to a bijection between
  equivalence classes of \(\cat{C}^T\) for an idempotent gabi-monad \(T\) on \(\cat{C}\)
  and equivalence classes of reflective skew-closed subcategories of \(\cat{C}\).

  Let \(\cat{C}\) be skew-closed monoidal, and let \((\cat{D},\iota,\free)\) be a reflective skew-closed subcategory of \(\cat{C}\).
  Denote by \(\eta\) and \(\epsilon\) the unit and counit of \(\free \dashv \iota\), and let \((\psi_0,\psi)\) be the closed structure of \(\iota\).
  Then \(\cat{D}\) is skew-closed monoidal with respect to the reflected tensor product
  \[
    A \otimes_{\cat{D}} B \coloneqq \free(\iota A \otimes_{\cat{C}} \iota B).
  \]
  In this case, the inclusion \(\iota \colon \cat{D} \to \cat{C}\) is lax monoidal via
  \begin{equation*}
    \phi_{A,B}\from  \iota A \otimes_{\cat{C}} \iota B \xrightarrow{\eta_{\iota A \otimes_{\cat{C}} \iota B}} \iota\left(A \otimes_{\cat{D}} B\right)
    \qquad \text{and} \qquad
    \phi_0 \from  \tu_{\cat{C}} \xrightarrow{\psi_0} \iota(\tu_{\cat{D}}),
  \end{equation*}
  and the reflector \(\free\) is colax monoidal via
  \begin{equation*}
    \varphi_{X,Y} \colon \free(X\otimes_{\cat{C}} Y) \xrightarrow{\free(\eta_X \otimes_{\cat{C}} \eta_Y)} \free X \otimes_{\cat{D}} \free Y \qquad \text{and} \qquad
    \varphi_0 \colon \free(\tu_{\cat{C}}) \xrightarrow{\free(\psi_0)} \free\iota(\tu_{\cat{D}}) \xrightarrow{\epsilon_{\tu_{\cat{D}}}} \tu_{\cat{D}}.
  \end{equation*}
  In particular, the category \(\cat{C}^{T}\) of algebras of any idempotent gabi-monad \(T\) on \(\cat{C}\) is skew-closed monoidal, with tensor product given by
  \begin{equation*}
    \boldsymbol{X} \otimes_{T} \boldsymbol{Y} \coloneqq \free^{T}(\U^T \boldsymbol{X} \otimes \U^T \boldsymbol{Y}),
  \end{equation*}
  and the comparison functor \(K \colon \cat{D} \to \cat{C}^T\) is strong monoidal.
\end{theorem}

\begin{proof}
  The first claim follows by combining \cref{thm:recognition-of-reflectiveness} with \cref{prop:gabi-moand-VS-strong-closed}:
    If \(\cat{D}\) is a skew-closed reflective subcategory of \(\cat{C}\), then \(\iota\from \cat{D} \to \cat{C}\) is strong closed by definition,
    and by \cref{thm:recognition-of-reflectiveness} the adjunction \(\free \dashv \iota\) is monadic, so that the idempotent monad \(T = \iota \free\) is a gabi-monad.
    The other direction is immediate.

    Concerning the second claim, the fact that the reflected tensor product induces a skew-closed monoidal structure on \(\cat{D}\) follows from the following computation and \cref{thm:bijection_skew_closed_monoidal_structures}:
    \begin{align*}
        \cat{D}\left(A \otimes_{\cat{D}} B,C\right) & = \cat{D}\left(\free\left(\iota A \otimes_{\cat{C}} \iota B\right),C\right) \cong \cat{C}\left(\iota A \otimes_{\cat{C}} \iota B, \iota C\right) \\
        & \cong \cat{C}\left(\iota A, \left[ \iota B, \iota C \right]_{\cat{C}}\right) \cong \cat{C}\left(\iota A, \iota \left[ B, C \right]_{\cat{D}}\right) \cong \cat{D}\left(A, \left[ B, C \right]_{\cat{D}}\right),
    \end{align*}
    natural in all entries.
    The coevaluation map \(\coev_A^B \colon A \to [B, A \otimes_{\cat{D}} B]_{\cat{D}}\)
    is the unique morphism
    \begin{equation}\label{eq:coev_exp_ideal}
      \begin{tikzcd}[column sep=80pt]
        \iota A & {\iota[B, A \otimes_{\cat{D}} B]_{\cat{D}}} \\
        {[\iota B, \iota A \otimes_{\cat{C}} \iota B]_{\cat{C}}} & {[\iota B, \iota(A \otimes_{\cat{D}} B)]_{\cat{C}}}
        \arrow["{\coev_{\iota A}^{\iota B}}"', from=1-1, to=2-1]
        \arrow["{\iota \coev_A^B}", dotted, from=1-1, to=1-2]
        \arrow["{\psi^{-1}_{B,A \otimes_{\cat{D}} B}}"', to=1-2, from=2-2]
        \arrow["{[\iota B, \eta_{\iota A \otimes_{\cat{C}} \iota B}]}"', from=2-1, to=2-2]
      \end{tikzcd}
    \end{equation}
    The evaluation map \(\ev_A^B \colon [B, A]_{\cat{D}}  \otimes_{\cat{D}} B \to A\)
    is given by the commutativity of the diagram
    \begin{equation}\label{eq:ev_exp_ideal}
      \begin{gathered}
        \begin{tikzcd}[ampersand replacement=\&,cramped]
{[B,A]_{\cat{D}} \otimes_{\cat{D}} B} \arrow[dd, dotted, "{\ev_A^B}"']\arrow[r, equal]\& {\free\left(\iota [B,A]_{\cat{D}} \otimes_{\cat{C}} \iota B \right)} \arrow[d, "{\free(\psi_{B,A} \otimes_{\cat{C}} \iota B)}"]\\
          \& {\free\left([\iota B,\iota A]_{\cat{C}} \otimes_{\cat{C}} \iota B \right)} \arrow[d, "{\free(\ev_{\iota A}^{\iota B})}"]\\
          {A} \& {\free \iota A} \arrow[l, "{\epsilon_A}"]
\end{tikzcd}
      \end{gathered}
    \end{equation}

    The functor \(\iota\) is strong closed by hypothesis, with structure isomorphisms
    \begin{align*}
      \psi_{A,B}  \colon \iota[A,B]_{\cat{D}} \to [\iota A,\iota B]_{\cat{C}}
      \qquad\text{and}\qquad
      \psi_0  \colon \tu_{\cat{C}} \to \iota(\tu_{\cat{D}}).
    \end{align*}
    Hence, by \cref{prop:closed-mon-bij}, \(\iota\) is lax monoidal with
    \[
      \phi_{A,B} \colon \iota A \otimes \iota B \xrightarrow{\iota\coev_A^B \otimes \iota B} \iota[B,A \otimes B] \otimes \iota B \xrightarrow{\psi_{B,A \otimes B} \otimes \iota B} [\iota B, \iota(A \otimes B)] \otimes \iota B \xrightarrow{\ev_{\iota(A \otimes B)}^{\iota B}} \iota(A \otimes B),
    \]
    which coincides with \(\eta_{\iota A \otimes_{\cat{C}} \iota B}\) by the commutativity of
    \[
        \mathscale{0.9}{\begin{tikzcd}[ampersand replacement=\&,column sep=3.15em,row sep=2.25em]
        	{\iota A \otimes_{\cat{C}} \iota B} \&\&\& {\iota [B,A \otimes_{\cat{D}} B] \otimes_{\cat{C}} \iota B} \\
        	{\left[\iota B,\iota A \otimes_{\cat{C}} \iota B\right] \otimes_{\cat{C}} \iota B} \&\& {\left[\iota B,\iota \free\left( \iota A \otimes_{\cat{C}} \iota B\right)\right] \otimes_{\cat{C}} \iota B} \& {\left[\iota B,\iota \left(A \otimes_{\cat{D}} B\right)\right] \otimes_{\cat{C}} \iota B} \\
        	{\iota A \otimes_{\cat{C}} \iota B} \&\& {\iota \free\left( \iota A \otimes_{\cat{C}} \iota B\right)} \& {\iota \left(A \otimes_{\cat{D}} B\right)}
        	\arrow["{\iota\coev_A^B \otimes_{\cat{C}} \iota B}", from=1-1, to=1-4]
        	\arrow["{\coev_{\iota A}^{\iota B} \otimes_{\cat{C}} \iota B}"{description}, from=1-1, to=2-1]
        	\arrow["{\psi_{B,A \otimes_{\cat{D}} B} \otimes_{\cat{C}}\iota B}"{description}, from=1-4, to=2-4]
        	\arrow["{\left[\iota B,\eta_{\iota A \otimes_{\cat{C}} \iota B}\right] \otimes_{\cat{C}} \iota B}", from=2-1, to=2-3]
        	\arrow["{\ev_{\iota A \otimes_{\cat{C}} \iota B}^{\iota B}}"{description}, from=2-1, to=3-1]
        	\arrow[equals, from=2-3, to=2-4]
        	\arrow["{\ev_{\iota \free(\iota A \otimes_{\cat{C}} \iota B)}^{\iota B}}"{description}, from=2-3, to=3-3]
        	\arrow["{\ev_{\iota(A \otimes_{ \cat{D}} B)}^{\iota B}}"{description}, from=2-4, to=3-4]
        	\arrow["{\eta_{\iota A \otimes_{\cat{C}} \iota B}}"', from=3-1, to=3-3]
        	\arrow[equals, from=3-3, to=3-4]
        \end{tikzcd}}
    \]
    and \(\phi_0 \coloneqq \psi_0 \colon \tu_{\cat{C}} \to \iota(\tu_{\cat{D}})\).
    Since \(\free\) is left adjoint to \(\iota\), it is naturally a colax monoidal functor with
    \[
      \mathscale{.9}{
      \varphi_{X,Y} \colon \free\left(X \otimes_{\cat{C}} Y\right) \xrightarrow{\free\left(\eta_X \otimes_{\cat{C}} \eta_Y\right)} \free\left(\iota \free X \otimes_{\cat{C}} \iota \free Y\right) \xrightarrow{\free\left(\phi_{\free X,\free Y}\right)} \free \iota\left(\free X \otimes_{\cat{D}} \free Y\right) \xrightarrow{\epsilon_{\free X \otimes_{\cat{D}} \free Y}}\free X \otimes_{\cat{D}} \free Y}
    \]
    and \(\varphi_0 \colon \free(\tu_{\cat{C}}) \xrightarrow{\free(\psi_0)} \free \iota(\tu_{\cat{D}}) \xrightarrow{\epsilon_{\tu_{\cat{D}}}} \tu_{\cat{D}}\), which is an isomorphism.
    Since \(\phi_{\free X,\free Y} = \eta_{\iota \free X \otimes_{\cat{C}} \iota \free Y}\) and \(\epsilon_{\free X \otimes_{\cat{D}} \free Y} = \epsilon_{\free(\iota \free X \otimes_{\cat{C}} \iota \free Y)}\), we have that \(\varphi_{X,Y} = \free(\eta_X \otimes_{\cat{C}} \eta_Y)\).

    To conclude, observe that the unit coherence of \(K\) is \(\psi_0\), and
    \[
      K(A \otimes_{\cat{D}} B) = K\free(\iota A \otimes \iota B) = \free^T(\iota A \otimes \iota B) = \free^T(\U^TK A \otimes \U^T K B) = KA \otimes_T KB.\qedhere
    \]
\end{proof}

\begin{definition}\label{def:exp-ideal}
  Let \(\cat{C}\) be a skew-closed monoidal category and \(\cat{D}\) a subcategory of \(\cat{C}\) with inclusion functor \(\iota\colon \cat{D} \to \cat{C}\).
  One calls \(\cat{D}\) an \emph{exponential ideal} (in the monoidal sense) if, for every \(X\) in \(\cat{C}\) and \(A\) in \(\cat{D}\),
  the object \([X,\iota A]_{\cat{C}}\) lies in the essential image of \(\iota\);
  i.e., there exists an object \(H_{X,A}\) in \(\cat{D}\) (unique up to isomorphism) such that \([X,\iota A]_{\cat{C}} \cong \iota H_{X,A}\).
\end{definition}

\begin{remark}
  If \((\cat{D},\iota,\free)\) is a reflective subcategory and an exponential ideal in a skew-closed monoidal category \(\cat{C}\), then we call it a \emph{reflective exponential ideal} for the sake of brevity.
  Since \(\cat{D}\) is reflective, being an exponential ideal is equivalent to \(\eta_{[X,\iota A]_{\cat{C}}}\) being an isomorphism for every \(X\) in \(\cat{C}\) and every \(A\) in \(\cat{D}\).
  In this case, one may take \(H_{X,A} = \free[X,\iota A]_{\cat{C}}\).
\end{remark}

The proof of \cref{thm:day-reflection-thm} below follows the same pattern as that of Day's Reflection Theorem~\cite{Day-reflection}
and its skew analogue~\cite{Lack-Street}, with some small adjustments dictated by the perspective we adopted.
It is reported for the convenience of the reader.

\begin{proposition}\label{thm:day-reflection-thm}
  Let \(\cat{C}\) be a skew-closed monoidal category and let \((\cat{D},\iota,\free)\) be a reflective exponential ideal of \(\cat{C}\)
  such that \(\tu_{\cat{C}}\) lies in the essential image of \(\iota\).
  Then \(\cat{D}\) is a skew-closed reflective subcategory of \(\cat{C}\) and is skew-closed monoidal with respect to the reflected tensor product.
  The reflector \(\free \colon \cat{C} \to \cat{D}\) is strong monoidal if and only if \([\eta_X,\iota A]_{\cat{C}} \colon [\iota \free X,\iota A]_{\cat{C}} \to [X,\iota A]_{\cat{C}}\) is an isomorphism for every \(X\) in \(\cat{C}\) and every \(A\) in \(\cat{D}\).
\end{proposition}

\begin{proof}
  The category \(\cat{D}\) is naturally a skew-closed category with respect to
  \[
    [A,B]_{\cat{D}} \coloneqq \free[\iota A,\iota B]_{\cat{C}}
    \qquad \text{and} \qquad
    \tu_{\cat{D}} \coloneqq \free(\tu_{\cat{C}}),
  \]
  and with respect to
  \begin{align*}
    i_A & \coloneqq \left( [\tu_{\cat{D}},A]_{\cat{D}} = \free[\iota \free\tu_{\cat{C}},\iota A]_{\cat{C}} \xrightarrow{\free[\eta_{\tu_{\cat{C}}},\iota A]_{\cat{C}}} \free[\tu_{\cat{C}},\iota A]_{\cat{C}} \xrightarrow{\free i_{\iota A}} \free\iota A \xrightarrow{\epsilon_A} A \right), \\
    j_A & \coloneqq \left(\tu_{\cat{D}} = \free \tu_{\cat{C}} \xrightarrow{F j_{\iota A}} \free[\iota A,\iota A]_{\cat{C}} = [A,A]_{\cat{D}}\right),
  \end{align*}
  and \(\Gamma_{B,C}^A \colon [B,C]_{\cat{D}} \to [[A,B]_{\cat{D}},[A,C]_{\cat{D}}]_{\cat{D}}\) given by
  \[
    \mathscale{.95}{
      \free[\iota B,\iota C]_{\cat{C}} \xrightarrow{\free\Gamma_{\iota B,\iota C}^{\iota A}} \free[[\iota A,\iota B]_{\cat{C}},[\iota A, \iota C]_{\cat{C}}]_{\cat{C}} \xrightarrow{\free\left[\eta^{-1}_{[\iota A,\iota B]_{\cat{C}}},\eta^{\phantom{-1}}_{[\iota A, \iota C]_{\cat{C}}}\right]_{\cat{C}}} \free[\iota \free[\iota A,\iota B]_{\cat{C}},\iota \free[\iota A, \iota C]_{\cat{C}}]_{\cat{C}}}.
  \]
  In particular, \((\cat{D},\iota)\) becomes a skew-closed reflective subcategory with respect to
  \[
    \mathscale{.95}{
      \psi_{A,B} \coloneqq \left(\iota[A,B]_{\cat{D}} = \iota \free[\iota A,\iota B]_{\cat{C}} \xrightarrow{\eta^{-1}_{[\iota A,\iota B]_{\cat{C}}}} [\iota A,\iota B]_{\cat{C}}\right)
      \quad \text{and} \quad
      \psi_0 \coloneqq \left(\tu_{\cat{C}} \xrightarrow{\eta_{\tu_{\cat{C}}}} \iota \free \tu_{\cat{C}} = \iota \tu_{\cat{D}}\right)}.
  \]
  In view of \cref{lem:recognition-of-idempotent-gabi-monads}, we already know that \(\cat{D}\) is a skew-closed monoidal category with respect to the reflected tensor product and that \(\free\) is colax monoidal with structure morphisms \(\free(\eta_X \otimes_{\cat{C}} \eta_Y)\) and \(\epsilon_{\tu_{\cat{D}}} \circ \free(\psi_0)\), and the latter is the identity.

  Now, we claim that \(\free(\eta_X \otimes_{\cat{C}} \eta_Y)\) is an isomorphism for every \(X,Y\) in \(\cat{C}\) if and only if \([\eta_X,\iota A]_{\cat{C}} \colon [\iota \free X,\iota A]_{\cat{C}} \to [X,\iota A]_{\cat{C}}\) is an isomorphism for every \(X\) in \(\cat{C}\) and every \(A\) in \(\cat{D}\).
  Consider the commutative diagram
  \[
    \begin{tikzcd}[ampersand replacement=\&,column sep=7em]
      {\cat{D}\left(\free\left(\iota \free X \otimes_{\cat{C}} Y\right),B\right)} \& {\cat{D}\left(\free\left(X \otimes_{\cat{C}} Y\right),B\right)} \\
      {\cat{C}\left(\iota \free X \otimes_{\cat{C}} Y,\iota B\right)} \& {\cat{C}\left(X \otimes_{\cat{C}} Y,\iota B\right)} \\
      {\cat{C}\left(\iota \free X , \left[Y,\iota B\right]_{\cat{C}}\right)} \& {\cat{C}\left(X , \left[Y,\iota B\right]_{\cat{C}}\right)} \\
      {\cat{C}\left(\iota \free X , \iota \free\left[Y,\iota B\right]_{\cat{C}}\right)} \& {\cat{C}\left(X , \iota \free\left[Y,\iota B\right]_{\cat{C}}\right).}
      \arrow["{\cat{D}\left(\free\left(\eta_X \otimes_{\cat{C}} Y\right),B\right)}", from=1-1, to=1-2]
      \arrow["\cong"', from=1-1, to=2-1]
      \arrow["\cong", from=1-2, to=2-2]
      \arrow["{\cat{C}\left(\eta_X \otimes_{\cat{C}} Y,\iota B\right)}"', from=2-1, to=2-2]
      \arrow["\cong"', from=2-1, to=3-1]
      \arrow["\cong", from=2-2, to=3-2]
      \arrow["{\cat{C}\left(\eta_X , \left[Y,\iota B\right]_{\cat{C}}\right)}"', from=3-1, to=3-2]
      \arrow["{\cat{C}\left(\iota \free X,\eta_{[Y,\iota B]_{\cat{C}}}\right)}"',"\cong", from=3-1, to=4-1]
      \arrow["{\cat{C}\left(X,\eta_{[Y,\iota B]_{\cat{C}}}\right)}","\cong"', from=3-2, to=4-2]
      \arrow["{\cat{C}\left(\eta_X , \iota \free\left[Y,\iota B\right]_{\cat{C}}\right)}"', from=4-1, to=4-2]
    \end{tikzcd}
  \]
  The bottom arrow is an isomorphism because \(\iota\) is fully faithful:
  the composition
  \[
    \cat{D}(\free X, A) \xrightarrow{\ \iota_{\free X,A}\ } \cat{C}(\iota \free X, \iota A) \xrightarrow{\cat{C}(\eta_X, \iota A)} \cat{C}(X, \iota A)
  \]
  sends \(f\) to \(\iota(f) \circ \eta_X\).
  Since \(\iota_{\free X,A}\) is also a bijection, \(\cat{C}(\eta_X, \iota A)\) is one, too.
  Hence, since the Yoneda embedding is conservative, \(\free(\eta_X \otimes_{\cat{C}} Y)\) is an isomorphism for every \(X,Y\) in \(\cat{C}\).

  Now, consider also the commutative diagram
  \[
    \begin{tikzcd}[ampersand replacement=\&,column sep=7em]
      {\cat{D}\left(\free\left(X \otimes_{\cat{C}} \iota \free Y\right),A\right)} \& {\cat{D}\left(\free\left(X \otimes_{\cat{C}} Y\right),A\right)} \\
      {\cat{C}\left(X \otimes_{\cat{C}} \iota \free Y,\iota A\right)} \& {\cat{C}\left(X \otimes_{\cat{C}} Y,\iota A\right)} \\
      {\cat{C}\left(X , \left[\iota \free Y,\iota A\right]_{\cat{C}}\right)} \& {\cat{C}\left(X , \left[Y,\iota A\right]_{\cat{C}}\right).}
      \arrow["{\cat{D}\left(\free\left(X \otimes_{\cat{C}} \eta_Y\right),A\right)}", from=1-1, to=1-2]
      \arrow["\cong"', from=1-1, to=2-1]
      \arrow["\cong", from=1-2, to=2-2]
      \arrow["{\cat{C}\left(X \otimes_{\cat{C}} \eta_Y,\iota A\right)}"', from=2-1, to=2-2]
      \arrow["\cong"', from=2-1, to=3-1]
      \arrow["\cong", from=2-2, to=3-2]
      \arrow["{\cat{C}\left(X , \left[\eta_Y,\iota A\right]_{\cat{C}}\right)}"', from=3-1, to=3-2]
    \end{tikzcd}
  \]
  The bottom arrow is an isomorphism for every \(X,Y\) in \(\cat{C}\) and every \(A\) in \(\cat{D}\) if and only if so is the top one. Therefore, as above, \(\free(X \otimes_{\cat{C}} \eta_Y)\) is an isomorphism for every \(X,Y\) in \(\cat{C}\) if and only if \([\eta_Y,\iota A]_{\cat{C}} \colon [\iota \free Y,\iota A]_{\cat{C}} \to [Y,\iota A]_{\cat{C}}\) is an isomorphism for every \(Y\) in \(\cat{C}\) and every \(A\) in \(\cat{D}\).

  Finally, we conclude by observing that
  \[
    \free(\eta_X \otimes_{\cat{C}} \eta_Y) = \free(\eta_X \otimes_{\cat{C}} \iota \free Y) \circ \free(X \otimes_{\cat{C}} \eta_Y). \qedhere
  \]
\end{proof}

\begin{corollary}\label{cor:refl_strong_for_sym_mon}
  Let \(\cat{C}\) be a closed braided monoidal category and let \((\cat{D},\iota,\free)\) be a reflective exponential ideal of \(\cat{C}\)
  such that \(\tu_{\cat{C}}\) lies in the essential image of \(\iota\).
  Then the reflector \(\free \from \cat{C} \to \cat{D}\) is strong monoidal.
\end{corollary}

\begin{proof}
    In the proof of \cref{thm:day-reflection-thm}, we proved that \(\free(\eta_X \otimes_{\cat{C}} Y)\) is an isomorphism for every \(X,Y\) in \(\cat{C}\). If \(\cat{C}\) is also braided with braiding \(\mathfrak{c}\), then this implies that
    \[
      \free(X \otimes_{\cat{C}} \eta_Y) = \free(\mathfrak{c}^{-1}_{X,\iota \free Y}) \circ \free(\eta_Y \otimes_{\cat{C}} X) \circ \free(\mathfrak{c}_{X,Y})
    \]
    is an isomorphism for every \(X,Y\) in \(\cat{C}\), and so \(\free(\eta_X \otimes_{\cat{C}} \eta_Y)\) also has to be an isomorphism.
\end{proof}

\begin{remark}
  Note that there exists a notion of braiding for skew-monoidal categories, see~\cite[Definition~2.2]{bourke20:braid}.
  Instead of a natural isomorphism in two variables, one requires a natural isomorphism
  \(
    (X \otimes A) \otimes B \to (X \otimes B) \otimes A
  \)
  satisfying certain coherence diagrams.
  In this more general setting, \cref{cor:refl_strong_for_sym_mon} does not hold:
  let \(\cat{C} \defeq \{\,0 \to 1 \to 2\,\}\) be a poset category.
  Notice that all isomorphisms in \(\cat{C}\) are automatically identities.
  For \(X,Y \in \cat{C}\), define a skew-monoidal structure by
  \[
    X \otimes Y \defeq \begin{cases} X,&Y=2,\\ 0,&Y\neq 2. \end{cases}
  \]
  The monoidal unit is \(2\).
  The structure morphisms
  \[
    \rho_X\from X \xrightarrow{=} X \otimes 2
    \qquad\text{and}\qquad
    \alpha_{X,Y,Z}\from (X \otimes Y) \otimes Z \xrightarrow{=} X \otimes (Y \otimes Z)
  \]
  are identities,
  and \(\lambda_X\from 2 \otimes X \to X\) is given by the unique morphism.
  Notably, \(\lambda\) is not an identity, so the skew-monoidal structure is not left normal.
  The skew-monoidal braiding \((X \otimes A) \otimes B \xrightarrow{=} (X \otimes B) \otimes A\) is also given by the identity.
  Overall, \(\cat{C}\) is a braided skew-monoidal category.
  Notice that there does not exist a braiding of the form \(\sigma_{XY}\from X \otimes Y \xrightarrow{=} Y \otimes X\), as for example \(1 \otimes 2 = 1\), but \(2 \otimes 1 = 0\).
  It is easy to see that the following internal hom turns \(\cat{C}\) into a skew-closed monoidal category:
  \[
    [X,Y] \defeq \begin{cases} Y,&X=2,\\ 2,&X\neq 2. \end{cases}
  \]

  Consider the full subcategory \((\cat{D} \defeq \{0 \to 2\},\, \iota\from \cat{D} \to \cat{C})\).
  This is a reflective subcategory: the left adjoint \(\free\from \cat{C}\to \cat{D}\) of \(\iota\) is given by \(\free(0) \defeq 0\) and \(\free(1) \defeq \free(2) \defeq 2\).
  In fact, \(\cat{D}\) is a reflective exponential ideal: for \(C \in \cat{C}\) and \(D \in \cat{D}\), the internal hom \([C,D]\) is either \(0\) or \(2\), both of which are in \(\cat{D}\).
  However, the reflector \(\free\) is not strong monoidal:
  \[
    \free(1 \otimes 1) = \free(0) = 0 \neq 2 = 2 \otimes 2 = \free(1) \otimes \free(1).\qedhere
  \]
\end{remark}

The subsequent \cref{prop:closed-subcat-is-of-same-type} is a useful consequence of \cref{thm:day-reflection-thm}, \cref{cor:refl_strong_for_sym_mon}, and the fact that a fully faithful functor is conservative.

\begin{proposition}\label{prop:closed-subcat-is-of-same-type}
  Let \(\cat[C]\) be a skew-closed category and let \((\cat[D],\iota)\) be a full subcategory which is also a skew-closed subcategory.
  If \(\cat[C]\) is left or right normal, then \(\cat[D]\) is left or right normal, respectively.
  If, additionally, \(\cat{C}\) is closed braided monoidal and \((\cat{D},\iota,\free)\) is a reflective exponential ideal, then \(\cat{D}\) is closed monoidal.
\end{proposition}

\begin{proof}
    Write \(\varphi_0\) and \(\varphi_{A,B}\) for the structure isomorphisms of the strong closed functor \(\iota\).
    Suppose that \(\cat[C]\) is left normal and consider the compositions
    \begin{gather*}
        \cat[D](A,B) \xrightarrow{\iota_{A,B}} \cat[C](\iota A,\iota B) \xrightarrow{\widehat\jmath_{\iota A,\iota B}} \cat[C](\tu_{\cat[C]},[\iota A,\iota B]_{\cat[C]}) \qquad \text{and} \\
        \cat[D](A,B) \xrightarrow{\widehat\jmath_{A,B}} \cat[D](\tu_{\cat[D]},[A,B]_{\cat[D]}) \xrightarrow{\iota_{\tu_{\cat[D]},[A,B]_{\cat[D]}}} \cat[C](\iota \tu_{\cat[D]},\iota[A,B]_{\cat[D]}) \xrightarrow{\cat[C](\varphi_0,\varphi_{A,B})} \cat[C](\tu_{\cat[C]},[\iota A,\iota B]_{\cat[C]}).
    \end{gather*}
    The first one maps \(f \colon A \to B\) in \(\cat[D]\) to \([\iota f,\iota B]_{\cat[C]} \circ j_{\iota B}\). The second one maps the same \(f\) to
    \begin{align*}
        \varphi_{A,B} \circ \iota\left([f,B]_{\cat[D]}\right) \circ \iota(j_B) \circ \varphi_0 & = [\iota f,\iota B]_{\cat[C]} \circ \varphi_{B,B} \circ \iota(j_B) \circ \varphi_0 = [\iota f,\iota B]_{\cat[C]} \circ j_{\iota B}.
    \end{align*}
    Thus, they are equal, and so \(\widehat\jmath_{A,B}\) is a bijection when \(\widehat\jmath_{\iota A,\iota B}\) is so.

    If \(\cat[C]\) is right normal, then
    \[\iota\left(i_A\right) = i_{\iota A} \circ [\varphi_0,\iota A]_{\cat[C]} \circ \varphi_{\tu_{\cat[D]},A}\]
    is an isomorphism and so is \(i_A\), since \(\iota\) is conservative.

    Finally, suppose that \(\cat{D}\) is a reflective exponential ideal and that \(\cat{C}\) is closed braided monoidal.
    Let \(\free\from \cat{C} \to \cat{D}\) be the left adjoint to the inclusion \(\iota\).
    The isomorphism \(\varphi_0 \colon \tu_{\cat{C}} \to \iota\tu_{\cat{D}}\) ensures that the unit condition in \cref{thm:day-reflection-thm} is satisfied.
    By \cref{thm:bijection_skew_closed_monoidal_structures}, a skew-closed category is associative normal if and only if the associated skew-monoidal structure is.
    The associator for \(\otimes_{\cat{D}}\) is
    \begin{align*}
      (X \otimes_{\cat{D}} Y) \otimes_{\cat{D}} Z
      & = \free(\iota \free(\iota X \otimes \iota Y) \otimes \iota Z)
        \cong \free((\iota X \otimes \iota Y) \otimes \iota Z) \\
      & \xrightarrow{\free\alpha} \free(\iota X \otimes (\iota Y \otimes \iota Z))
        \cong \free(\iota X \otimes \iota \free(\iota Y \otimes \iota Z))\\
      & = X \otimes_{\cat{D}} (Y \otimes_{\cat{D}} Z),
    \end{align*}
    where the unnamed isomorphisms follow by \cref{thm:day-reflection-thm,cor:refl_strong_for_sym_mon}.
    Hence, \(\cat{D}\) is closed monoidal.
\end{proof}

Summing up, we have proven the following general result.

\begin{theorem}\label{thm:refl_subcats_give_gabi}
  Let \(\cat{C}\) be a closed braided monoidal category and let \((\cat{D},\iota,\free)\) be a reflective exponential ideal of \(\cat{C}\)
  such that \(\tu_{\cat{C}}\) lies in the essential image of \(\iota\).
  Then \(\cat{D} \simeq \cat{C}^T\) for the idempotent gabi-monad \(T = \iota \free\).
  Moreover, the reflector \(\free\) is strong monoidal, \(\cat{C}^T\) is closed monoidal, and \(T\) is a normal gabi-monad.
\end{theorem}

\begin{proof}
  In view of \cref{thm:day-reflection-thm,cor:refl_strong_for_sym_mon}, \(\cat{D}\) is a reflective skew-closed subcategory of \(\cat{C}\) and is in fact skew-closed monoidal with the reflected tensor product.
  Moreover, the reflector \(\free \colon \cat{C} \to \cat{D}\) is strong monoidal.
  By \cref{lem:recognition-of-idempotent-gabi-monads}, the associated idempotent monad \(T = \iota \free\) is a gabi-monad whose category of algebras is skew-closed monoidal.
  Finally,
  \cref{lem:recognition-of-idempotent-gabi-monads,thm:gabi-bijection}, \cref{prop:closed-subcat-is-of-same-type},
  and the equivalence \(\cat{C}^T \simeq \cat{D}\) show that \(\cat{C}^T\) is closed monoidal.
  Hence, \(T\) is a normal gabi-monad.
\end{proof}

\subsection{Torsion submodules}\label{sec:torsion-submodules}

In the following, \(R\) denotes an integral domain with unit \(1\).

\begin{definition}\label{def:torsion}
  An element \(m\in M\) of an \(R\)-module \(M \in \hmodM[R]\) is \emph{torsion} if there exists some non-zero \(r \in R\) such that \(rm =0\).

  We write \(\Tor(M)\) for the set of torsion elements of \(M\) and call \(M\) a \emph{torsion module} if \(M = \Tor(M)\).
  Similarly, we say \(M\) is \emph{torsion-free} if \(\Tor(M)=\{0\}\).

  We denote by \(\hmodM[R][][\mathrm{tf}]\) the full subcategory of \(\hmodM[R]\) consisting of torsion-free modules.
\end{definition}

Since \(R\) is an integral domain, the set \(\Tor(M)\) of torsion elements of a module \(M\) forms a submodule of \(M\) and the quotient \(M/\Tor(M)\) is torsion-free.
Thus, we obtain a functor \(L \from \hmodM[R] \to \hmodM[R][][\mathrm{tf}]\).

\begin{lemma}\label{lem:torsion-free-is-a-reflexive-subcategory}
  The category \(\hmodM[R][][\mathrm{tf}]\) is a reflective subcategory of \(\hmodM[R]\) with reflector
  \[
    L \from  \hmodM[R] \to \hmodM[R][][\mathrm{tf}],\qquad L(M) = M/\Tor(M).
  \]
\end{lemma}
\begin{proof}
  A direct computation shows that \(L\) is well-defined.

  To show that \(L\) is left adjoint to the canonical inclusion \(\forget \from \hmodM[R][][\mathrm{tf}] \to \hmodM[R]\), we first consider an \(R\)-linear map \(f\from M \to N\) with target a torsion-free module.
  As the image \(f(m)\) of any torsion element \(m \in M\) must again be a torsion element, we obtain a commuting diagram
  \begin{equation*}
    \begin{tikzcd}[ampersand replacement=\&,cramped]
      M \&\& N \\
      \& {M/\Tor(M)}
      \arrow["f", curve={height=-18pt}, from=1-1, to=1-3]
      \arrow[two heads, from=1-1, to=2-2]
      \arrow["{\exists! f'}"', dashed, from=2-2, to=1-3]
    \end{tikzcd}
  \end{equation*}
  Thus, there exists a natural isomorphism
  \begin{equation*}
    \phi_{M,N} \from \Hom_{\hmodM[R]}(M,\forget N) \to \Hom_{\hmodM[R][][\mathrm{tf}]}(LM,N), \qquad \phi_{M,N}(f)[m] = f(m). \qedhere
  \end{equation*}
\end{proof}

Concerning the next definition, see for example~\cite[Theorem~4.2]{chase61:direc}.

\begin{definition}\label{def:Prufer-domain}
  An integral domain \(R\) is \emph{Pr\"ufer} if the tensor product of two torsion-free modules is torsion-free
\end{definition}

By \cref{thm:recognition-of-reflectiveness}, the torsion-free modules over \(R\) coincide with the algebras of the monad \(T\from \hmodM[R] \to \hmodM[R]\), \(M\mapsto M/\Tor(M)\).
Its multiplication is given by the identity map \(\mu_{M}=\id_{TM} \from T^{2}(M)\to T(M)\) and the unit is the canonical projection \(\eta_{M} = \pi_{M} \from M \to T(M)\).

\begin{theorem}\label{thm:normal-gaby-monad-torsion}
  The monad \(T\from \hmodM[R] \to \hmodM[R]\), \(M\mapsto M/\Tor(M)\) is a normal idempotent gabi-monad whose category of algebras is closed monoidal.
  Its structure morphisms are given for all \(M, N \in\hmodM[R]\) by
  \begin{equation}\label{eq:structure-morphisms-pruefer-example}
    \begin{gathered}
      s_{M,N} \from T \Hom_{\hmodM[R]}(TM,N) \to \Hom_{\hmodM[R]}(M,TN), \qquad s_{M,N}(f) = \pi_{N}f \pi_{M}, \\
      \act_{R} = \id_{R} \from T(R) = R \to R.
    \end{gathered}
  \end{equation}

  Moreover, \(T\) is Hopf if and only if \(R\) is a Pr\"ufer domain.
\end{theorem}
\begin{proof}
  We begin by showing that the Eilenberg--Moore category \(\hmodM[R][][ \mathrm{tf}]\) of \(T\) is a closed subcategory of \(\hmodM[R]\) by proving the stronger condition that for every module \(M\) and torsion-free module \(N\) the hom-space \(\Hom_{\hmodM[R]}(M,N)\) is torsion-free.

  Consider a  map \(f\from M \to N\) such that there is an \(r\in R\setminus\{0\}\) with \(rf=0\).
  Thus, for any \(m\in M\), we have \(0 = (rf)m = r(f(m))\).
  Since \( N\) is torsion-free and \(0\neq r\), it follows that  \(f(m) = 0\) for all \(m\in M\).
  In other words, \(f\) is the zero morphism and \(\Hom_{\hmodM[R]}(M,N)\) is torsion-free.

  As a consequence, \(\hmodM[R][][\mathrm{tf}]\) is a reflective exponential ideal in the closed symmetric monoidal category \(\hmodM[R]\) and therefore \cref{prop:closed-subcat-is-of-same-type} entails that \(\hmodM[R][][\mathrm{tf}]\) is a normal closed subcategory of \(\hmodM[R]\).
  By \cref{thm:refl_subcats_give_gabi}, \(T\) is a normal idempotent gabi-monad and its Eilenberg--Moore category \(\hmodM[R][][ \mathrm{tf}]\) is closed monoidal with reflected tensor product
  \begin{equation*}
    M\otimes_{\mathrm{tf}} N \eqdef (M\otimes_{R}N)/\Tor(M\otimes_{R}N)
  \end{equation*}
  for all \(M,N \in \hmodM[R][][ \mathrm{tf}]\).
  Its structure morphisms, stated in \eqref{eq:structure-morphisms-pruefer-example}, are obtained via a direct computation using \eqref{eq:SfromAct} and the fact that any \(M\in \hmodM[R]\) has at most one \(T\)-algebra structure, respectively.

  Now \(T\) being Hopf implies that the inclusion \(\U^{T}\from \hmodM[R][][ \mathrm{tf}] \to \hmodM[R]\) is strict monoidal.
  Therefore, the tensor product of torsion-free modules must be torsion-free and \(R\) is a Pr\"ufer domain.

  Conversely, if \(R\) is Pr\"ufer, the inclusion \(\U^{T}\from \hmodM[R][][ \mathrm{tf}] \to \hmodM[R]\) is strict monoidal and \(T\) therefore is a bimonad.
  Theorem~3.6 of~\cite{BruguieresLackVirelizier} now shows that \(T\) is a Hopf monad.
\end{proof}

The following result is proven in \cite[Theorem~2.6]{brewer06:multip}.
\begin{proposition}\label{rmk:integral-not-pruefer}
  An integral domain \(R\) is a field if and only if \(R[x]\) is a Prüfer domain.
\end{proposition}

By combining \cref{thm:normal-gaby-monad-torsion} with \cref{rmk:integral-not-pruefer}, we obtain a first set of examples of non-Hopf but normal gabi-monads.
\begin{corollary}\label{cor:non-Hopf}
  Consider an integral domain \(R\).
  The normal gabi-monad
  \begin{equation*}
    T \from \hmodM[ R[x]] \to \hmodM[ R[x]], \qquad M \mapsto M/\Tor(M)
  \end{equation*}
  is Hopf if and only if \(R\) is a field.
\end{corollary}

\begin{remark}\label{rmk:torsion}
  The situation of \cref{lem:torsion-free-is-a-reflexive-subcategory} may be generalised in the following way.
  Let \(\cat{A}\) be an abelian category, and \(\cat{T}\) and \(\cat{\free}\) full additive subcategories.
  The pair \((\cat{T}, \cat{\free})\) is called a \emph{torsion theory} if the following conditions hold:
  \begin{itemize}[leftmargin=0.8cm]
    \item \(\cat{A}(T,\free)=0\) for all \(T \in \cat{T}\) and \(\free \in \cat{\free}\).
    \item If \(\cat{A}(T,Y) = 0\) for all \(T \in \cat{T}\), then \(Y \in \cat{\free}\).
    \item If \(\cat{A}(X,\free) = 0\) for all \(\free \in \cat{\free}\), then \(X \in \cat{T}\).
    \item For all \(A \in \cat{A}\), there exists some subobject \(T \subseteq A\) with \(T \in \cat{T}\) and \(A/T \in \cat{\free}\).
  \end{itemize}

  Given \(A \in \cat{A}\), it is well known that the subobject \(T\) satisfying the last condition is unique.
  In particular, the torsion-free subcategory is reflective, with left adjoint to the inclusion given by
  \[
    L\from \cat{A} \to \cat{\free},
    \qquad A \mapsto A/T,\ \text{for}\ T \subseteq A, T \in \cat{T},
  \]
  and so the resulting monad \(\iota L\) on \(\cat{A}\) is idempotent.
  If \(\cat{A}\) is closed braided monoidal and \(\cat{\free}\) is a reflective exponential ideal containing \(\tu_{\cat{A}}\), then \(\iota L\) is even a normal gabi-monad by \cref{thm:refl_subcats_give_gabi}.
  However, as seen in \cref{thm:normal-gaby-monad-torsion}, in general \(\iota L\) will not be Hopf.
\end{remark}

\begin{remark}\label{rmk:torsiontu}
  In the setting of \cref{rmk:torsion}, note that \(\cat{\free}\) being an exponential ideal is equivalent to \(\cat{T}\) being a monoidal ideal,
  in the sense that \(T \in \cat{T}, A \in \cat{A} \implies T \otimes A \in \cat{T}\):
  indeed, for \(\free \in \cat{\free}\) we have
  \[
    \cat{A}(T,[A,\free]) \cong \cat{A}(T \otimes A, \free) = 0.
  \]

  The condition that \(\tu \defeq \tu_{\cat{A}} \in \cat{\free}\) is equivalent to the fact that \(\tu\) has no nonzero torsion subobject.
  That is, the unique \(T \subseteq \tu\) such that \(T \in \cat{T}\) and \(\tu/T \in \cat{\free}\) is \(0\);
  see for example~\cite[Proposition~VI.1.4]{assem06:elemen}.
\end{remark}

\subsection{Preorders}\label{sec: preorders}
The present example is inspired by~\cite[Example~3.9]{Garner-Lack}.
Definitions and results are well known, and they are reported here mainly for the convenience of the reader.

\begin{definition}\label{def:digraph}
  A \emph{digraph} (or \emph{directed graph}) \(D\) consists of a set of \emph{vertices} \(V\) and a set of \emph{arrows} \(A \subseteq V \times V\).
  A \emph{morphism of digraphs} from a digraph \(D = (V,A)\) to a digraph \(D'= (V',A')\) is a pair of functions \(\varphi_V \colon V \to V'\) and \(\varphi_A \colon A \to A'\) such that the following diagram commutes
  \begin{equation}\label{eq:mor-of-refl-dig}
    \begin{tikzcd}
      A & {V \times V} \\
      {A'} & {V' \times V'}
      \arrow["\subseteq", from=1-1, to=1-2]
      \arrow["{\varphi_A}"', from=1-1, to=2-1]
      \arrow["{\varphi_V \times \varphi_V}", from=1-2, to=2-2]
      \arrow["\subseteq"', from=2-1, to=2-2]
    \end{tikzcd}
  \end{equation}
  We denote the category of digraphs by \(\plgraph\).
\end{definition}

Note that our terminology here is non-standard:
usually it is required that a directed graph \((V,A)\) does not contain any self-loops;
i.e., \(A \subseteq (V \times V) \setminus \{(v,v) \mid v\in V\}\).
However, we are most interested in the case where all vertices admit loops.

\begin{definition}\label{def:morphism-of-digraphs}
  A digraph \((V,A)\) is \emph{reflexive} if \((v,v)\in A\) for all vertices \(v\in V\);
  that is, there is an inclusion \(i\from V \to A, v \mapsto (v,v)\).
  The full subcategory of \(\plgraph\) whose objects are reflexive digraphs will be denoted by \(\rgraph\).
\end{definition}
Consider a digraph \((V,A)\).
The canonical projections \(\mathfrak{s}, \mathfrak{t} \from A \to V\), given by  \(\mathfrak{s} (v,w) = v\) and \(\mathfrak{t}(v,w) = w\), determine the \emph{source} and \emph{target} of the edge \((v,w)\).
Subsequently, we call \(\mathfrak{s}\) and \(\mathfrak{t}\) the \emph{source} and \emph{target maps} of \((V,A)\), respectively.

\begin{definition}\label{def:preorder}
  Denote by \(A \mathop{_{\mathfrak{s}}\times_{\mathfrak{t}}} A\) the pullback of \(A\) along \(\mathfrak{s}\) and \(\mathfrak{t}\).
  A \emph{preorder} is a reflexive digraph with an additional \emph{composition map}
  \[
    c \colon A \mathop{_{\mathfrak{s}}\times_{\mathfrak{t}}} A \to A,
  \]
  such that
  \begin{gather*}
    \mathfrak{s}(c(g,f)) = \mathfrak{s}(f), \qquad \mathfrak{t}(c(g,f)) = \mathfrak{t}(g), \\
    c(f,i(\mathfrak{s}(f))) = f = c(i(\mathfrak{t}(f)),f), \qquad c(h,c(g,f)) = c(c(h,g),f). \qedhere
  \end{gather*}
\end{definition}

\begin{remark}
  A preorder in the sense of \cref{def:preorder} is simply a category internal to the category of sets,
  such that the source and target maps are jointly monic, ensuring there is at most one morphism between any two objects.
  That is, it coincides with the usual definition of a preorder as a reflexive and transitive relation on a set.
\end{remark}

Instead of a structure, being a preorder can also be understood as a certain completeness property for reflexive digraphs.
More precisely, while reflexive digraphs correspond to reflexive relations, preorders are the reflexive and transitive relations.
This leads to the following observation.

\begin{lemma}\label{lem:existence-and-uniqueness-of-preorders}
  Any reflexive digraph \((V,A)\) admits at most one structure of a preorder.
  Moreover, \((V,A)\) is a preorder if and only if for every pair of arrows
  \((w,x)\in A\) and \((v,w)\in A\) we have \((v,x)\in A\).
\end{lemma}

Suppose that \((V,A,c)\) and \((V',A', c')\) are preorders.
By \cref{lem:existence-and-uniqueness-of-preorders}, any morphism \((\varphi_{V}, \varphi_{A})\) between the underlying reflexive digraphs \((V,A)\) and \((V', A')\) is automatically compatible with the composition maps \(c\) and \(c'\).
Thus,  the category \(\preord\) of preorders and their morphisms is a full subcategory of the category \(\rgraph\).
The inclusion functor \(\forget \colon \preord \to \rgraph\) simply forgets the composition map \(c\);
we shall often suppress it in notation.

\begin{remark}\label{rmk:reflector}
  To any reflexive digraph \(D = (V,A)\), we associate its \emph{transitive closure} \(R(D)\).
  It is a preorder with the same set of vertices \(V\) and an arrow from a vertex \(v\) to a vertex \(w\) if and only if there is a path from \(v\) to \(w\).

  Moreover, we can extend any morphism of reflexive digraphs \(\varphi \colon D \to D'\) to a morphism of preorders \(R(\varphi) \from R(D) \to R(D')\) as follows.
  On vertices we have \((R\varphi)_{V} \defeq \varphi_{V} \from V \to V'\).
  Given an arrow \((v,w) \in R(D)\), one easily verifies that \((\varphi_{V}(v), \varphi_{V}(w))\) is also an arrow in \(R(D')\) and we set \((R\varphi)_{A}(v,w) \eqdef (\varphi_{V}(v), \varphi_{V}(w))\).

  In particular, taking the transitive closure defines a functor \(R \from \rgraph \to \preord\).
\end{remark}

\begin{example}
  If we represent a reflexive digraph \(D\) in the plane, then taking its transitive closure \(R(D)\) amounts to adding edges for all paths in the graph as can be observed below.
  \medskip

  
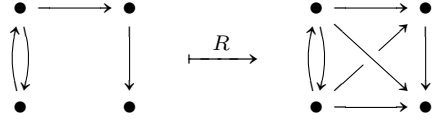
\begin{equation*}
    \begin{gathered}
      \begin{tikzcd}[ampersand replacement=\&,cramped,row sep=25pt,column sep=normal,every matrix/.append style={inner sep=0pt}]
{\bullet} \arrow[r]\arrow[loop, in=125, out=55, distance=7mm,overlay]\arrow[d, bend left=12]\& {\bullet} \arrow[loop, in=125, out=55, distance=7mm,overlay]\arrow[d]\\
        {\bullet} \arrow[u, bend left=12]\arrow[loop, in=305, out=235, distance=7mm,overlay]\& {\bullet} \arrow[loop, in=305, out=235, distance=7mm,overlay]
\end{tikzcd}
    \end{gathered}
    \quad \xmapsto{\;\;\;R\;\;\;} \quad
    \begin{gathered}
      \begin{tikzcd}[ampersand replacement=\&,cramped,row sep=25pt,column sep=normal,every matrix/.append style={inner sep=0pt}]
{\bullet} \arrow[r]\arrow[loop, in=125, out=55, distance=7mm,overlay]\arrow[d, bend left=12]\& {\bullet} \arrow[loop, in=125, out=55, distance=7mm,overlay]\arrow[d]\\
        {\bullet} \arrow[u, bend left=12]\arrow[r]\arrow[ur]\arrow[loop, in=305, out=235, distance=7mm,overlay]\& {\bullet} \arrow[loop, in=305, out=235, distance=7mm,overlay]\arrow[from=1-1, to=2-2, crossing over]
\end{tikzcd}
    \end{gathered}
  \end{equation*}
\end{example}

Note that we can define natural transformations \(\{\eta_{D}\from D \to \forget R(D)\}_{D \in \rgraph}\), given by the canonical inclusion, and \(\{\epsilon_{P} \from R\forget(P) \to P\}_{P\in \preord}\), given by the identity.
This yields an adjunction \(\adj{R}{\forget}{\rgraph}{\preord}\).

\begin{proposition}\label{prop:preord-refl-rgraph}
  The category \(\preord\) is a reflective subcategory of \(\rgraph\).
\end{proposition}

We will now establish a closed monoidal structure on the category \(\rgraph\).

\begin{definition}\label{def:cartesian-product-and-internal hom}
  Consider two reflexive digraphs \(D\) and \(D'\).
  \begin{enumerate}[label=(\arabic*),leftmargin=0.8cm]
    \item Their \emph{cartesian product} \(D \square D'\) has \(V \times V'\) as its set of vertices and \(\big((v,v'),(w,w')\big)\) is an arrow if and only if
    \begin{itemize}
      \item \(v=w\) and \((v',w')\in A'\) is an arrow of \(D'\), or
      \item \(v'= w'\) and \((v,w)\in A\) is an arrow of \(D\).
    \end{itemize}
    \item The \emph{exponential digraph} \([D, D']\) has \(\rgraph(D,D')\) as its set of vertices.  Given two graph homomorphisms \(\varphi, \psi \in \rgraph(D,D')\), we have an arrow \((\varphi, \psi)\) of \([D, D']\) if and only if for all \(v\in V\) the tuple \((\varphi(v), \psi(v))\in A'\) is an arrow of \(D'\). \qedhere
  \end{enumerate}
\end{definition}

It is well known that the cartesian product defines a symmetric monoidal structure on the category of reflexive digraphs, see for example \cite{ImrichPeterin} or, more generally, \cite[\S4.1]{HammackImrichKlavzar}, and \cite{BoikoCunoImrichLehnerWoestijne}, as well as \cite{Feigenbaum}.

\newcommand{\vwrefldig}{\begin{tikzcd}[cramped, column sep=tiny, ampersand replacement=\&]v \& w
    \arrow[from=1-1, to=1-1, loop, shift right, in=160, out=200, distance=5mm]
    \arrow[from=1-1, to=1-2]
    \arrow[from=1-2, to=1-2, loop, shift right, in=340, out=20 , distance=5mm]\end{tikzcd}}
\newcommand{\bulletrefldig}{\begin{tikzcd}[every cell/.append style={inner sep=3pt}, ampersand replacement=\&]
    \bullet
    \arrow[loop, shift right, in=160, out=200, distance=5mm]
  \end{tikzcd}}
\begin{example}\label{ex:rgraph-mon}
  Consider the reflexive digraph \(D=\!\!\vwrefldig\!\!\).
  Taking the cartesian product of \(D\) with itself yields
  
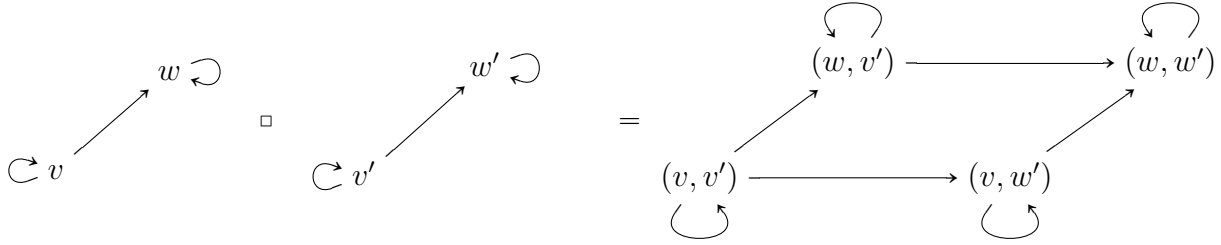
\begin{equation}\label{eq:tp-of-interval-graphs}
    \begin{gathered}
      \begin{tikzcd}[ampersand replacement=\&,cramped,row sep=25pt,column sep=normal,every matrix/.append style={inner sep=0pt}]
\& {w} \arrow[loop, in=340, out=20, distance=5mm]\\
        {v} \arrow[ur]\arrow[loop, in=160, out=200, distance=5mm]\&
\end{tikzcd}
    \end{gathered} \quad \square \quad \begin{gathered}
      \begin{tikzcd}[ampersand replacement=\&,cramped,row sep=25pt,column sep=normal,every matrix/.append style={inner sep=0pt}]
\& {w'} \arrow[loop, in=340, out=20, distance=5mm]\\
        {v'} \arrow[ur]\arrow[loop, in=160, out=200, distance=5mm]\&
\end{tikzcd}
    \end{gathered}
    \qquad =
    \begin{gathered}
      \begin{tikzcd}[ampersand replacement=\&,cramped,row sep=24pt,column sep=20pt,every matrix/.append style={inner sep=0pt}]
\& {(w,v')} \arrow[loop, in=125, out=55, distance=7mm]\arrow[rr]\& \& {(w,w')} \arrow[loop, in=125, out=55, distance=7mm]\\
        {(v,v')} \arrow[ur]\arrow[loop, in=305, out=235, distance=7mm]\arrow[rr]\& \& {(v,w')} \arrow[ur]\arrow[loop, in=305, out=235, distance=7mm]\&
\end{tikzcd}
    \end{gathered}
  \end{equation}
  \smallskip

  Note that there are three endomorphisms of \(D\), determined by the maps
  \[
    \mathrm{pr}_{v}, \mathrm{pr}_{w}, \id \from\{v,w\} \to\{v,w\},
  \]
  \begin{align*}
    &\mathrm{pr}_{v}(v) = v, &&\mathrm{pr}_{w}(v)= w, &&&\id(v) = v, \\
    &\mathrm{pr}_{v}(w) = v, &&\mathrm{pr}_{w}(w)= w, &&&\id(w) = w.
  \end{align*}
  Thus, we have
  \begin{equation*}
    [D,D] \quad = \begin{gathered}\begin{tikzcd}[ampersand replacement=\&,cramped]
      {\mathrm{pr}_v} \&\& \id \\
      \\
      {\mathrm{pr}_w}
      \arrow[from=1-1, to=1-1, loop, in=55, out=125, distance=10mm]
      \arrow[from=1-1, to=1-3]
      \arrow[from=1-1, to=3-1]
      \arrow[from=1-3, to=1-3, loop, in=35, out=325, distance=10mm]
      \arrow[from=1-3, to=3-1]
      \arrow[from=3-1, to=3-1, loop, in=305, out=235, distance=10mm]
    \end{tikzcd}\end{gathered}\qedhere
  \end{equation*}
\end{example}

\begin{proposition}\label{prop:RGraph_symm_closed_mon}
  The category \(\rgraph\) becomes symmetric closed monoidal when endowed with the cartesian product as tensor product, the digraph \(\!\!\bulletrefldig\!\!\) as monoidal unit, and taking exponential digraphs as internal hom.

  Further, \(\preord\) is a reflective exponential ideal of \(\rgraph\), and so a closed subcategory.
\end{proposition}
\begin{proof}
  By adapting \cite[Proposition~2.5]{Kapulkin-Kershaw}, \(\rgraph\) becomes closed symmetric monoidal when endowed with the cartesian product and the internal hom given by taking exponential digraphs. Therefore, we only need to verify that \(\preord\) is an exponential ideal.

  First, note that \(\!\!\bulletrefldig\!\!\) is a preorder.
  Given a reflexive digraph \(D\), a preorder \(D'\), and a sequence of morphisms \(\varphi_{0}, \dots , \varphi_{n} \from D \to D'\) such that \((\varphi_{i}, \varphi_{i+1})\) is an arrow in \([D, D']\), for any \(v \in V\) there is a path \((\varphi_{0}(v), \dots , \varphi_{n}(v))\) in \(D'\).
  As \(D'\) is a preorder, this implies that \((\varphi_{0}(v), \varphi_{n}(v))\) is also an arrow in \(D'\) and thus \((\varphi_{0}, \varphi_{n})\) is an arrow in \([D,D']\).
  Thus, we conclude by \cref{thm:day-reflection-thm,prop:closed-subcat-is-of-same-type}.
\end{proof}

Combining \cref{prop:RGraph_symm_closed_mon,lem:recognition-of-idempotent-gabi-monads}, we obtain a monoidal structure on \(\preord\).

\begin{corollary}\label{cor:closed-mon-structure-on-preorders}
  Let \(\times \from \preord \times \preord \to \preord\), \((D, D') \mapsto R(D\square D')\) be the functor that sends two preorders \(D\) and \(D'\) to the transitive hull \(R(D \square D')\) of the cartesian product of their underlying digraphs.
  Then, \(\preord\) becomes closed monoidal with tensor product \(\times\), monoidal unit \(\!\!\bulletrefldig\!\!\), and exponential digraphs as internal hom.
\end{corollary}

Note that \(\preord\) is equivalent to the category \(\thinCat\) of (small) thin categories.\footnote{\,A category \( \cat{C}\) is \emph{thin} if \( | \cat{C}(x,y) | \leq 1\) for all \(x,y\in \cat{C}\).}
From this point of view, the tensor product and internal hom of \(\preord\) discussed above coincide with the restriction of the cartesian closed monoidal structure of \(\Cat\) to \(\thinCat\).

\begin{theorem}\label{thm:s-map-reflexive-graphs}
  Consider the monad \(P=\forget R \from \rgraph \to \rgraph\). Then:
  \begin{enumerate}[label=\textnormal{(\arabic*)},leftmargin=0.8cm]
    \item\label{item:reflgraphs1} The comparison functor \(K \from \preord \to \rgraph^{P}\) is an equivalence of categories.
    \item\label{item:reflgraphs2} Set \(\mathfrak{a}_{\tu}\from P(\tu) \to \tu\) to be the identity map and, for all \(D, D' \in \rgraph\),
    define \(s_{D,D'}\colon P[P(D),D'] \to [D,P(D')]\) to be the map that is uniquely determined by
    \begin{align*}
      \rgraph(R(D), D') &\to \rgraph(D,R(D')), \\
      (\phi_V,\phi_A) &\mapsto \big(\phi_V,(\eta_{D'})_A \circ \phi_A \circ (\eta_D)_A\big).
    \end{align*}
    Then \((P,s,\mathfrak{a}_{\tu})\) is a normal gabi-monad whose category of algebras is closed monoidal and \(K \from \preord \to \rgraph^{P}\) is strict closed monoidal.
    \item\label{item:reflgraphs3} The gabi-monad \((P,s, \mathfrak{a}_{\tu})\) is not Hopf.
  \end{enumerate}
\end{theorem}
\begin{proof}
  Claim \ref{item:reflgraphs1} follows immediately from \cref{thm:recognition-of-reflectiveness},
  and claim \ref{item:reflgraphs2} follows from
  \cref{thm:refl_subcats_give_gabi,lem:recognition-of-idempotent-gabi-monads,thm:gabi-bijection}.
  The map \(s_{D,D'}\from P[P(D), D']\to [D, P(D')]\) is now obtained by applying \cref{eq:SfromAct}.
  Furthermore, we have \(\mathfrak{a}_{\tu} \from P\tu = \forget R\tu\xrightarrow{\;\;\epsilon_{\tu}\;\;}\tu\).

  Finally, in order to conclude that \((P, s, \mathfrak{a}_{\tu})\) is not Hopf, it suffices to show  that the inclusion functor \(\forget \from \preord \to \rgraph\) is not strict monoidal.
  To that end, we set \(D\defeq\!\!\vwrefldig\!\!\) and observe that by \cref{eq:tp-of-interval-graphs}, we get \(\forget(D \times D) = \forget R(D\square D) \neq D\square D\).
\end{proof}

\subsection{Simplicial complexes}\label{sec:simplicial complexes}
Based on \cite[Chapter IV, \S7]{Eilenberg-Kelly}, we will now interpret the cartesian closed category of sets as algebras of a Hopf monad on simplicial complexes.

\begin{definition}\label{def:simplicial-complex}
  A \emph{simplicial complex} is a set \(V\), whose elements are called \emph{vertices}, and a set \(\Sigma \subseteq \cP_{\mathrm{fin}}^{+}(V)\) of finite non-empty subsets of \(V\), called \emph{simplices}, such that
  \begin{enumerate}[leftmargin=0.8cm]
    \item if \(\sigma \in \Sigma\) is a simplex and \(\tau \subseteq \sigma\), \(\tau \neq \varnothing\), then \(\tau \in \Sigma\),
    \item every singleton \(\{v\}\) for \(v \in V\) is a simplex.
  \end{enumerate}

  A \emph{morphism} \((V,\Sigma) \to (V',\Sigma')\) of simplicial complexes is a function \(f \colon V \to V'\) such that whenever \(\sigma \in \Sigma\), also \(f(\sigma) \in \Sigma'\).

  Simplicial complexes and their morphisms form the category  \({\sf SC}\).
\end{definition}

There is an obvious forgetful functor \(L \colon {\sf SC} \to {\sf Set}\), sending a simplicial complex \((V, \Sigma)\) to its set of vertices \(V\).
It admits a fully faithful right adjoint \(R \colon {\sf Set} \to {\sf SC}\) given by \(R(X) = (X,\cP_{\mathrm{fin}}^{+}(X))\) on objects and by \(R(f) = f\) on morphisms. The unit is explicitly given by the canonical embedding
\[
  \eta_{(V,\Sigma)} \colon (V,\Sigma) \xrightarrow{\id_V} (V,\cP_{\mathrm{fin}}^{+}(V))
\]
and the counit is the identity map
\[
  \epsilon_{X} \colon X \xrightarrow{\id_X} X.
\]

This implies the following result.
\begin{lemma}\label{lem:reflective-subcategory}
  The monad  \(H = RL \from \mathsf{SC}\to \mathsf{SC}\), \((V, \Sigma) \mapsto (V, \cP_{\mathrm{fin}}^{+}(V))\) is idempotent and the comparison functor establishes an equivalence \(\mathsf{SC}^H \simeq \Set\).
\end{lemma}

We can endow \({\sf SC}\) with the structure of a cartesian closed category as follows.
The product of \((V,\Sigma_V)\) and \((W,\Sigma_W)\) is \((V \times W, \Sigma_{V \times W})\) where a finite subset \(\forget \subseteq V \times W\) is in \(\Sigma_{V \times W}\) if and only if \(\pi_V(\forget) \in \Sigma_V\) and \(\pi_W(\forget) \in \Sigma_W\).
The internal hom \([(V,\Sigma_V),(W,\Sigma_W)]_{{\sf SC}}\) is \(({\sf SC}((V,\Sigma_V),(W,\Sigma_W)),\Sigma_{[V,W]})\) where \(\{f_1,\ldots,f_n\} \in \Sigma_{[V,W]}\) if and only if for every \(\sigma \in \Sigma_V\), \(f_1(\sigma) \cup f_2(\sigma) \cup \cdots \cup f_n(\sigma) \in \Sigma_W\).

\begin{theorem}\label{thm:Hopf-monad-set}
  The category \(\mathsf{SC}\) is cartesian closed.
  Moreover, if we endow \(\Set\) with its cartesian closed structure,  the inclusion \(R\from \Set\to \mathsf{SC}\) becomes strict closed monoidal, endowing \(H=RL\) with the structure of a Hopf monad.
\end{theorem}
\begin{proof}
  That \(\mathsf{SC}\) is cartesian closed is proven in~\cite[Chapter IV, \S7]{Eilenberg-Kelly}.

  In order to observe that \(R \from \Set\to \mathsf{SC}\) is strict closed monoidal, we first observe that the monoidal unit of \(\mathsf{SC}\) is given by \((\{\ast\}, \cP_{\mathrm{fin}}^{+}(\{\ast\})) = R(\{\ast\})\).
  Moreover, for two sets \(V, W\) we have
  \begin{align*}
    R(V\times W) &= (V\times W, \cP_{\mathrm{fin}}^{+}(V\times W)) = (V, \cP_{\mathrm{fin}}^{+}(V)) \times (W, \cP_{\mathrm{fin}}^{+}(W)) = R(V) \times R(W),\\
    R([V,W]) &= (\Set(V,W), \cP_{\mathrm{fin}}^{+}(\Set(V,W))) = [(V, \cP_{\mathrm{fin}}^{+}(V)), (W, \cP_{\mathrm{fin}}^{+}(W))] = [R(V), R(W)].
  \end{align*}
  The claim now follows from \cite[Theorem~3.6]{BruguieresLackVirelizier}.
\end{proof}

\begin{remark}
  Let \(\mathsf{SC}_{\text{fin}}\) be the full subcategory of \(\mathsf{SC}\) whose objects are simplicial complexes with finitely many vertices.
  Analogous to the preceding discussion, we consider the category of finite sets \(\Set_{\text{fin}}\) as a reflective full subcategory of \(\mathsf{SC}_{\text{fin}}\) with inclusion and reflection
  \begin{align*}
    L \from \mathsf{SC}_{\text{fin}}& \leftrightarrows \Set_{\text{fin}} \colon R \\
    (V, \Sigma)& \mapsto V \\
    (W, \cP_{\mathrm{fin}}^{+}(W)) &\mapsfrom W
  \end{align*}
  We now define a functor \(\langle \blank , \bblank \rangle \from \mathsf{SC}_{\text{fin}}\op\times \mathsf{SC}_{\text{fin}}\to \mathsf{SC}_{\text{fin}}\) by setting
  \begin{equation*}
    \langle (V,\Sigma_{V}), (W, \Sigma_{W}) \rangle = \big(\mathsf{SC}_{\text{fin}}((V,\Sigma_V),(W,\Sigma_W)),(\Sigma_{\langle V,W\rangle}) \big),
  \end{equation*}
  where \(\{f_{1}, \dots , f_{n}\}\in \Sigma_{\langle V, W \rangle}\) if and only if either \(n=1\) or \(f_{1}(V) \cup \dots \cup f_{n}(V) \in \Sigma_{W}\).
  As shown  in \cite[page 561]{Eilenberg-Kelly}, this defines a closed but non-monoidal structure on \(\mathsf{SC}_{\text{fin}}\).

  Given two finite sets \(V\) and \(W\), we observe that
  \begin{equation*}
    \langle R(V), R(W) \rangle
    = \langle (V, \cP_{\mathrm{fin}}^{+}(V)),(W, \cP_{\mathrm{fin}}^{+}(W))\rangle
    = (\Set(V,W), \cP_{\mathrm{fin}}^{+}(\Set(V,W)))
    = R([V,W]).
  \end{equation*}
  Thus, the monad \(H=RL \from \mathsf{SC}_{\text{fin}} \to \mathsf{SC}_{\text{fin}}\) can be endowed with the structure of a gabi-monad (with respect to the internal hom \(\langle \blank , \bblank \rangle\) on \(\mathsf{SC}_{\text{fin}}\)).
  Moreover, in contrast to \((\mathsf{SC}_{\text{fin}}, \langle \blank , \bblank\rangle)\),  its Eilenberg--Moore category \(\mathsf{SC}_{\text{fin}}^{H} \cong \Set_{\text{fin}}\) is closed monoidal with the cartesian product of sets as tensor product.
\end{remark}

\subsection{Almost an example: pointed sets}

Suppose that \(\cat{C}\) is a skew-closed monoidal category and that \(\cat{D}\) is an exponential ideal in \(\cat{C}\) with inclusion functor \(\iota\colon \cat{D} \to \cat{C}\).
Assume, moreover, that the objects \(H_{X,A}\) from \cref{def:exp-ideal} can be chosen such that there is a lift
\[
  H\from \cat[C]\op \times \cat[D] \to \cat[D],
  \qquad \iota H(X,A) \cong [X,\iota A]_{\cat{C}}.
\]
As we observed in the proof of \cref{thm:day-reflection-thm}, restricting \(H\) along \(\iota\op \times \id_{\cat[D]}\) induces an internal hom construction on \(\cat{D}\) by setting \( [A,B]_{\cat{D}} \defeq H(\iota A,B)\).
If, in addition, \(\iota\) is monadic, then one can transfer such an internal hom to the Eilenberg--Moore category of algebras of the corresponding monad.
In spite of this, in general this is not enough to obtain a gabi-monad, as we show here.

\medskip

Consider the category \(\Set_\term\) of pointed sets and point-preserving maps. Since \(\term\) is a monoid in \((\Set,\sqcup,\varnothing)\), the functor \((\blank)_+ \defeq \star \sqcup \blank\) is naturally a monad on \(\Set\), called the \emph{maybe monad}, whose Eilenberg--Moore category of algebras is exactly \(\Set_\term\). Thus, we also write \(\boldsymbol{X} \defeq (X, x_0) \in \Set_\term\).

The category \(\Set_\term\) is an exponential ideal in \(\Set\), because if \((Y,y_0)\) is a pointed set and \(X\) is a set, then \([X,Y] = \Set(X,Y)\) becomes pointed via the constant map \(c_{y_0}(x) \defeq y_0\).
For a map \(f\colon X' \to X\) of sets and a point-preserving map \(g\colon (Y,y_0) \to (Y',y'_0)\), the induced map
\[
  [f,g]\colon [X,Y] \to [X',Y'],
  \qquad\qquad
  h \mapsto g \circ h \circ f,
\]
is point-preserving, since \([f,g](c_{y_0})=c_{y'_0}\).
Thus, these choices define a functorial lift of the internal hom,
and \(\Set_\term\) becomes skew-closed when equipped with
\begin{itemize}[leftmargin=0.8cm]
    \item the object \(\boldsymbol{1} \defeq (\{0,1\}, 0)\),

    \item the functor
    \[\llbracket \blank,\bblank \rrbracket\from {\Set_\term}\op \times \Set_\term \to \Set_\term, \qquad ((X,x_0),(Y,y_0)) \mapsto (\Set(X,Y),c_{y_0}),\]

    \item the natural transformation \(\boldsymbol{i}_{\boldsymbol{X}} \colon \llbracket \boldsymbol{1},\boldsymbol{X} \rrbracket \to \boldsymbol{X}\) given by \(f \mapsto f(1)\),

    \item the dinatural transformation \(\boldsymbol{j}_{\boldsymbol{X}} \colon \boldsymbol{1} \to \llbracket \boldsymbol{X},\boldsymbol{X} \rrbracket\)
    given by \(0 \mapsto c_{x_0}\) and \(1 \mapsto \id_X\), and

    \item the (di)natural transformation \(\boldsymbol{\Gamma}_{\boldsymbol{Y},\boldsymbol{Z}}^{\boldsymbol{X}} \colon \llbracket \boldsymbol{Y},\boldsymbol{Z} \rrbracket \to \llbracket \llbracket \boldsymbol{X},\boldsymbol{Y} \rrbracket,\llbracket \boldsymbol{X},\boldsymbol{Z} \rrbracket \rrbracket\)
    given by \(\Gamma_{Y,Z}^X\), i.e., post-composition of functions.
\end{itemize}
The underlying functor \(\forget \colon \Set_\star \to \Set\) is closed, with structure morphisms
\[
    \varphi_0\from \term \xrightarrow{\ \ast \, \mapsto \, 1\ } \{0,1\} = \forget\boldsymbol{1}
    \qquad\text{and}\qquad
    \varphi_{\boldsymbol{X},\boldsymbol{Y}} \from \forget\llbracket\boldsymbol{X},\boldsymbol{Y}\rrbracket \xrightarrow{\ \id_{[X,Y]}\ } [\forget\boldsymbol{X},\forget\boldsymbol{Y}],
\]
but not strong closed, because \(\varphi_0\) is not bijective: the maybe monad is not a gabi-monad, despite its Eilenberg--Moore category of algebras being closed monoidal with the internal hom lifted from \(\Set\).

\printbibliography

\clearpage

\pagestyle{empty}

\appendix

\begin{amssidewaysfigure}\vspace{1cm}
    \[\hspace{-2.5cm}\mathscale{0.65}{
        \begin{tikzcd}[ampersand replacement=\&,column sep=large,row sep=large]
        {T[TX,TY]} \&\&\&\& {T[[TP,TX],[TP,TY]]} \& {T[T[TP,TX],[TP,TY]]} \\
        \\
        {T\forget[\free X,\free Y]} \& {T\forget[[\free P,\free X],[\free P,\free Y]]} \&\& {T[\forget[\free P,\free X],\forget[\free P,\free Y]]} \& {T[[TP,TX],\forget[\free P,\free Y]]} \\
        \\
        {\forget[\free X,\free Y]} \& {\forget[[\free P,\free X],[\free P,\free Y]]} \&\& {T[T\forget[\free P,\free X],\forget[\free P,\free Y]]} \& {T[T[TP,TX],\forget[\free P,\free Y]]} \\
        \\
        \& {[\forget[\free P,\free X],\forget[\free P,\free Y]]} \&\& {[T\forget[\free P,\free X],\forget[\free P,\free Y]]} \& {T[T[TP,TX],T\forget[\free P,\free Y]]} \& {T[T[TP,TX],T[TP,TY]]} \\
        \\
        \& {[\forget[\free P,\free X],[TP,TY]]} \& {[[TP,TX],\forget[\free P,\free Y]]} \& {[T[TP,TX],\forget[\free P,\free Y]]} \& {[T[TP,TX],T\forget[\free P,\free Y]]} \& {[T[TP,TX],T[TP,TY]]} \\
        \\
        \&\&\& {[T[TP,TX],[TP,TY]]} \&\& {[[TP,TX],T[TP,TY]]} \\
        \\
        {[TX,TY]} \& {[[TP,TX],[TP,TY]]} \&\& {[[TP,TX],[TP,TY]]} \& {[[TP,TX],\forget[\free P,\free Y]]} \& {[[TP,TX],T\forget[\free P,\free Y]]}
        \arrow["{{{{{T\Gamma_{TX,TY}^{TP}}}}}}", from=1-1, to=1-5]
        \arrow["{{{{{T\varphi^{-1}_{\free X,\free Y}}}}}}"{description}, from=1-1, to=3-1]
        \arrow["\begin{array}{c} \substack{\forget \, \mathsf{closed} \\[1pt] [\blank,\bblank] \, \mathsf{funct}} \end{array}"{description}, draw=none, from=1-1, to=3-5]
        \arrow["{{{{{T[\act_{[TP,TX]},\mathsf{id}]}}}}}", from=1-5, to=1-6]
        \arrow["{{{T[\mathsf{id},\varphi^{-1}_{\free P,\free Y}]}}}"{description}, from=1-5, to=3-5]
        \arrow["\begin{array}{c} \substack{\eta \, \mathsf{nat} \\[1pt] [\blank,\bblank] \, \mathsf{funct}} \end{array}"{description}, draw=none, from=1-5, to=7-6]
        \arrow["{{{T[\mathsf{id},\eta_{[TP,TY]}]}}}"{description}, from=1-6, to=7-6]
        \arrow["{{{{{T\forget\Gamma_{\free X,\free Y}^{\free P}}}}}}", from=3-1, to=3-2]
        \arrow["{{{{{\forget\epsilon_{[\free X,\free Y]}}}}}}"{description}, from=3-1, to=5-1]
        \arrow["{\epsilon \, \mathsf{nat}}"{description}, draw=none, from=3-1, to=5-2]
        \arrow["{{{T\varphi_{[\free P,\free X],[\free P,\free Y]}}}}", from=3-2, to=3-4]
        \arrow["{{{{{\forget\epsilon_{[[\free P,\free X],[\free P,\free Y]]}}}}}}"{description}, from=3-2, to=5-2]
        \arrow["{{{T[\forget\epsilon_{[\free P,\free X]},\forget[\free P,\free Y]]}}}"{description}, from=3-4, to=5-4]
        \arrow["{\mathsf{def} \, \act}"{description}, draw=none, from=3-4, to=5-5]
        \arrow[""{name=0, anchor=center, inner sep=0}, "{{{\act_{[\forget[\free P,\free X],\forget[\free P,\free Y]]}}}}"{description}, from=3-4, to=7-2]
        \arrow["{{{T[\varphi_{\free P,\free X},\mathsf{id}]}}}"', from=3-5, to=3-4]
        \arrow["{{{T[\act_{[TP,TX]},\mathsf{id}]}}}"{description}, from=3-5, to=5-5]
        \arrow["{{{{{\forget\Gamma_{\free X,\free Y}^{\free P}}}}}}"', from=5-1, to=5-2]
        \arrow["{{{{{\varphi_{\free X,\free Y}}}}}}"{description}, from=5-1, to=13-1]
        \arrow["{\forget \, \mathsf{closed}}"{description, pos=0.45}, draw=none, from=5-1, to=13-2]
        \arrow["{{{{{\varphi_{[\free P,\free X],[\free P,\free Y]}}}}}}"{description}, from=5-2, to=7-2]
        \arrow["{{{T[T\varphi^{-1}_{\free P,\free X},\mathsf{id}]}}}", from=5-4, to=5-5]
        \arrow["{{{\act_{[T\forget[\free P,\free X],\forget[\free P,\free Y]]}}}}"{description}, from=5-4, to=7-4]
        \arrow["{\act \, \mathsf{nat}}"{pos=0.4}, draw=none, from=5-4, to=9-5]
        \arrow["{{{T[\mathsf{id},\eta_{\forget[\free P,\free Y]}]}}}"{description}, from=5-5, to=7-5]
        \arrow["{{{[\forget\epsilon_{[\free P,\free X]},\mathsf{id}]}}}", curve={height=-6pt}, from=7-2, to=7-4]
        \arrow["{{{{{[\mathsf{id},\varphi_{\free P,\free Y}]}}}}}"', from=7-2, to=9-2]
        \arrow["{{{[\varphi^{-1}_{\free P,\free X},\mathsf{id}]}}}"{description}, from=7-2, to=9-3]
        \arrow["{\eta \, \mathsf{nat}}"{description, pos=0.6}, draw=none, from=7-2, to=9-4]
        \arrow["{{{[\eta_{\forget[\free P,\free X]},\mathsf{id}]}}}", curve={height=-6pt}, from=7-4, to=7-2]
        \arrow["{{{[T\varphi^{-1}_{\free P,\free X},\mathsf{id}]}}}"{description}, from=7-4, to=9-4]
        \arrow["{{{\act_{[T[TP,TX],T\forget[\free P,\free Y]]}}}}"{description}, from=7-5, to=9-5]
        \arrow["{\act \, \mathsf{nat}}"{description}, draw=none, from=7-5, to=9-6]
        \arrow["{{{T[\mathsf{id},T\varphi^{-1}_{\free P,\free Y}]}}}"', from=7-6, to=7-5]
        \arrow["{{{\act_{[T[TP,TX],T[TP,TY]]}}}}"{description}, from=7-6, to=9-6]
        \arrow["{{{{{[\varphi^{-1}_{\free P,\free X},\mathsf{id}]}}}}}"', from=9-2, to=13-2]
        \arrow[""{name=1, anchor=center, inner sep=0}, "{{{[\mathsf{id},\varphi_{\free P,\free Y}]}}}"{description}, from=9-3, to=13-2]
        \arrow["{[\blank,\bblank] \, \mathsf{funct}}"'{pos=0.4}, draw=none, from=9-3, to=13-4]
        \arrow["{{{[\eta_{[TP,TX]},\mathsf{id}]}}}"', from=9-4, to=9-3]
        \arrow["{{{[\mathsf{id},\eta_{\forget[\free P,\free Y]}]}}}", curve={height=-6pt}, from=9-4, to=9-5]
        \arrow["{{{[\mathsf{id},\varphi_{\free P,\free Y}]}}}"{description}, from=9-4, to=11-4]
        \arrow["{{{[\mathsf{id},\forget\epsilon_{[\free P,\free Y]}]}}}", curve={height=-6pt}, from=9-5, to=9-4]
        \arrow["{{{[\mathsf{id},T\varphi^{-1}_{\free P,\free Y}]}}}"', from=9-6, to=9-5]
        \arrow[""{name=2, anchor=center, inner sep=0}, "{[T[TP,TX],\act_{[TP,TY]}]}"{description}, from=9-6, to=11-4]
        \arrow["{{{[\eta_{[TP,TX]},\mathsf{id}]}}}"{description}, from=9-6, to=11-6]
        \arrow["{[\blank,\bblank] \, \mathsf{funct}}"{description}, draw=none, from=11-4, to=11-6]
        \arrow["{{{[\eta_{[TP,TX]},\mathsf{id}]}}}"{description}, from=11-4, to=13-4]
        \arrow[""{name=3, anchor=center, inner sep=0}, "{[[TP,TX],\act_{[TP,TY]}]}"{description}, from=11-6, to=13-4]
        \arrow["{{{[\mathsf{id},T\varphi^{-1}_{\free P,\free Y}]}}}"{description}, from=11-6, to=13-6]
        \arrow["{{{{{\Gamma_{TX,TY}^{TP}}}}}}"', from=13-1, to=13-2]
        \arrow["{{{=}}}"{description}, no head, from=13-2, to=13-4]
        \arrow["{{{[\mathsf{id},\varphi_{\free P,\free Y}]}}}", from=13-5, to=13-4]
        \arrow["{{{[\mathsf{id},\forget\epsilon_{[\free P,\free Y]}]}}}", from=13-6, to=13-5]
        \arrow["{\mathsf{def} \, \act}"{description}, draw=none, from=3-2, to=0]
        \arrow["{\act \, \mathsf{nat}}"{description}, draw=none, from=0, to=7-4]
        \arrow["{[\blank,\bblank] \, \mathsf{funct}}"{description}, draw=none, from=9-2, to=1]
        \arrow["{\mathsf{def} \, \act}"'{pos=0.4}, draw=none, from=9-4, to=2]
        \arrow["{\mathsf{def} \, \act}"{pos=0.7}, draw=none, from=3, to=13-6]
            \arrow["{\act_{[TX,TY]}}"{description}, from=1-1, to=13-1, rounded corners, color=red,
              to path={-- ([xshift=-2cm]\tikztostart.center) -- ([xshift=-2cm]\tikztotarget.center) \tikztonodes -- (\tikztotarget.west)}]
            \arrow["{s_{[TP,TX],[TP,TY]}}"{description}, from=1-6, to=11-6, rounded corners, color=red,
              to path={-- ([xshift=3cm]\tikztostart.center) -- ([xshift=3cm,yshift=.1cm]\tikztotarget.center) \tikztonodes -- ([yshift=.1cm]\tikztotarget.east)}]
            \arrow["{[[TP,TX],\act_{[TP,TY]}]}", from=11-6, to=13-4, rounded corners, color=red,
              to path={let \p1=([xshift=3cm]\tikztostart.center), \p2=([yshift=-1cm]\tikztotarget.center) in
                       ([yshift=-.1cm]\tikztostart.east) -- ([yshift=-.1cm]\p1) -- (\x1,\y2) -- (\p2) \tikztonodes -- (\tikztotarget.south)}]
          \end{tikzcd}
        }
    \]
    \caption{\(\Gamma^{TP}_{TX,TY}\) is a morphism of \(T\)-algebras}%
    \label{fig:GammaTalgs}%
  \end{amssidewaysfigure}

\clearpage

\vspace*{\fill}

\begin{figure}[htbp]
  \[
    \begin{tikzcd}[column sep=60pt, row sep=25pt]
      {T^2[TX,Y]} & & {T^2[T^2X,Y]} \\
      {T[TX,Y]} & {T^2[TX,TY]} & {T^2[T^2X,TY]} \\
      {T[TX,TY]} & & \\
      & {T[TX,TY]} & {T[T^2X,TY]} \\
      & & {T[TX,TY]} \\
      & {T[TX,TY]} & {T[TX,T^2Y]} \\
      {[TX,TY]} & & {[TX,T^2Y]} \\
      {[X,TY]} & & {[X,T^2Y]}
      \arrow["{\mu_{[TX,Y]}}"', from=1-1, to=2-1]
      \arrow["{T^2[\mu_X,Y]}", from=1-1, to=1-3]
      \arrow["{T^2[TX,\eta_Y]}"{description}, from=1-1, to=2-2]
      \arrow["{T^2[T^2X,\eta_Y]}", from=1-3, to=2-3]
      \arrow["{[-,=]\,\mathsf{functor}}"{description, pos=0.4}, draw=none, from=1-3, to=2-1]
      \arrow["{T[TX,\eta_Y]}"', from=2-1, to=3-1]
      \arrow["{T^2[\mu_X,TY]}", from=2-2, to=2-3]
      \arrow["{T\act_{[TX,TY]}}"', from=2-2, to=4-2]
      \arrow[""{name=0, anchor=center, inner sep=0}, "{\mu_{[TX,TY]}}"{description}, from=2-2, to=3-1]
      \arrow["{\mu\,\mathsf{nat}}"{description}, draw=none, from=1-1, to=0]
      \arrow["{T\act_{[T^2X,TY]}}", from=2-3, to=4-3]
      \arrow["{\act\,\mathsf{nat}}"{description}, draw=none, from=2-3, to=4-2]
      \arrow["{T\mbox{-}\mathsf{alg}}"{description}, draw=none, from=3-1, to=6-2]
      \arrow["{\act_{[TX,TY]}}"', from=3-1, to=7-1]
      \arrow["{T[\mu_X,TY]}", from=4-2, to=4-3]
      \arrow[equals, from=4-2, to=6-2]
      \arrow["{T[\eta_{TX},TY]}", from=4-3, to=5-3]
      \arrow["{T[TX,\eta_{TY}]}", from=5-3, to=6-3]
      \arrow[""{name=1, anchor=center, inner sep=0}, "{\act_{[TX,TY]}}"', from=6-2, to=7-1]
      \arrow["{\act_{[TX,T^2Y]}}", from=6-3, to=7-3]
      \arrow["{T[TX,\mu_Y]}", from=6-3, to=6-2]
      \arrow["{\act\,\mathsf{nat}}"{description}, shift left=3, draw=none, from=6-3, to=1]
      \arrow["{[\eta_X,TY]}"', from=7-1, to=8-1]
      \arrow["{[\eta_X,T^2Y]}", from=7-3, to=8-3]
      \arrow["{[TX,\mu_Y]}", from=7-3, to=7-1]
      \arrow["{[-,=]\,\mathsf{functor}}"{description}, draw=none, from=7-3, to=8-1]
      \arrow["{[X,\mu_Y]}", from=8-3, to=8-1]
      \arrow["{Ts_{TX,Y}}", from=1-3, to=5-3, rounded corners, color=red,
      to path={-- ([xshift=2cm]\tikztostart.center) -- ([xshift=2cm,yshift=.1cm]\tikztotarget.center) \tikztonodes -- ([yshift=.1cm]\tikztotarget.east)}]
      \arrow["{s_{X,TY}}", from=5-3, to=8-3, rounded corners, color=red,
      to path={([yshift=-.1cm]\tikztostart.east) -- ([xshift=2cm,yshift=-.1cm]\tikztostart.center) -- ([xshift=2cm]\tikztotarget.center) \tikztonodes -- (\tikztotarget.east)}]
      \arrow["{s_{X,Y}}"', from=2-1, to=8-1, rounded corners, color=red,
      to path={-- ([xshift=-2cm]\tikztostart.center) -- ([xshift=-2cm]\tikztotarget.center) \tikztonodes -- (\tikztotarget.west)}]
    \end{tikzcd}
  \]
  \caption{Commutativity of the second diagram of \eqref{eq:lifting_homs} in \cref{prop:gabi-moand-VS-strong-closed}.}
  \label{fig:gabi-moand-VS-strong-closed-1}
\end{figure}
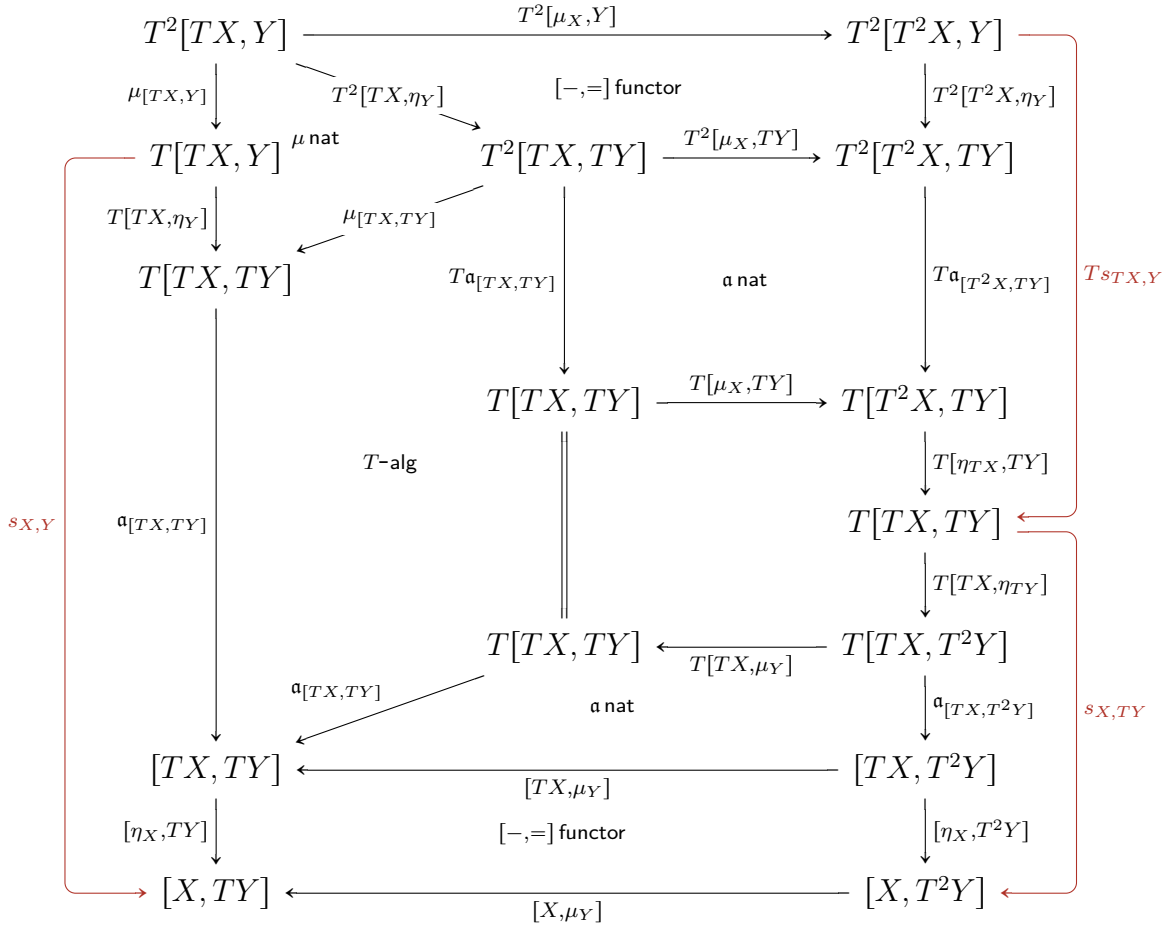

\vspace*{\fill}

\clearpage

\begin{amssidewaysfigure}\vspace{3.5cm}
  \[
    \hspace{-2.5cm}\mathscale{0.8}{\begin{tikzcd}[ampersand replacement=\&, row sep=large]
        {T[\tu, X]} \&\&\& TX \&\&\& {[\tu,TX]} \\
        \& {T[\tu, TX]} \& T^2X \& {\forget[\tu,\free X]} \& {[\forget\tu,TX]} \& {[\forget\tu,TX]} \\
        \\
        {T[\forget\tu, X]} \& {T[\forget\tu, TX]} \& {T[\forget\tu, TX]} \& {T\forget[\tu,\free X]} \& {\forget[\tu,\free X]} \& {[T\forget\tu,TX]} \& {[T\tu,TX]} \\
        {T[T\forget\tu,X]} \&\& {T[T\forget\tu, TX]} \& {T\forget[\free\forget\tu,\free X]} \& {\forget[\free\forget\tu,\free X]} \&\& {\forget[\free\tu,\free X]} \\
        {T[T\tu, X]} \&\& {T[T\tu, TX]} \& {T\forget[\free\tu, \free X]} \&\&\& {T\forget[\free\tu, \free X]}
        \arrow[""{name=0, anchor=center, inner sep=0}, "{{T i_X}}", from=1-1, to=1-4]
        \arrow["{{T[\tu,\eta_X]}}"{description}, from=1-1, to=2-2]
        \arrow["{{T[\varphi_0^{-1},X]}}"', from=1-1, to=4-1]
        \arrow["{{T\eta_X}}"{description}, curve={height=6pt}, from=1-4, to=2-3]
        \arrow[""{name=1, anchor=center, inner sep=0}, "{{i_{TX}}}"', from=1-7, to=1-4]
        \arrow[""{name=2, anchor=center, inner sep=0}, "{{Ti_{TX}}}", from=2-2, to=2-3]
        \arrow["{{T[\varphi_0^{-1},X]}}", from=2-2, to=4-2]
        \arrow["{{\forget\epsilon_{\free X}}}"{description}, curve={height=6pt}, from=2-3, to=1-4]
        \arrow["{{\forget i_{\free X}}}"', from=2-4, to=1-4]
        \arrow[""{name=3, anchor=center, inner sep=0}, "{{\varphi_{\tu,\free X}}}"', from=2-4, to=2-5]
        \arrow[equals, from=2-5, to=2-6]
        \arrow["{{[\varphi_0, TX]}}"{description}, from=2-6, to=1-7]
        \arrow["{{[\forget\epsilon_\tu,TX]}}"{description}, from=2-5, to=4-6]
        \arrow["{{[{-},{=}]\ \mathsf{functor}}}"{description}, draw=none, from=4-1, to=2-2]
        \arrow[""{name=4, anchor=center, inner sep=0}, "{{T[\forget\tu, \eta_X]}}", from=4-1, to=4-2]
        \arrow["{{T[\forget\epsilon_\tu, X]}}"', from=4-1, to=5-1]
        \arrow[equals, from=4-2, to=4-3]
        \arrow[""{name=5, anchor=center, inner sep=0}, "{{T\varphi^{-1}_{\tu,\free X}}}", from=4-3, to=4-4]
        \arrow["{{T[\forget\epsilon_\tu,TX]}}"', from=4-3, to=5-3]
        \arrow[""{name=6, anchor=center, inner sep=0}, "{{T\forget i_{\free X}}}", from=4-4, to=2-3]
        \arrow["{{\forget\epsilon_{[\tu,\free X]}}}"', from=4-4, to=2-4]
        \arrow[""{name=7, anchor=center, inner sep=0}, "{{\forget\epsilon_{[\tu,\free X]}}}", from=4-4, to=4-5]
        \arrow["{{T\forget[\epsilon_\tu, \free X]}}"{description}, from=4-4, to=5-4]
        \arrow[""{name=8, anchor=center, inner sep=0}, "{{\varphi_{\tu,\free X}}}", from=4-5, to=2-5]
        \arrow["{{\forget[\epsilon_\tu,\free X]}}"{description}, from=4-5, to=5-5]
        \arrow[""{name=9, anchor=center, inner sep=0}, "{[\eta_{\forget\tu}, T X]}"{description}, from=4-6, to=2-6]
        \arrow[""{name=10, anchor=center, inner sep=0}, "{{[T\varphi_0, TX]}}"', from=4-6, to=4-7]
        \arrow[""{name=11, anchor=center, inner sep=0}, "{{[\eta_\tu, TX]}}"', from=4-7, to=1-7]
        \arrow[""{name=12, anchor=center, inner sep=0}, "{{T[T\forget\tu, \eta_X]}}"{description}, from=5-1, to=5-3]
        \arrow["{{T[T\varphi_0,X]}}"', from=5-1, to=6-1]
        \arrow[""{name=13, anchor=center, inner sep=0}, "{{T\varphi^{-1}_{\free\forget\tu, \free X}}}", from=5-3, to=5-4]
        \arrow["{{T[T\varphi_0, TX]}}"', from=5-3, to=6-3]
        \arrow[""{name=14, anchor=center, inner sep=0}, "{{\forget\epsilon_{[\free\forget\tu,\free X]}}}", from=5-4, to=5-5]
        \arrow["{{T[\free\varphi_0, \free X]}}"{description}, from=5-4, to=6-4]
        \arrow["{{\varphi_{\free\forget\tu,\free X}}}"{description}, from=5-5, to=4-6]
        \arrow[""{name=15, anchor=center, inner sep=0}, "{{\forget[\free\varphi_0, \free X]}}", from=5-5, to=5-7]
        \arrow["{{\varphi_{\free\tu,\free X}}}"', from=5-7, to=4-7]
        \arrow[""{name=16, anchor=center, inner sep=0}, "{{T[T\tu, \eta_X]}}"', from=6-1, to=6-3]
        \arrow[""{name=17, anchor=center, inner sep=0}, "{{T\varphi^{-1}_{\free\tu,\free X}}}"', from=6-3, to=6-4]
        \arrow[equals, from=6-4, to=6-7]
        \arrow[""{name=18, anchor=center, inner sep=0}, "{{\forget\epsilon_{[\free\tu,\free X]}}}"', from=6-7, to=5-7]
        \arrow["{{\mathsf{nat}}}"{description}, draw=none, from=0, to=2]
        \arrow["{{\forget\ \mathsf{closed}}}"{description}, draw=none, from=1, to=3]
        \arrow["{{\forget\ \mathsf{closed}}}"{description}, draw=none, from=2, to=4-3]
        \arrow["\equiv"{description}, draw=none, from=3, to=7]
        \arrow["{{[\blank, \bblank] \mathsf{ functor}}}"{description}, draw=none, from=4, to=12]
        \arrow["{{\mathsf{nat}}}"{description}, draw=none, from=5, to=13]
        \arrow["{{\mathsf{nat}}}"{description}, draw=none, from=6, to=2-4]
        \arrow["{{\mathsf{nat}}}"{description}, draw=none, from=7, to=14]
        \arrow["{{\mathsf{nat}}}"{description}, draw=none, from=8, to=4-6]
        \arrow["{{\mathsf{nat}}}"{description}, draw=none, from=9, to=11]
        \arrow["{{\mathsf{nat}}}"{description}, draw=none, from=10, to=15]
        \arrow["{{[\blank, \bblank] \mathsf{ functor}}}"{description}, draw=none, from=12, to=16]
        \arrow["{{\mathsf{nat}}}"{description}, draw=none, from=13, to=17]
        \arrow["{{\mathsf{nat}}}"{description}, draw=none, from=5-5, to=18]
        \arrow["{T[\act_{\tu},X]}"', from=1-1, to=6-1, rounded corners, color=red,
        to path={-- ([xshift=-2cm]\tikztostart.center) -- ([xshift=-2cm]\tikztotarget.center) \tikztonodes -- (\tikztotarget.west)}]
        \arrow["{s_{\tu,X}}"', from=6-1, to=1-7, rounded corners, color=red,
        to path={let \p1=([yshift=-1cm]\tikztostart.center), \p2=([xshift=2cm]\tikztotarget.center) in
          -- (\p1) -- (\x2, \y1) -- (\p2) \tikztonodes -- (\tikztotarget.east)}]
      \end{tikzcd}}
  \]
  \caption{Commutativity of the leftmost diagram of \cref{eq:siotaj} in \cref{prop:gabi-moand-VS-strong-closed}.}%
  \label{fig:badmagic}%
\end{amssidewaysfigure}
\vspace*{\fill}

\begin{figure}[htbp]
  \[
    \begin{tikzcd}[column sep=65pt, row sep=50pt]
      {T\tu} & {T[TM,TM]} & {T[TM,TTM]} \\
      {T\forget\tu} & {T\forget[\free M,\free M]} & {T\forget[\free M,\free TM]} \\
      {\forget\tu} & {\forget[\free M,\free M]} & {\forget[\free M,\free TM]} \\
      \tu & {[TM,TM]} & {[TM,TTM]} \\
      {[M,M]} & {[M,TM]} & {[M,TTM]}
      \arrow["{Tj_{TM}}", from=1-1, to=1-2]
      \arrow["{T\varphi_0}"', from=1-1, to=2-1]
      \arrow["\mathsf{closed}"{description}, draw=none, from=1-1, to=2-2]
      \arrow["{T[TM,\eta_{TM}]}", from=1-2, to=1-3]
      \arrow["{T\varphi^{-1}_{\free M,\free M}}", from=1-2, to=2-2]
      \arrow["{\varphi\,\mathsf{nat}}"{description}, draw=none, from=1-2, to=2-3]
      \arrow["{T[TM,\forget\epsilon_{\free M}]}", dotted, curve={height=-15pt}, from=1-3, to=1-2]
      \arrow["{T\varphi^{-1}_{\free M,\free TM}}", from=1-3, to=2-3]
      \arrow["{T\forget j_{\free M}}", from=2-1, to=2-2]
      \arrow["{\forget\epsilon_{\tu}}"', from=2-1, to=3-1]
      \arrow["{\epsilon\,\mathsf{nat}}"{description}, draw=none, from=2-1, to=3-2]
      \arrow["{\forget\epsilon_{[\free M,\free M]}}", from=2-2, to=3-2]
      \arrow["{\epsilon\,\mathsf{nat}}"{description}, draw=none, from=2-2, to=3-3]
      \arrow["{T\forget[\free M,\epsilon_{\free M}]}"', from=2-3, to=2-2]
      \arrow["{\forget\epsilon_{[\free M,\free TM]}}", from=2-3, to=3-3]
      \arrow["{\forget j_{\free M}}", from=3-1, to=3-2]
      \arrow["{\varphi_0^{-1}}"', from=3-1, to=4-1]
      \arrow["\mathsf{closed}"{description}, draw=none, from=3-1, to=4-2]
      \arrow["{\varphi_{\free M,\free M}}", from=3-2, to=4-2]
      \arrow["{\varphi\,\mathsf{nat}}"{description}, draw=none, from=3-2, to=4-3]
      \arrow["{\forget[\free M,\epsilon_{\free M}]}"', from=3-3, to=3-2]
      \arrow["{\varphi_{\free M,\free TM}}", from=3-3, to=4-3]
      \arrow["{j_{TM}}", from=4-1, to=4-2]
      \arrow["{j_M}"', from=4-1, to=5-1]
      \arrow["{j\,\mathsf{dinat}}"{description}, draw=none, from=4-1, to=5-2]
      \arrow["{[\eta_M,TM]}", from=4-2, to=5-2]
      \arrow["{[-,=]\,\mathsf{functor}}"{description}, draw=none, from=4-2, to=5-3]
      \arrow["{[TM,\forget\epsilon_{\free M}]}"', from=4-3, to=4-2]
      \arrow["{[\eta_M,TTM]}", from=4-3, to=5-3]
      \arrow["{[M,\eta_M]}", dotted, curve={height=-15pt}, from=5-1, to=5-2]
      \arrow["{[M,\act_M]}", from=5-2, to=5-1]
      \arrow["{[M,\forget\epsilon_{\free M}]}"', from=5-3, to=5-2]
      \arrow["{\act_{\tu}}"', from=1-1, to=4-1, rounded corners, color=red,
        to path={-- ([xshift=-1.25cm]\tikztostart.center) -- ([xshift=-1.25cm]\tikztotarget.center) \tikztonodes -- (\tikztotarget.west)}]
      \arrow["{s_{M,TM}}", from=1-2, to=5-3, rounded corners, color=red,
        to path={let \p1=([yshift=1cm]\tikztostart.center), \p2=([xshift=2.5cm]\tikztotarget.center) in
                 -- (\p1) -- (\x2, \y1) -- (\p2) \tikztonodes -- (\tikztotarget.east)}]
    \end{tikzcd}
  \]
  \caption{Commutativity of the rightmost diagram of \eqref{eq:siotaj} in \cref{prop:gabi-moand-VS-strong-closed}.}
  \label{fig:gabi-moand-VS-strong-closed-2}
\end{figure}
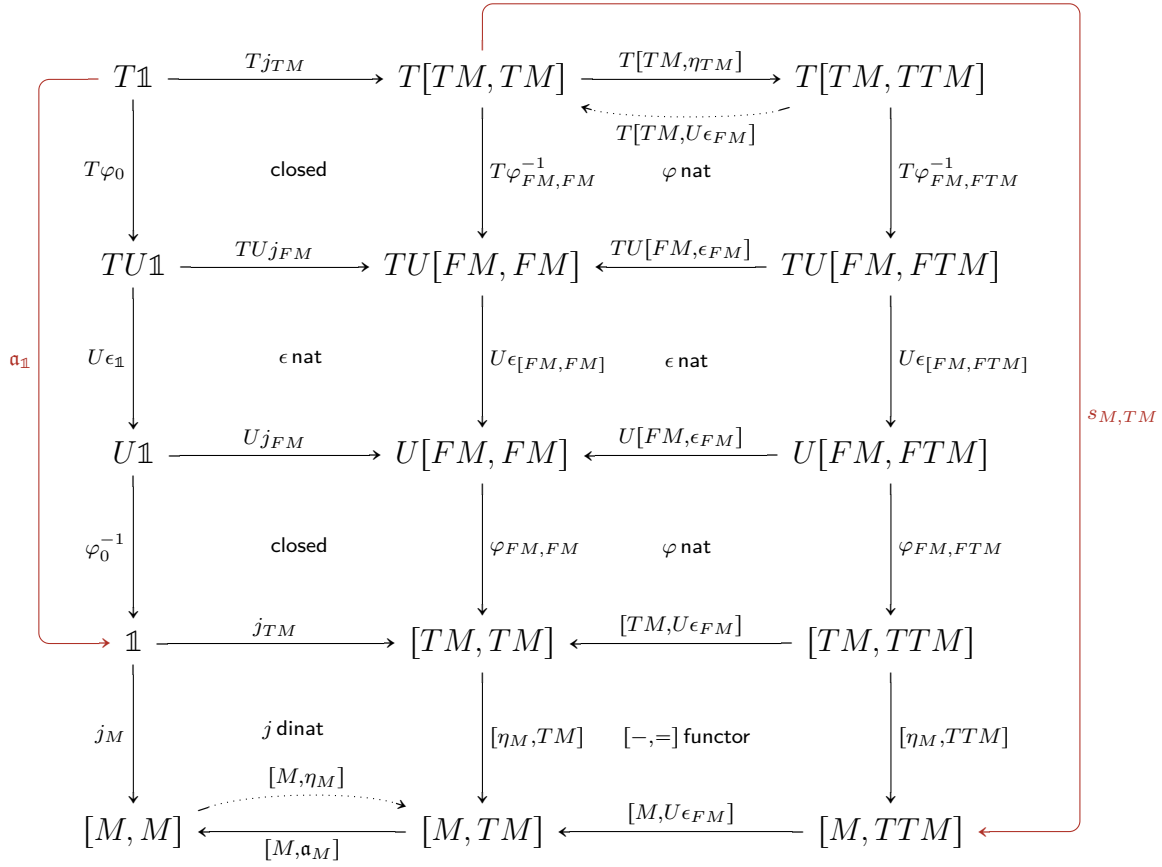

\vspace*{\fill}

\clearpage

\begin{amssidewaysfigure}\vspace{1.3cm}
    \[\hspace{-2.8cm}\mathscale{0.725}{
        \begin{tikzcd}[ampersand replacement=\&]
        	{T[[P,TX],[P,Y]]} \&[-10pt]\& {T[TX,Y]} \&[-10pt]\&\&\& {T[TX,TY]} \\
        	\\
        	{T[[P,TX],[TP,Y]]} \&[-10pt]\& {T[[TP,TX],[TP,Y]]} \&[-10pt]\&\& {T[[TP,TX],[TP,TY]]} \\
        	\\
        	{T[[TP,TX],[TP,Y]]} \&[-10pt]\& {T[T[TP,TX],[TP,Y]]} \&[-10pt]\&\& {T[T[TP,TX],[TP,TY]]} \\
        	\\
        	\&[-10pt] {T[[P,TX],[TP,Y]]} \&\&[-10pt] {[[TP,TX],T[TP,Y]]} \&\& {[[TP,TX],T[TP,TY]]} \& {[TX,TY]} \\
        	{T[T[TP,TX],[TP,Y]]} \&[-10pt]\& {T[T[TP,X],[TP,Y]]} \&[-10pt]\&\& {[[TP,TX],[TP,TY]]} \\[-10pt]
        	\&[-10pt]\&\&[-10pt]\& {[[TP,TX],[P,TY]]} \\[-10pt]
        	\&[-10pt] {T[T[P,TX],[TP,Y]]} \&\&[-10pt] {[[TP,X],T[TP,Y]]} \&\&\& {[X,TY]} \\
        	\&[-10pt]\&\&[-10pt]\&\& {[[TP,X],[TP,TY]]} \\
        	{T[T[TP,X],[TP,Y]]} \&[-10pt]\& {T[T[P,X],[TP,Y]]} \&[-10pt]\& {[[TP,X],[P,TY]]} \\
        	\\
        	{[[TP,X],T[TP,Y]]} \&[-10pt]\&\&[-10pt] {[[P,X],T[TP,Y]]} \& {[[P,X],[P,TY]]}
        	\arrow["{{{{T[\id,[\act_P,Y]]}}}}"{description}, from=1-1, to=3-1]
        	\arrow["{{{{T\Gamma_{TX,Y}^P}}}}"{description}, from=1-3, to=1-1]
        	\arrow["{{T[TX,\eta_Y]}}"{description}, from=1-3, to=1-7]
        	\arrow["{{{\Gamma\,\mathsf{dinat}}}}"{description}, draw=none, from=1-3, to=3-1]
        	\arrow["{{{{T\Gamma_{TX,Y}^{TP}}}}}", from=1-3, to=3-3]
        	\arrow["{{{{T\Gamma_{TX,TY}^{TP}}}}}"{description}, from=1-7, to=3-6]
        	\arrow["{{\act_{[TX,TY]}}}"{description}, from=1-7, to=7-7]
        	\arrow["{{\eqref{eq:GammaTalgsFree}}}"{description, pos=0.4}, draw=none, from=1-7, to=8-6]
        	\arrow["{T[[\eta_P,TX],[TP,Y]]}"{description}, from=3-1, to=5-1]
        	\arrow["{{{\Gamma\,\mathsf{nat}}}}"{description}, draw=none, from=3-3, to=1-7]
        	\arrow["{T[[\act_P,TX],\id]}"{description}, from=3-3, to=3-1]
        	\arrow["{{T[\id,[TP,\eta_Y]]}}"{description}, from=3-3, to=3-6]
        	\arrow["{{{{T[\act_{[TP,TX]},\id]}}}}", from=3-3, to=5-3]
        	\arrow["{{\mathsf{funct}}}"{description}, draw=none, from=3-3, to=5-6]
        	\arrow["{T[[\act_P,TX],\id]}"', from=3-3, to=7-2]
        	\arrow["{{T[\act_{[TP,TX]},\id]}}"', from=3-6, to=5-6]
        	\arrow["{T\,\mathsf{alg}}"{description}, draw=none, from=5-1, to=3-3]
        	\arrow["{T[[\act_P,TX],\id]}"{description}, from=5-1, to=7-2]
        	\arrow["{{{{T[\act_{[TP,TX]},\id]}}}}"{description}, from=5-1, to=8-1]
        	\arrow["{{T[\id,[TP,\eta_Y]]}}"{description}, from=5-3, to=5-6]
        	\arrow["{{s_{[TP,TX],[TP,Y]}}}"{description}, from=5-3, to=7-4]
        	\arrow[""{name=0, anchor=center, inner sep=0}, "{{{{T[T[TP,\eta_X],\id]}}}}"{description}, from=5-3, to=8-3]
        	\arrow["{{{s\,\mathsf{nat}}}}"{description}, draw=none, from=5-3, to=14-4]
        	\arrow["{{s_{[TP,TX],[TP,TY]}}}"{description}, from=5-6, to=7-6]
        	\arrow[""{name=1, anchor=center, inner sep=0}, "{T[\act_{[P,TX]},\id]}"{description}, from=7-2, to=10-2]
        	\arrow["{{{s\,\mathsf{nat}}}}"{description}, draw=none, from=7-4, to=5-6]
        	\arrow["{{[\id,T[TP,\eta_Y]]}}"{description}, from=7-4, to=7-6]
        	\arrow["{{\mathsf{def} \, s}}"{description}, draw=none, from=7-4, to=8-6]
        	\arrow["{{{{[\id,s_{P,Y}]}}}}"{description}, from=7-4, to=9-5]
        	\arrow["{{[[TP,\eta_X],\id]}}"{description}, from=7-4, to=10-4]
        	\arrow["{{[\id,\act_{[TP,TY]}]}}"{description}, from=7-6, to=8-6]
        	\arrow["{{{{\Gamma_{TX,TY}^{TP}}}}}"{description}, from=7-7, to=8-6]
        	\arrow["{{[\eta_X,TY]}}"{description}, from=7-7, to=10-7]
        	\arrow["{{\Gamma\,\mathsf{nat}}}"{description}, draw=none, from=7-7, to=11-6]
        	\arrow["\def\arraystretch{0.8}\begin{array}{c} \act\,\mathsf{nat} \\ \act_P\,\mathsf{of\,algs} \end{array}"{description}, draw=none, from=8-1, to=7-2]
        	\arrow[""{name=2, anchor=center, inner sep=0}, "{{{{T[T[\act_P,TX],\id]}}}}"{description}, from=8-1, to=10-2]
        	\arrow["{{{{T[T[TP,\eta_X],\id]}}}}"{description}, from=8-1, to=12-1]
        	\arrow["{{{{T[T[\act_P,X],\id]}}}}"{description, pos=0.6}, from=8-3, to=12-3]
        	\arrow["{{[\id,[\eta_P,TY]]}}"{description}, from=8-6, to=9-5]
        	\arrow["{{[[TP,\eta_X],\id]}}"{description}, from=8-6, to=11-6]
        	\arrow["{{\mathsf{funct}}}"{description}, draw=none, from=9-5, to=11-6]
        	\arrow["{{[[TP,\eta_X],\id]}}"{description}, from=9-5, to=12-5]
        	\arrow["{T[T[P,\eta_X],\id]}"{description}, from=10-2, to=12-3]
        	\arrow["{[[\act_P,X],\id]}"{description}, from=10-4, to=14-4]
        	\arrow["{{{{\Gamma_{X,TY}^{TP}}}}}"{description}, from=10-7, to=11-6]
        	\arrow["{{[\id,[\eta_P,TY]]}}"{description}, from=11-6, to=12-5]
        	\arrow[""{name=3, anchor=center, inner sep=0}, "{{{{T[T[\act_P,X],\id]}}}}"{description, pos=0.4}, from=12-1, to=12-3]
        	\arrow["{{{{s_{[TP,X],[TP,Y]}}}}}"{description, pos=0.6}, from=12-1, to=14-1]
        	\arrow["{{{{s_{[P,X],[TP,Y]}}}}}"{description, pos=0.4}, from=12-3, to=14-4]
        	\arrow["{[[\act_P,X],\id]}"', curve={height=12pt}, from=12-5, to=14-5]
        	\arrow["{{{s\,\mathsf{nat}}}}"{description}, draw=none, from=14-1, to=12-3]
        	\arrow["{{{{[[\act_P,X],\id]}}}}"{description}, from=14-1, to=14-4]
        	\arrow["{{{\mathsf{funct}}}}"{description}, draw=none, from=14-4, to=9-5]
        	\arrow["{{{{[\id,s_{P,Y}]}}}}"', from=14-4, to=14-5]
        	\arrow["{{\Gamma\,\mathsf{dinat}}}"', draw=none, from=14-5, to=10-7]
        	\arrow["{[[\eta_P,X],\id]}"', dotted, curve={height=12pt}, from=14-5, to=12-5]
        	\arrow["\def\arraystretch{0.8}\begin{array}{c} \act\,\mathsf{nat} \\ \act_P\,\mathsf{of\,algs} \end{array}"{description}, draw=none, from=1, to=0]
        	\arrow["{\mathsf{funct}}"{description}, draw=none, from=2, to=3]
              \arrow["{T[s_{P,X},\id]}"{description,pos=0.45}, from=3-1, to=12-1, rounded corners, color=red,
                to path={-- ([xshift=-2.5cm]\tikztostart.center) -- ([xshift=-2.5cm]\tikztotarget.center) \tikztonodes -- (\tikztotarget.west)}]
              \arrow["{s_{X,Y}}"{description}, from=1-3, to=10-7, rounded corners, color=red,
                to path={let \p1=([yshift=1cm]\tikztostart.center), \p2=([xshift=1.5cm]\tikztotarget.center) in
                         -- (\p1) -- (\x2, \y1) -- (\p2) \tikztonodes -- (\tikztotarget.east)}]
              \arrow["{\Gamma^P_{X,TY}}", from=10-7, to=14-5, rounded corners, color=black,
                to path={let \p1=([yshift=-1cm]\tikztostart.center), \p2=([xshift=2cm]\tikztotarget.center) in
                         -- (\p1) -- (\x1, \y2) -- (\p2) \tikztonodes -- (\tikztotarget.east)}]
        \end{tikzcd}
    }\]
  \caption{Commutativity of \eqref{eq:sGamma} in \cref{prop:gabi-moand-VS-strong-closed}.}%
  \label{fig:gabi-moand-VS-strong-closed-3}%
\end{amssidewaysfigure}

\clearpage

\begin{amssidewaysfigure}\vspace{2cm}
  \[
    \hspace{-3.2cm}\mathscale{.6}{\begin{tikzcd}[ampersand replacement=\&,column sep=large,row sep=large]
        {T[TX,Y]} \&\&\& {[T^2X,TY]} \&\& {[TX,TY]} \&\& {[X,TY]} \\
        \&\&\& {[T^2X,TY]} \&\& {[[P,T^2X],[P,TY]]} \& {[[P,TX],[P,TY]]} \& {[[P,X],[P,TY]]} \\
        \&\&\& {[[TP,T^2X],[TP,TY]]} \&\& {[[TP,T^2X],[P,TY]]} \& {[[TP,TX],[P,TY]]} \\
        {T[[P,TX],[P,Y]]} \& {[T[P,TX],T[P,Y]]} \&\& {[T[P,TX],[TP,TY]]} \&\& {[T[P,TX],[P,TY]]} \& {[T[P,X],[P,TY]]} \& {[[P,X],[P,TY]]} \\
        {T[[TP,TX],[P,Y]]} \& {[T[TP,TX],T[P,Y]]} \\
        {T[[T^2P,TX],[P,Y]]} \& {[T[T^2P,TX],T[P,Y]]} \\
        {T[T[TP,X],[P,Y]]} \& {[T^2[TP,X],T[P,Y]]} \&\& {[T[TP,X],T[P,Y]]} \&\&\& {[T[TP,X],[P,TY]]} \\
        \&\&\& {[[TP,X],T[P,Y]]} \&\& {[[TP,X],[TP,TY]]} \& {[[TP,X],[P,TY]]} \\
        \&\&\& {[[TP,X],T[TP,Y]]} \&\& {[[TP,X],[T^2P,TY]]} \\
        {T[T[TP,X],[TP,Y]]} \& {[T^2[TP,X],T[TP,Y]]} \& {[T[TP,X],T[TP,Y]]} \& {[[TP,X],T[TP,Y]]} \& {[[P,X],T[TP,Y]]} \& {[[P,X],[T^2P,TY]]} \& {[[P,X],[TP,TY]]} \& {[[P,X],[P,TY]]}
        \arrow["\xi_{TX,Y}", from=1-1, to=1-4]
        \arrow["T\Gamma^P_{TX,Y}"', from=1-1, to=4-1]
        \arrow["{{T\ \mathsf{closed}}}"{description}, draw=none, from=1-1, to=4-4]
        \arrow["{[T\eta_{X}, \id]}", from=1-4, to=1-6]
        \arrow[equals, from=1-4, to=2-4]
        \arrow["{[\eta_X,\id]}", from=1-6, to=1-8]
        \arrow["{{\mathsf{nat}\ \Gamma}}"{description}, draw=none, from=1-8, to=2-4]
        \arrow["\Gamma^P_{X,TY}", from=1-8, to=2-8]
        \arrow[""{name=0, anchor=center, inner sep=0}, "\Gamma^P_{T^2X,TY}", from=2-4, to=2-6]
        \arrow["\Gamma^{TP}_{T^2X,TY}"', from=2-4, to=3-4]
        \arrow["{[[\id,T\eta_{X}],\id]}", from=2-6, to=2-7]
        \arrow["{[[\eta_P,\id],\id]}"', from=2-6, to=3-6]
        \arrow["\equiv"{description}, draw=none, from=2-6, to=3-7]
        \arrow["{[[\id,\eta_X],\id]}", from=2-7, to=2-8]
        \arrow["{[[\eta_P,\id],\id]}", from=2-7, to=3-7]
        \arrow["{{\genfrac{}{}{0pt}{}{\mathsf{closed}}{\mathsf{monad}}}}"{description}, draw=none, from=2-8, to=4-7]
        \arrow[""{name=1, anchor=center, inner sep=0}, "{{{[\id,[\eta_P,\id]]}}}", from=3-4, to=3-6]
        \arrow["{[\xi_{P,TX},\id]}"', from=3-4, to=4-4]
        \arrow["{[[\id,T\eta_{X}],\id]}", from=3-6, to=3-7]
        \arrow["{[\xi_{P,TX},\id]}"', from=3-6, to=4-6]
        \arrow["{{\mathsf{nat}\ \xi}}"{description}, draw=none, from=3-6, to=4-7]
        \arrow["{[\xi_{P,X},\id]}", from=3-7, to=4-7]
        \arrow["\xi_{[P,TX],[P,Y]}", from=4-1, to=4-2]
        \arrow["{T[[\eta_P, \id], \id]}"', from=4-1, to=5-1]
        \arrow["{{\mathsf{nat}\ \xi}}"{description}, draw=none, from=4-1, to=10-2]
        \arrow["{[\id,\xi_{P,Y}]}"', from=4-2, to=4-4]
        \arrow["{{{{{{[T[\eta_P,\id],\id]}}}}}}", from=4-2, to=5-2]
        \arrow["\equiv"{description}, draw=none, from=4-4, to=3-6]
        \arrow["{{{[\id,[\eta_P,\id]]}}}", from=4-4, to=4-6]
        \arrow["\equiv"', draw=none, from=4-4, to=8-7]
        \arrow["{{{[T[\id,\eta_X],\id]}}}", from=4-6, to=4-7]
        \arrow["{{{[\eta_{[P,X]},\id]}}}", from=4-7, to=4-8]
        \arrow[""{name=2, anchor=center, inner sep=0}, "{[T[\eta_P,\id],\id]}"', curve={height=12pt}, from=4-7, to=7-7]
        \arrow[equals, from=4-8, to=2-8]
        \arrow[equals, from=4-8, to=10-8]
        \arrow["{T[[\eta_{TP}, \id],\id]}"', from=5-1, to=6-1]
        \arrow["{{{{{{[T[\eta_{TP},\id],\id]}}}}}}"', from=5-2, to=6-2]
        \arrow["{[T[\id,\eta_X],\id]}"{description}, curve={height=-24pt}, from=5-2, to=7-4]
        \arrow["{{\genfrac{}{}{0pt}{}{\mathsf{closed}}{\mathsf{monad}}}}"', draw=none, from=5-2, to=7-4]
        \arrow["{{{{{{{{T[\xi_{TP,X},\id]}}}}}}}}"', from=6-1, to=7-1]
        \arrow["{{{{{{[T\xi_{TP,X},\id]}}}}}}"', from=6-2, to=7-2]
        \arrow["{{{{{{{{T[\id,[\act_P,\id]]}}}}}}}}"', from=7-1, to=10-1]
        \arrow["{[T\eta_{[TP,X]},\id]}"', from=7-2, to=7-4]
        \arrow["{[\id,T[\act_P,Y]]}"', from=7-2, to=10-2]
        \arrow["{{{{{{[\eta_{[TP,X]},\id]}}}}}}", from=7-4, to=8-4]
        \arrow[""{name=3, anchor=center, inner sep=0}, "{[T[\act_P,\id],\id]}"', curve={height=12pt}, from=7-7, to=4-7]
        \arrow["{[\eta_{[TP,X]},\id]}"', from=7-7, to=8-7]
        \arrow["{[\id,\xi_{P,Y}]}", from=8-4, to=8-6]
        \arrow["{[\id,T[\act_P,Y]]}", from=8-4, to=9-4]
        \arrow["{{\mathsf{nat}\ \xi}}"{description}, draw=none, from=8-4, to=9-6]
        \arrow["{[\id,[\eta_P,\id]]}"', from=8-6, to=8-7]
        \arrow["{[\id,[T\act_P,\id]]}"', from=8-6, to=9-6]
        \arrow["{{\mathsf{nat}\ \eta}}"{description}, draw=none, from=8-6, to=10-7]
        \arrow["{{\mathsf{nat}\ \eta}}"{description}, draw=none, from=8-7, to=4-8]
        \arrow[""{name=4, anchor=center, inner sep=0}, "{[[\act_P,\id],[\act_P,\id]]}"{description}, from=8-7, to=10-7]
        \arrow["{[[\act_P,\id],\id]}"{description}, from=8-7, to=10-8]
        \arrow["{[\id,\xi_{TP,Y}]}"', from=9-4, to=9-6]
        \arrow[equals, from=9-4, to=10-4]
        \arrow["\equiv"{description}, draw=none, from=9-6, to=10-4]
        \arrow["{[[\act_P,X],\id]}"', from=9-6, to=10-6]
        \arrow["{{{{{{{{\xi_{T[TP,X],[TP,Y]}}}}}}}}}"', from=10-1, to=10-2]
        \arrow["\equiv"{description}, draw=none, from=10-2, to=7-4]
        \arrow["{[T\eta_{[TP,X]},\id]}"', from=10-2, to=10-3]
        \arrow["{{{{{{{{[\eta_{[TP,X]},\id]}}}}}}}}"', from=10-3, to=10-4]
        \arrow["{{{{{{{{[[\act_P,X],\id]}}}}}}}}"', from=10-4, to=10-5]
        \arrow["{{{{{{{{[\id,\xi_{TP,Y}]}}}}}}}}"', from=10-5, to=10-6]
        \arrow["{{{{{{{{[\id,[\eta_{TP},\id]]}}}}}}}}"', from=10-6, to=10-7]
        \arrow["{{{{{{{{[\id,[\eta_P,\id]]}}}}}}}}"', from=10-7, to=10-8]
        \arrow["{{\mathsf{dinat}\ \Gamma}}"{description}, draw=none, from=0, to=1]
        \arrow["{{\mathsf{alg}}}"{description}, draw=none, from=2, to=3]
        \arrow["{\mathsf{alg}}"{description}, draw=none, from=10-8, to=4]
        \arrow["{s_{X,Y}}", from=1-1, to=1-8, rounded corners, color=red,
          to path={-- ([yshift=.7cm]\tikztostart.center) -- ([yshift=.7cm]\tikztotarget.center) \tikztonodes -- (\tikztotarget.north)}]
        \arrow["{T[s_{P,X},[\act_P,Y]]}"', from=4-1, to=10-1, rounded corners, color=red,
          to path={-- ([xshift=-2.5cm]\tikztostart.center) -- ([xshift=-2.5cm]\tikztotarget.center) \tikztonodes -- (\tikztotarget.west)}]
        \arrow["{s_{[TP,X],[TP,Y]}}"', from=10-1, to=10-4, rounded corners, color=red,
          to path={-- ([yshift=-.9cm]\tikztostart.center) -- ([yshift=-.9cm,xshift=-.1cm]\tikztotarget.center) \tikztonodes -- ([xshift=-.1cm]\tikztotarget.south)}]
        \arrow["{[[\act_P,X],s_{P,Y}]}"', from=10-4, to=10-8, rounded corners, color=red,
          to path={([xshift=.1cm]\tikztostart.south) -- ([yshift=-.9cm,xshift=.1cm]\tikztostart.center) -- ([yshift=-.9cm]\tikztotarget.center) \tikztonodes -- (\tikztotarget.south)}]
      \end{tikzcd}}
  \]
  \caption{The coherence map \(\Gamma\) lifts to the Eilenberg–Moore category of \(T\) in \cref{prop:normal-closed-is-gabi}.}
  \label{fig:closed-monad-lift-gamma}
\end{amssidewaysfigure}

\clearpage

\begin{amssidewaysfigure}\vspace{1cm}
  \[
    \hspace{-2cm}\mathscale{0.89}{\begin{tikzcd}[ampersand replacement=\&,sep=3em]
	\& {T[TX,\{TY,Z\}]} \&\&\& {T\{TY,[TX,Z]\}} \& \\
	\& {T[TTX,\{TY,Z\}]} \& {T[TX,\{TY,TZ\}]} \& {T\{TY,[TX,TZ]\}} \& {T\{TTY,[TX,Z]\}} \\
	{T[TX,\{TY,Z\}]} \& {[TX,T\{TY,Z\}]} \& {T[TTX,\{TY,TZ\}]} \& {T\{TTY,[TX,TZ]\}} \& {\{TY,T[TX,Z]\}} \& {T\{TY,[TX,Z]\}} \\
	{[X,T\{TY,Z\}]} \& {[TX,T\{TTY,Z\}]} \& {[TX,T\{TY,TZ\}]} \& {\{TY,T[TX,TZ]\}} \& {\{TY,T[TTX,Z]\}} \& {\{Y,T[TX,Z]\}} \\
	\& {[TX,\{TY,TZ\}]} \& {[TX,T\{TTY,TZ\}]} \& {\{TY,T[TTX,TZ]\}} \& {\{TY,[TX,TZ]\}} \\
	\&\& {[TX,\{TY,TTZ\}]} \& {\{TY,[TX,TTZ]\}} \&  \\
	\&\& {[TX,\{TY,TZ\}]} \& {\{TY,[TX,TZ]\}} \\
	{[X,\{Y,TZ\}]} \&\&\&\&\& {\{Y,[X,TZ]\}}
	\arrow[""{name=0, anchor=center, inner sep=0}, "{T\xi_{TX,TY,Z}}"{description}, from=1-2, to=1-5]
	\arrow["{T[\mu_X,\{TY,Z\}]}"{description}, from=1-2, to=2-2]
	\arrow["{T[TX,\{TY,\eta_Z\}]}"{description}, from=1-2, to=2-3]
	\arrow[""{name=1, anchor=center, inner sep=0}, curve={height=24pt}, equals, from=1-2, to=3-1]
	\arrow["{T\{TY,[TX,\eta_Z]\}}"{description}, from=1-5, to=2-4]
	\arrow["{T\{\mu_Y,[TX,Z]\}}"{description}, from=1-5, to=2-5]
	\arrow[""{name=2, anchor=center, inner sep=0}, curve={height=-24pt}, equals, from=1-5, to=3-6]
	\arrow["{T[T\eta_X,\{TY,Z\}]}"{description}, from=2-2, to=3-1]
	\arrow["{s_{TX,\{TY,Z\}}}"{description}, from=2-2, to=3-2]
	\arrow[""{name=3, anchor=center, inner sep=0}, "{T\xi_{TX,TY,TZ}}", from=2-3, to=2-4]
	\arrow["{T[\mu_X,\{TY,TZ\}]}"{description}, from=2-3, to=3-3]
	\arrow["{\scriptscriptstyle\mathsf{nat\ and\ funct}}"{description}, draw=none, from=2-3, to=5-2]
	\arrow["{T\{\mu_Y,[TX,TZ]\}}"{description}, from=2-4, to=3-4]
	\arrow["{\scriptscriptstyle\mathsf{nat\ and\ funct}}"{description}, draw=none, from=2-4, to=5-5]
	\arrow["{t_{TY,[TX,Z]}}"{description}, from=2-5, to=3-5]
	\arrow["{T\{T\eta_Y,[TX,Z]\}}"{description}, from=2-5, to=3-6]
	\arrow["{\scriptscriptstyle\mathsf{nat}}"{description}, draw=none, from=3-1, to=3-2]
	\arrow["{s_{X,\{TY,Z\}}}"{description}, from=3-1, to=4-1]
	\arrow[""{name=4, anchor=center, inner sep=0}, "{[\eta_X,T\{TY,Z\}]}"{description}, from=3-2, to=4-1]
	\arrow["{[TX,T\{\mu_Y,Z\}]}"{description}, from=3-2, to=4-2]
	\arrow["{s_{TX,\{TY,TZ\}}}"{description}, from=3-3, to=4-3]
	\arrow["{t_{TY,[TX,TZ]}}"{description}, from=3-4, to=4-4]
	\arrow["{\scriptscriptstyle\mathsf{nat}}"{description}, draw=none, from=3-5, to=3-6]
	\arrow["{\{TY,T[\mu_X,Z]\}}"{description}, from=3-5, to=4-5]
	\arrow[""{name=5, anchor=center, inner sep=0}, "{\{\eta_Y,T[TX,Z]\}}"{description}, from=3-5, to=4-6]
	\arrow["{t_{Y,[TX,Z]}}"{description}, from=3-6, to=4-6]
	\arrow[""{name=6, anchor=center, inner sep=0}, "{[X,t_{Y,Z}]}"', from=4-1, to=8-1]
	\arrow["{[\eta_X,T\{T\eta_Y,Z\}]}", from=4-2, to=4-1]
	\arrow["{[TX,t_{TY,Z}]}"{description}, from=4-2, to=5-2]
	\arrow["{[TX,T\{\mu_Y,TZ\}]}"{description}, from=4-3, to=5-3]
	\arrow["{\{TY,T[\mu_X,TZ]\}}"{description}, from=4-4, to=5-4]
	\arrow["{\{\eta_Y,T[T\eta_X,Z]\}}"', from=4-5, to=4-6]
	\arrow["{\{TY,s_{TX,Z}\}}"{description}, from=4-5, to=5-5]
	\arrow[""{name=7, anchor=center, inner sep=0}, "{\{Y,s_{X,Z}\}}"{description}, from=4-6, to=8-6]
	\arrow["{[TX,\{TY,T\eta_Z\}]}"{description}, from=5-2, to=6-3]
	\arrow[""{name=8, anchor=center, inner sep=0}, curve={height=30pt}, equals, from=5-2, to=7-3]
	\arrow["{[\eta_X,\{\eta_Y,TZ\}]}"{description}, from=5-2, to=8-1]
	\arrow["{[TX,t_{TY,TZ}]}"{description}, from=5-3, to=6-3]
	\arrow["{\{TY,s_{TX,TZ}\}}"{description}, from=5-4, to=6-4]
	\arrow["{\{TY,[TX,T\eta_Z]\}}"{description}, from=5-5, to=6-4]
	\arrow[""{name=9, anchor=center, inner sep=0}, curve={height=-30pt}, equals, from=5-5, to=7-4]
	\arrow["{\{\eta_Y,[\eta_X,TZ]\}}"{description}, from=5-5, to=8-6]
	\arrow["{[TX,\{TY,\mu_Z\}]}"{description}, from=6-3, to=7-3]
	\arrow["{\{TY,[TX,\mu_Z]\}}"{description}, from=6-4, to=7-4]
	\arrow[""{name=10, anchor=center, inner sep=0}, "{\xi_{TX,TY,TZ}}"', from=7-3, to=7-4]
	\arrow["{[\eta_X,\{\eta_Y,TZ\}]}"{description}, from=7-3, to=8-1]
	\arrow["{\{\eta_Y,[\eta_X,TZ]\}}"{description}, from=7-4, to=8-6]
	\arrow[""{name=11, anchor=center, inner sep=0}, "{\xi_{X,Y,TZ}}"{description}, from=8-1, to=8-6]
	\arrow["{\scriptscriptstyle T\mbox{-}\mathsf{alg}}"{description}, draw=none, from=2-2, to=1]
	\arrow["{\scriptscriptstyle\mathsf{nat}}"{description}, draw=none, from=3, to=0]
	\arrow["{\scriptscriptstyle T\mbox{-}\mathsf{alg}}"{description}, draw=none, from=2-5, to=2]
	\arrow["{\scriptscriptstyle T\mbox{-}\mathsf{alg}}"{description}, draw=none, from=4-2, to=4]
	\arrow["{\scriptscriptstyle T\mbox{-}\mathsf{alg}}"{description}, draw=none, from=4-5, to=5]
	\arrow["{\scriptscriptstyle\mathsf{nat}}"{description}, draw=none, from=5-2, to=6]
	\arrow["{\scriptscriptstyle\mathsf{nat}}"{description}, draw=none, from=5-5, to=7]
	\arrow["{\scriptscriptstyle T\mbox{-}\mathsf{alg}}"{description}, draw=none, from=6-3, to=8]
	\arrow["{\scriptscriptstyle T\mbox{-}\mathsf{alg}}"{description}, draw=none, from=6-4, to=9]
	\arrow["{\scriptscriptstyle\mathsf{nat}}"{description}, draw=none, from=10, to=11]
\end{tikzcd}}\]
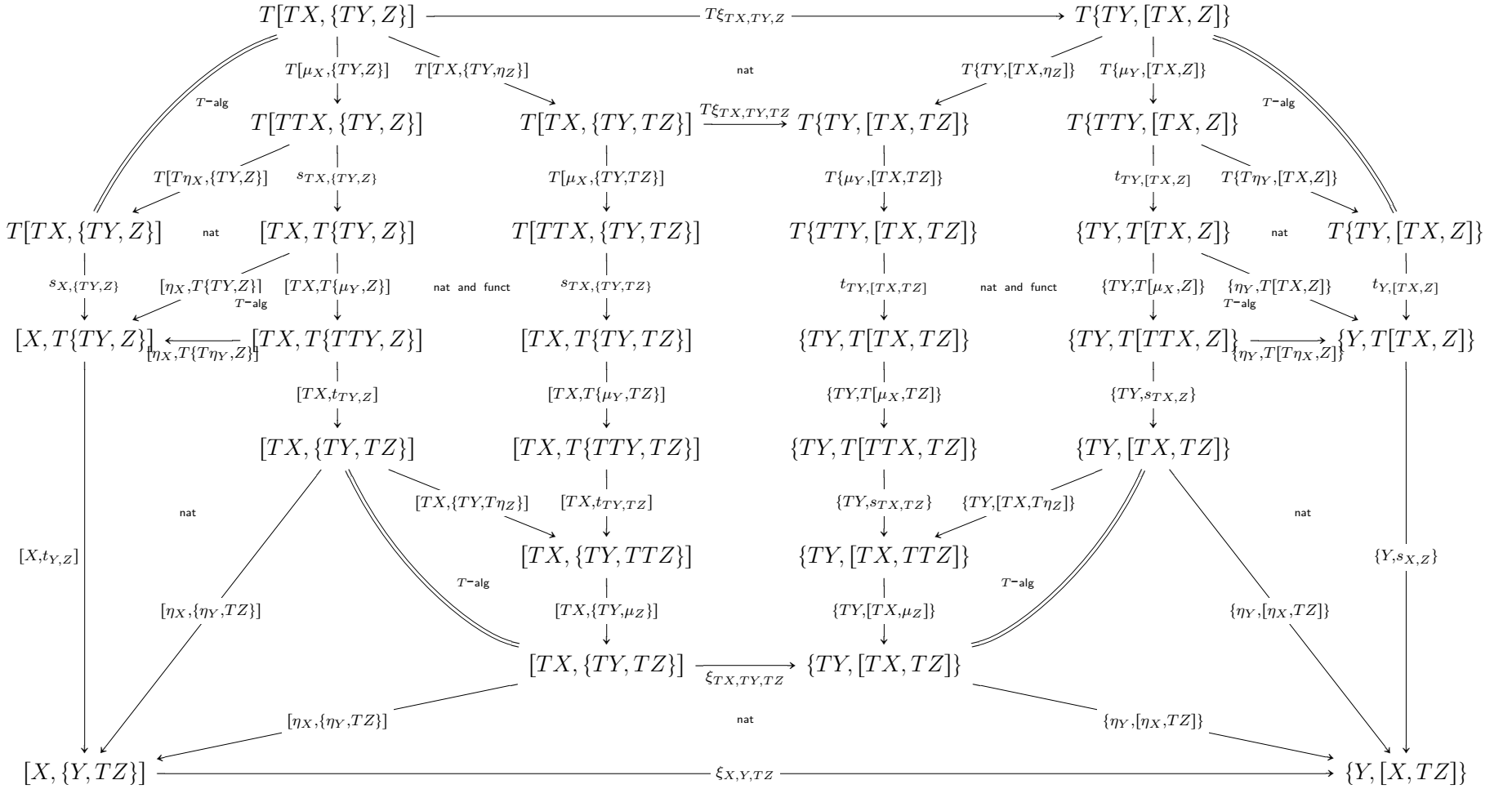
  \caption{If the isomorphism \(\xi\) is a morphism of $T$-algebras in \cref{prop:dgas-recover-nice-biclosed-structures}, then \eqref{eq:dyadic} holds.}
  \label{fig:dyadic-gabimonad-lift-xi}
\end{amssidewaysfigure}

\end{document}